\documentclass[11pt,twoside,a4paper]{article}
\usepackage[a4paper,centering,bindingoffset=0cm,marginpar=2cm,margin=2.5cm]{geometry}

\usepackage[utf8]{inputenc}
\usepackage[T1]{fontenc}
\usepackage{microtype}
\usepackage{amsmath,amssymb}
\numberwithin{equation}{section}
\usepackage{lmodern}
\usepackage[dvipsnames]{xcolor}

\usepackage[hyperref,amsmath,thmmarks]{ntheorem}
\usepackage{aliascnt}

\usepackage[pagestyles]{titlesec}
\usepackage{mathtools}

\usepackage{graphicx,tikz,pgfplots}
\pgfplotsset{compat=newest}
\graphicspath{{images/}}

\usepackage{subcaption}
\usepackage{sidecap}
\mathtoolsset{showonlyrefs}

\usepackage[pdftex,colorlinks=true,linkcolor=blue,citecolor=green,urlcolor=blue,bookmarks=true,bookmarksnumbered=true,pdfusetitle]{hyperref}

\usepackage[skip=3ex minus 1ex]{parskip}

\newtheorem{lemma}{Lemma}[section]

\newaliascnt{proposition}{lemma}
\newtheorem{proposition}[proposition]{Proposition}
\aliascntresetthe{proposition}

\newaliascnt{corollary}{lemma}

\aliascntresetthe{corollary}

\newaliascnt{theorem}{lemma}
\newtheorem{theorem}[theorem]{Theorem}
\aliascntresetthe{theorem}

\theorembodyfont{\normalfont}
\newaliascnt{definition}{lemma}

\aliascntresetthe{definition}

\newaliascnt{assumption}{lemma}

\aliascntresetthe{assumption}

\newaliascnt{notation}{lemma}

\aliascntresetthe{notation}

\newaliascnt{example}{lemma}

\aliascntresetthe{example}

\newaliascnt{experiment}{lemma}

\aliascntresetthe{experiment}

\newaliascnt{remark}{lemma}
\newtheorem{remark}[remark]{Remark}
\aliascntresetthe{remark}

\theoremstyle{nonumberplain}
\theoremseparator{:}
\theoremheaderfont{\normalfont\itshape}

\theoremsymbol{\ensuremath{\square}}

\newtheorem{proof}{Proof}

\usepackage[ruled]{algorithm2e}

\newcommand{\N}{\mathbb{N}}
\newcommand{\R}{\mathbb{R}}

\renewcommand{\S}{\mathbb{S}}
\newcommand{\C}{\mathbb{C}}
\newcommand{\Z}{\mathbb{Z}}
\newcommand{\SO}{\mathrm{SO}(3)}

\newcommand{\abs}[1]{\left|#1\right|}
\newcommand{\norm}[1]{\left\|#1\right\|}

\newcommand{\inner}[2]{\left<#1,#2\right>}
\newcommand{\inn}[1]{\left<#1\right>}
\newcommand{\ktran}{\mathcal{F}}
\newcommand{\ktranTwo}{\mathcal{F}^{(2)}}

\newcommand{\e}{\mathrm e}

\renewcommand{\i}{\mathrm i}
\newcommand{\dd}{\, \mathrm{d}}

\newcommand{\zb}[1]{\boldsymbol{#1}}
\newcommand{\bn}{{\boldsymbol{n}}}

\newcommand{\bd}{{\boldsymbol{d}}}

\newcommand{\bx}{{\boldsymbol x}}
\newcommand{\by}{{\boldsymbol y}}
\newcommand{\bk}{{\boldsymbol k}}

\newcommand{\bv}{{\boldsymbol v}}
\newcommand{\bw}{{\boldsymbol w}}
\newcommand{\be}{{\boldsymbol e}}
\newcommand{\bz}{{\boldsymbol z}}
\newcommand{\bh}{{\boldsymbol h}}

\newcommand{\bphi}{{\boldsymbol \Phi}}
\newcommand{\bgam}{{\boldsymbol \gamma}}
\newcommand{\bsigma}{{\boldsymbol \sigma}}
\newcommand{\bpi}{{\boldsymbol \pi}}

\newcommand{\ui}{u^{\mathrm{inc}}}
\newcommand{\us}{u^{\mathrm{sca}}}
\newcommand{\rs}{{r_{\mathrm{s}}}}
\newcommand{\rM}{r_{\mathrm{M}}}

\begin{document}

\title{Motion Detection in Diffraction Tomography\\ by
 Common Circle Methods}
\author{\texorpdfstring{
Michael Quellmalz\footnotemark[1]
\and
Peter Elbau\footnotemark[2]
\and
Otmar Scherzer\footnotemark[2] \footnotemark[3] \footnotemark[4]
\and 
Gabriele Steidl\footnotemark[1]}{Michael Quellmalz, Peter Elbau, Otmar Scherzer, Gabriele Steidl}
}
\date{\today}

\maketitle

\footnotetext[1]{
TU Berlin,
Stra{\ss}e des 17. Juni 136, 
D-10587 Berlin, Germany,
\{quellmalz, steidl\}@math.tu-berlin.de.
} 
\footnotetext[2]{
University of Vienna,
Oskar-Morgenstern-Platz 1,
A-1090 Vienna, Austria,
\{otmar.scherzer, peter.elbau\}\linebreak @univie.ac.at}
\footnotetext[3]{
Johann Radon Institute for Computational and Applied Mathematics (RICAM),
Altenbergerstraße 69. A-4040 Linz, Austria
}
\footnotetext[4]{
Christian Doppler Laboratory
for Mathematical Modeling and Simulation
of Next Generations of Ultrasound Devices (MaMSi),
Oskar-Morgenstern-Platz 1,
A-1090 Vienna, Austria}

\begin{abstract}
The method of common lines is a well-established reconstruction technique in cryogenic electron microscopy (cryo-EM), 
which can be used to extract the relative orientations of an object given tomographic projection images from different directions.

In this paper, we deal with an analogous problem in optical diffraction tomography.
Based on the Fourier diffraction theorem, we show that rigid motions of the object, i.e.,  rotations and translations, 
can be determined by detecting common circles in the Fourier-transformed data.
We introduce two methods to identify common circles.
The first one is motivated by the common line approach for projection images 
and detects the relative orientation by parameterizing the common circles in the two images. 
The second one assumes a smooth motion over time and calculates the angular velocity 
of the rotational motion via an infinitesimal version of the common circle method.
Interestingly, using the stereographic projection, both methods can be reformulated as common line
methods, but these lines are, in contrast to those used in cryo-EM, 
not confined to pass through the origin and allow for a full reconstruction of the relative orientations.
Numerical proof-of-the-concept examples demonstrate the performance of our reconstruction methods.

\medskip
\noindent
\textit{Keywords.}
Diffraction tomography,
motion detection,
Fourier diffraction theorem,
common circle method,
optical imaging.

\medskip
\noindent
\textit{Math Subject Classifications.}
92C55, 
78A46, 
94A08, 
42B05.  
\end{abstract}

\section{Introduction}

A key task in many imaging modalities consists in recovering an object's inner structure given images of its illuminations from different directions.
The \emph{X-ray computed tomography} (CT) is based on a number of assumptions, most prominently that the light travels along straight lines.
If, however, the object is small compared to the wavelength of the illumination,
the optical diffraction cannot be neglected anymore.
This occurs for example when examining structures with a size of a few micrometers such as biological cells with visible light.
In the so-called \emph{optical diffraction tomography}, we respect the wave character of the light and take optical diffraction into account. 

As biological samples should be imaged preferably in a natural environment,
contact-free manipulation methods are used for rotating the object during the image acquisition process.
Such rotations can be induced by optical \cite{JonMarVol15} 
or acoustical tweezers \cite{dholakia2020comparing,KvaPreRit21,ThaSteMeiHilBerRit11}.

Therefore,
additional effort is necessary if the rigid motion of the object during the image acquisition process is unknown and has to be reconstructed from the captured images.
In this paper, we propose to tackle this problem by a method of common circles and its infinitesimal version
which is inspired by the well-known method of common lines for projection images as applied in cryogenic electron microscopy (cryo-EM) \cite{SinCoiSigCheShk10,Hee87,WanSinWen13}. 
Let us briefly recall this method first.

\paragraph{Method of common lines.} 
In computed tomography, the aim is the reconstruction of an object from given (optical) projection images for different directions of the imaging wave or, 
equivalently, different rotations of the object.
The object's absorption properties are described by a function $f\colon \R^3\to\R$, which has to be recovered.
We assume the object moves in time $t$ according to rotation matrices $R_t$, and the illumination is in direction $\be^3 = (0,0,1)^\top$.
Then, the \emph{ray transform} of $f$ is given by
\begin{equation} \label{eq:ray}
  \mathcal X_{R_t} [f](x_1,x_2)
  \coloneqq
  \int_{-\infty}^{\infty}
  f\left(R_t \, (x_1,x_2,x_3)^\top\right) \dd x_3
  ,\qquad (x_1,x_2)^\top \in\R^2.
\end{equation}
The reconstruction of $f$ is based on the Fourier slice theorem, see e.g.\ \cite[Theorem~2.11]{NatWue01}, which states that 
\begin{equation} \label{fs}
  \ktranTwo[\mathcal X_{R_t} [f]] (k_1,k_2)
  =
  \sqrt{2\pi}\,
  \mathcal F^{(3)}[f](R_t\, (k_1,k_2,0)^\top),
  \qquad (k_1,k_2)^\top \in\R^2,
\end{equation}
where $\mathcal F^{(2)}$ and $\mathcal F^{(3)}$ denote the two- and three-dimensional Fourier transforms, see \eqref{eq:FourierDef}.
Hence, given the data $\mathcal X_{R_t} [f]$ for the rotation $R_t$,
we obtain the Fourier transform of $f$ on the plane 
$P_{R_t}\coloneqq \{R_t\, (k_1,k_2,0)^\top: (k_1,k_2)^\top\in\R^2\}$ through the origin.
If the rotations $R_t$ are known and the planes $P_{R_t}$ fully cover $\R^3$, 
e.g.\ when the object makes a full turn around a fixed rotation axis other than $\be^3$,
then we can reconstruct $f$ by the inverse 3D Fourier transform.
However, in cryo-EM, the rotations $R_t$ of the object are not known. 
The method of common lines makes use of the fact that two planes $P_{R_s}$ and $P_{R_t}$ intersect for $R_s\be^3\ne\pm R_t\be^3$ in a common line which contains the origin.
This common line can be detected from the projection data $\mathcal X_{R_s}$ and $\mathcal X_{R_t}$ by
maximizing the correlation of all possible combinations of lines in the two planes,
which is a minimization problem in two variables.
Note that the common line detection is usually not performed directly in the Fourier space, but by comparing lines of
the 2D Radon transform of $\mathcal X_{R_t}[f]$, see \cite{Hee00} and also \cite{BenBarSin20} for computational methods.
Keeping one plane fixed, the second plane is not uniquely determined just by their common line, see \autoref{fig:common_line} (left side).
We have to compute the pairwise common lines between three planes to  determine the rotation angles between them, 
see \autoref{fig:common_line} (right side).
Alternatively, the reconstruction can be done by moment-based methods \cite{KetLam11}.
Furthermore, Kam's method considers reconstructing $f$ without the need of computing the motion parameters first \cite{Kam80,ShaKilKhoLanSin20}.
\begin{figure}\centering
\includegraphics[height=4.5cm]{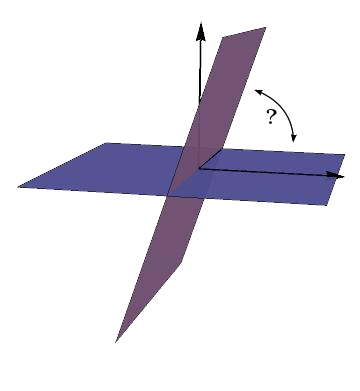}\quad
\includegraphics[height=4.5cm]{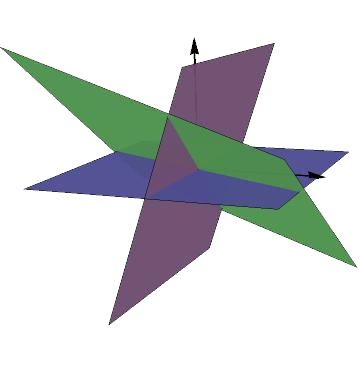}
\caption{
  {\em Left:} A common line pair between two planes $P_{R_s}$ and $P_{R_t}$ does not determine the relative angle between them uniquely.
  {\em Right:}
  Knowing pairwise common lines between three planes, we can determine their orientations uniquely (except for degenerate cases). Courtesy of Denise Schmutz from \cite[Figure 6.2 and Figure 6.3]{Schm17}. 
\label{fig:common_line}}
\end{figure}

\paragraph{Diffraction tomography.}
In optical diffraction tomography, we use a modeling based on Born's or Rytov's approximation of the scattered wave, see e.g.\ \cite[Chapter~6]{KakSla01}.
The Fourier diffraction theorem \cite{Wol69} provides a relation between the measured and Fourier-transformed data
and the Fourier transform of the scattering potential, which we want to reconstruct conceptually similar as for the method of common lines in \eqref{fs}.
Once we know the motion parameters,
the scattering potential can be reconstructed using a backpropagation formula \cite{Dev82,MueSchuGuc15_report}
or inverse discrete Fourier methods \cite{KirQueRitSchSet21}, which can deal with arbitrary, irregular motions.
Under Born's approximation and certain conditions on the moments of the scattering potential, it was shown \cite{KurZic21} that there exists a unique solution to the problem of determining the scattering potential given measurements with unknown object rotations
in an experiment where all possible rotations of the object are performed.

In this paper, we are interested in an experimental setup where the parameters of the rigid motion need to be determined in parallel to the tomographic reconstruction.
We show that the rotations can be determined by an approach which we call the ``method of common circles''.
It is based on computing a common circular arc of intersecting hemispheres. 
In particular, we show that only two hemispheres
are required to compute the rotation,
whereas, in the context the inversion of the ray transform
based on the Fourier slice theorem,
one needs to consider the intersection of three planes.
Furthermore, the object's translation can be completely determined from the measurement data
under some assumptions on the object.
This is in contrast to the ray transform, where the measurements are invariant to every translation of the object in direction of the incident wave. 
The diffraction data is sensitive to the third component of the translations, which allows the full recovery.
The concept of common circles or common arcs was addressed in an empirical way in \cite{HulSzoHaj03}, 
and its application for recovering rotations in the context of crystallography was sketched in \cite{BorTeg11}.
In this paper, we give a rigorous mathematical treatment of the motion reconstruction,
which includes also the determination of translations of the object and an approach based on a time-continuous motion.
For instance, time-continuous motions are appropriate to model tomographic experiments where the object is moved with 
tweezers. In these experiments the motion is continuous but not uniform as in medical CT. 
Such models are in general simpler than Cyro-tomographic experiments, where in a pre-processing steps $X$-ray projection images need to be aligned (numbered) according to their orientations. In this sense tomographic reconstructions based on a time-continuous movement are simpler than standard Cryo-tomographic problems. Moreover, when we assume a time-continuous rigid movement, we can use an infinitesimal calculus for deriving reconstruction methods, leading to the method of \emph{infinitesimal common circle} motion estimation
(see \autoref{sec:cont-cc}).

\paragraph{Outline of this paper.}
In \autoref{sec:ODT}, we describe the model of diffraction tomography with the object undergoing a rigid motion.
Then, in \autoref{sec:comm_circ}, we derive the common circle method for reconstructing the object's \emph{rotations}. 
In \autoref{sec:cont-cc}, we give an infinitesimal version of the common circle method, 
where we assume that the rotations depend smoothly on the time.
\autoref{sec:translations} covers the reconstruction of the \emph{translations} of the object.
In \autoref{sec:recon}, we describe reconstruction methods based on our theoretical findings.
We perform numerical proof-of-concept simulations in \autoref{sec:numerics}
of the proposed methods with two different phantoms and two different motion experiments.
Moreover, based the infinitesimal approach, we can efficiently compute an initialization 
for our optimization algorithm in the direct common circle method.
We postpone technical proofs to 
Appendix \ref{sec:proof} - \ref{sec:proof_2}.
An interesting relation between common circles and common lines
based on the stereographic projection is outlined in Appendix \ref{se:stereo}.

\section{Diffraction Tomography}\label{sec:ODT}
\subsection{Fourier diffraction theorem}

\begin{figure}
	\begin{center}
	\includegraphics[width=0.7\textwidth]{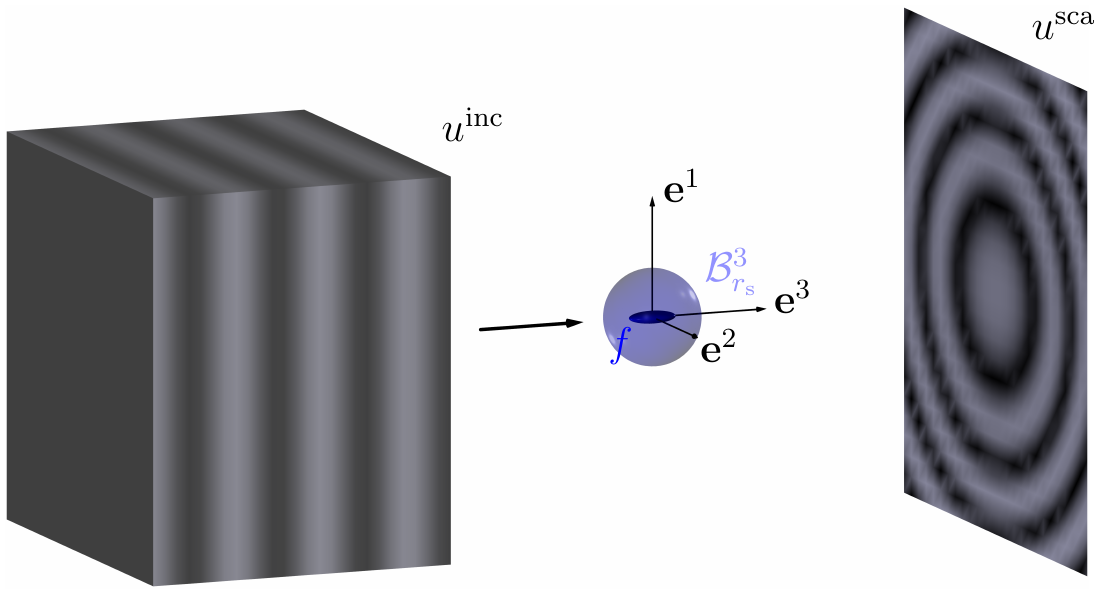}
	\end{center}

\caption{Experimental setup of transmission imaging in optical diffraction tomography. }
	\label{fig:trans}
\end{figure}

Throughout this paper, 
we consider the following experimental setup of \emph{optical diffraction tomography}, 
which is described in detail in \cite{KirQueRitSchSet21,Wol69}.
The unknown object is illuminated by an \emph{incoming plane wave}, 
which propagates in  direction $\be^3 
$ with wave number $k_0 > 0$. 
This is represented by a function
\begin{equation} \label{eq:plane_wave}
\ui(\bx) \coloneqq \e^{\i k_0 x_3}, \quad \bx\in\R^3,
\end{equation}
where we normalized the amplitude to one.
The object shall be contained in an open ball $\mathcal B^3_{\rs}$ of some radius $\rs>0$, where we use the notation
$$\mathcal B^d_{r} \coloneqq \{ \bx\in\R^d : \norm{\bx}< r \}\quad\text{for }d\in\N,\;r>0,$$
with the Euclidean norm $\norm{\cdot}$.
Further, we will need spheres
$\partial\mathcal B_{r}^d(\bz) \coloneqq \{\by\in\R^d: \norm{\by-\bz}=r\}$.
Then the incident wave $\ui$ induces a \emph{scattered wave} $\us$ 
which is recorded in transmission imaging in a plane $\{\bx\in\R^3:x_3=\rM\}$ at a position $\rM > \rs$ outside the object, see \autoref{fig:trans}.
The scattered wave $\us$ can be calculated from the incoming wave $\ui$ and the \emph{scattering potential} $f\colon\R^3\to\R_{\ge0}$, which is a piecewise continuous function with support in $\mathcal B^3_{\rs}$, of the unknown object as a solution of the partial differential equation
\begin{equation} \label{eq:totII}
-(\Delta + k_0^2)\us(\bx) = f(\bx) \left(\us(\bx)+\ui(\bx)\right)
,\qquad \bx\in\R^3,
\end{equation}
which fulfills the Sommerfeld radiation condition
\begin{equation}\label{eq:Sommerfeld}
\lim_{r\to\infty} \max_{\norm\bx=r} \norm{\bx} \abs{\langle\nabla\us(\bx),\tfrac{\bx}{\norm\bx}\rangle-\i k_0\us(\bx)}=0.
\end{equation}
The condition that $f$ is real-valued means that no absorption occurs in the object.
If $\norm f_\infty$ is sufficiently small, the solution $\us$ is small in comparison to $\ui$, so that it can be neglected on the right-hand side of \eqref{eq:totII} and we obtain the \emph{Born approximation} $u$ of the scattered field $\us$, determined by
\begin{subequations}\label{eq:BornSystem}\noeqref{eq:Born}\noeqref{eq:BornSommerfeld}
\begin{align}
-(\Delta + k_0^2) u &= f\ui,\label{eq:Born} \\
\lim_{r\to\infty} \max_{\norm\bx=r} \norm{\bx} \abs{\langle\nabla u(\bx),\tfrac\bx{\norm\bx}\rangle-\i k_0 u(\bx)}&=0.\label{eq:BornSommerfeld}
\end{align}
 \end{subequations}
In the following, we assume that the Born approximation of the scattered wave is valid,
which holds true for small objects which mildly scatter, cf.\ \cite{FauKirQueSchSet21,KakSla01}.

The advantage of the Born approximation is that the solution $u$ 
of the Helmholtz equation \eqref{eq:Born} fulfilling the radiation condition \eqref{eq:BornSommerfeld} can be explicitly written  in the form
\[ u(\bx) = \int_{\R^3}\frac{\e^{\i k_0\norm{\bx-\by}}}{4\pi\norm{\bx-\by}}f(\by)\ui(\by)\dd\by, \]
see, for example, \cite[Theorem 8.1 and 8.2]{ColKre13}.

To calculate from the detected field $u(x_1,x_2,\rM)$, $(x_1,x_2)\in\R^2$, the scattering potential $f$, we use the Fourier diffraction theorem, which relates the two-dimensional Fourier transform of the measurement $(x_1,x_2)\mapsto u(x_1,x_2,\rM)$ to the three-dimensional Fourier transform of the scattering potential~$f$, see for instance \cite[Section~6.3]{KakSla01}, \cite[Theorem~3.1]{NatWue01} or \cite{Wol69}. We use here the version \cite[Theorem~3.1]{KirQueRitSchSet21} derived for the more general case $f\in L^p(\R^3\to\C)$, $p>1$, which states that
\begin{equation} \label{eq:recon}
\ktran_{1,2}[u](\bk, \rM)=\sqrt{\frac{\pi}{2}}\,\frac{\i \e^{\i\kappa(\bk) \rM}}{\kappa(\bk)}\ktran [f]\left(\bh(\bk) \right)\quad\text{for all}\quad\bk = (k_1,k_2)\in\mathcal B_{k_0}^2,
\end{equation}
where $\bh\colon\mathcal B^2_{k_0} \to\R^3$ is defined by 
\begin{equation}\label{eq:h}
\bh(\bk)
\coloneqq
\begin{pmatrix}\bk\\\kappa(\bk)-k_0\end{pmatrix}, \quad \kappa(\bk) \coloneqq \sqrt{k_0^2-\norm{\bk}^2}.
\end{equation}
Here the $d$-dimensional \emph{Fourier transform} 
is defined for $g\colon\R^3\to\C$ 
by
\begin{equation}\label{eq:FourierDef}
\mathcal F[g](\by)
  \coloneqq (2\pi)^{-d/2} \int_{\R^d} g(\bx)\,\e^{-\i \langle\bx,\by\rangle } \dd \bx
  ,\quad\by\in\R^d.
\end{equation}
Moreover, we define the \emph{partial Fourier transform} 
in the first two components as 
\begin{equation}\label{eq:partFourierDef}
\mathcal F_{1,2}[g](k_1,k_2,x_3)\coloneqq (2\pi)^{-1} \int_{\R^2} g(x_1,x_2,x_3)\, \e^{-\i (x_1k_1+x_2k_2)} \dd (x_1,x_2)
,\quad(k_1,k_2,x_3)^\top \in\R^3.
\end{equation}
We skip the dependence of the Fourier transform on the dimension in the notation, 
since this becomes clear from the context.
Geometrically, the Fourier diffraction theorem can be interpreted as follows:
The left-hand side of \eqref{eq:recon} is the Fourier transform of the (two-dimensional) measured images,
while the right-hand side evaluates the three-dimensional Fourier transform of $f$ 
on a hemisphere whose north pole is the origin $\zb0$, 
see the blue hemisphere in \autoref{fig:intersection_spheres}.

\subsection{Motion of the object}
In our setting, we record diffraction images while exposing the object of interest to an unknown rigid motion
$\Psi \colon[0,T]\times\R^3 \to\R^3$,
\begin{equation}\label{eq:motion0}
\Psi_t(\bx)\coloneqq R_t^\top\bx+\bd_t,
\end{equation}
which rotates the object by the rotation matrix $R_t^\top \in\SO\coloneqq\{Q\in\R^{3\times3}: Q^\top Q = I, \, \mathrm{det} \, Q = 1\}$, and translates it by the vector $\bd_t\in\R^3$. Hereby, we consider the object at time $t=0$ as the reference object 
and set correspondingly $\Psi_0$ to be the identity map, that is, $R_0\coloneqq I$ and $\bd_0\coloneqq\zb0$.
The scattering potential of the object that is exposed to this rigid motion $\Psi$ is then described by the function $t\mapsto f\circ\Psi_t^{-1}$, where the inverse function $\Psi_t^{-1}$ is explicitly given by
\begin{equation}\label{eq:motion}
\Psi_t^{-1}(\by)=R_t(\by-\bd_t).
\end{equation}
The diffraction images are now obtained by continuously illuminating the moving object with the incident wave $\ui$ given by \eqref{eq:plane_wave}, and recording the resulting scattered wave (which we will approximate by its Born approximation) on the detector surface $\{\bx\in\R^3:x_3=\rM\}$. We  denote by $u_t$, $t\in[0,T]$, the Born approximation of the wave scattered in the presence of the transformed scattering potential $f \circ \Psi_t^{-1}$, which satisfies the system \eqref{eq:BornSystem} with $f$ replaced by $f \circ \Psi_t^{-1}$, that is, the differential equation
\[ \Delta u_t + k_0^2 u_t = - (f \circ \Psi_t^{-1})\ui \]
together with the radiation condition
\[ \lim_{r\to\infty} \max_{\norm\bx=r} \norm{\bx} \abs{\inner{\nabla u_t(\bx)}{\tfrac x{\norm x}}-\i k_0 u_t(\bx)}=0. \]
Then the recorded measurement data is given by the function $m\colon[0,T]\times\R^2 \to\R$ with
\begin{equation}\label{eq:meas}
 m_t(x_1,x_2)\coloneqq u_t(x_1,x_2,\rM).
\end{equation}
Switching to the Fourier domain with respect to $x_1$ and $x_2$, we find by \eqref{eq:recon} for $\bk\in\mathcal B_{k_0}^2$ that
\begin{equation*}
\begin{split}
	\mathcal F[m_t](\bk) 
	&= \sqrt{\frac{\pi}{2}}\frac{\i \e^{\i\kappa(\bk)\rM}}{\kappa(\bk)}\ktran[f\circ\Psi_t^{-1}]\left(\bh(\bk) \right) \\
	&= \sqrt{\frac{\pi}{2}}\frac{\i \e^{\i\kappa(\bk)\rM}}{\kappa(\bk)}(2\pi)^{-\frac32}\int_{\R^3}f\big(R_t(\by-\bd_t)\big)\e^{-\i\inner{\by}{\bh(\bk)}}\dd\by \\
	&= \sqrt{\frac{\pi}{2}}\frac{\i \e^{\i\kappa(\bk)\rM}}{\kappa(\bk)}(2\pi)^{-\frac32}\int_{\R^3}f(\zb{\tilde y})\,\e^{-\i\inner{R^\top\zb{\tilde y}+\bd_t}{\bh(\bk)}}\dd\zb{\tilde y} ,
\end{split}
\end{equation*}
and hence the explicit relation
\begin{equation}\label{eq:recon_n}
    \mathcal F[m_t](\bk)
    = \sqrt{\frac{\pi}{2}}\frac{\i \e^{\i\kappa(\bk)\rM}}{\kappa(\bk)}\ktran [f]\left(R_t\bh(\bk) \right)\e^{-\i\inner{\bd_t}{\bh(\bk)}}
\end{equation}
between the measured data $m_t$ and the unknown scattering potential $f$.

However, this depends on the unknown parameters $R_t$ and $\bd_t$ describing the motion of the object.
The aim of this paper is to recover both unknown motion parameters.
We will first reconstruct the rotation matrix $R_t$ 
from the absolute values $\abs{\mathcal F[m_t]}$ by  two different approaches, namely
i) 
the method of common circles, which is in the spirit of the common lines method in ray transforms, 
and
ii)
the infinitesimal method for finding changes in the angular velocity during the motion which assumes
smooth rotations $R_t$ in time.
Relying just on absolute values $\abs{\mathcal F[m_t]}$ removes the dependency on the translations $\bd_t$, 
which only enter into the Fourier transform as a phase factor. 
Therefore, we will use the full data $\mathcal F[m_t]$ to reconstruct the translation vectors~$\bd_t$ in the second step.
Knowing $R_t$ and $\bd_t$, relation \eqref{eq:recon_n} 
can be used to  reconstruct the scattering potential $f$ as described in \cite{KirQueRitSchSet21}.

At this point, we want to stress that the possible reduction of the data to the absolute values $\abs{\mathcal F[m_t]}$ for the reconstruction of the rotations is not directly connected to the phaseless optical diffraction measurements, where only the absolute values $\abs{\us(x_1,x_2,\rM)}$, $x_1,x_2\in\R$, of the scattered wave $\us$ are detected, see \eqref{eq:recon_n}.
\definecolor{tubR}{RGB}{197,14,31}
\definecolor{vieB}{RGB}{0,105,170}
\begin{figure}[!ht]
\begin{center}
  \centering\includegraphics[height=15em]{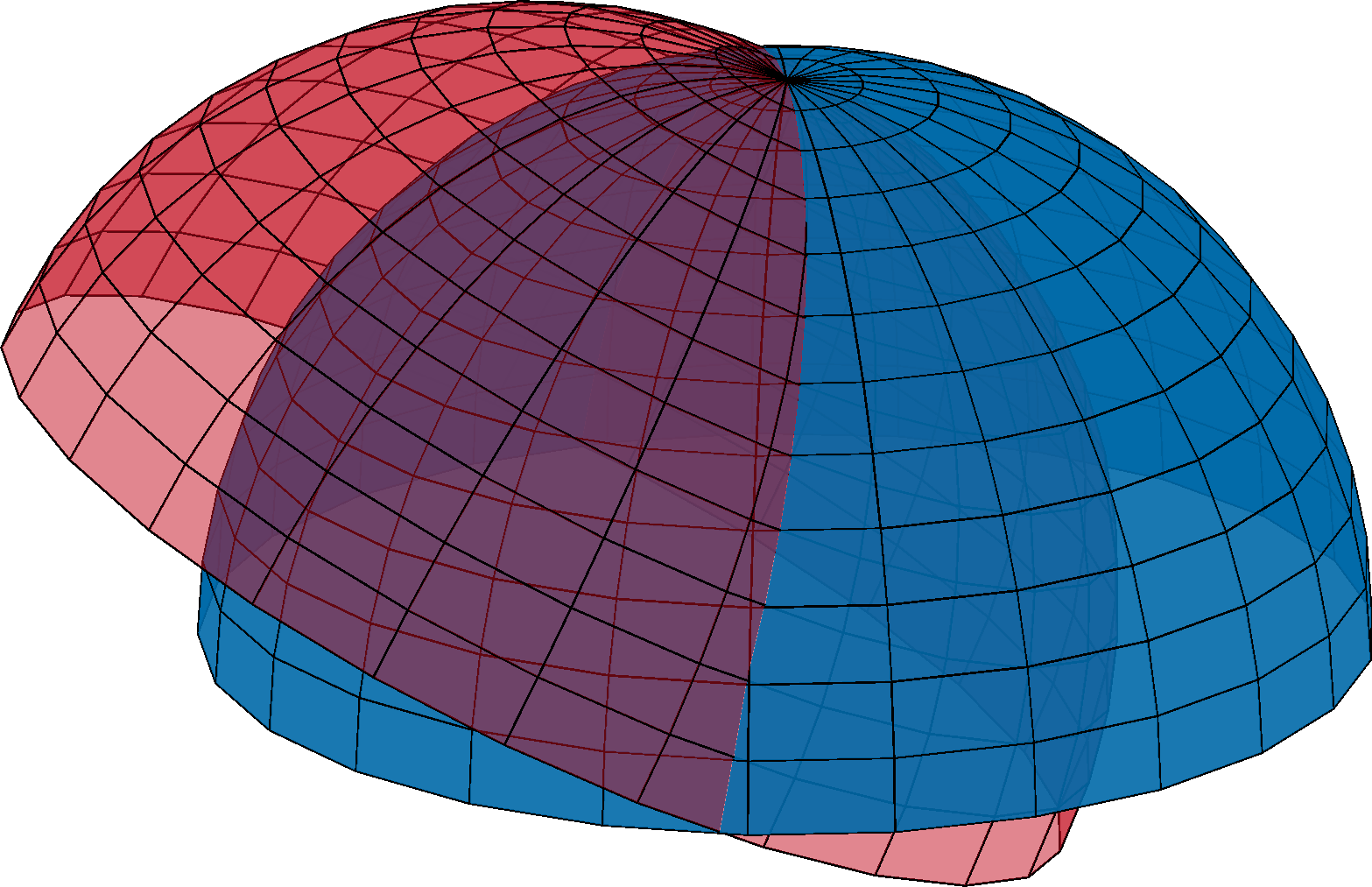}
  \caption{Illustration of the common circles. 
    Two hemispheres $\mathcal H_0$ and $\mathcal H_t$ intersect in a common circle arc. 
  	The north pole of the hemispheres is at $\zb0$.  \label{fig:intersection_spheres}
}
\end{center}
\end{figure}
%
\section{Common Circle Method} \label{sec:comm_circ}
%
Given measurements $m_t$, $t \in [0,T]$ from \eqref{eq:meas}, we can compute their \emph{scaled squared energy} 
$\nu_t \colon\mathcal B^2_{k_0} \to[0,\infty)$ by 
\begin{equation} \label{eq:nu}
 \nu_t(\bk) \coloneqq \frac2\pi \kappa^2(\bk)\abs{\mathcal F[m_t](\bk)}^2
\end{equation}
with $\kappa$ from \eqref{eq:h}.
According to \eqref{eq:recon_n}, this can be expressed in terms of the scattering potential $f$ as
\begin{equation}\label{eq:nu-F}
\nu_t(\bk)
= \abs{\ktran[f] (R_t\bh(\bk))}^2
,\qquad \bk \in \mathcal B^2_{k_0}.
\end{equation}
Thus we observe for $s \not = t$ that 
\begin{equation} \label{match}
\nu_s(\bk_{s,t})=\nu_t(\bk_{t,s})
\end{equation}
holds for all pairs $(\bk_{s,t},\bk_{t,s})\in\mathcal B^2_{k_0}\times\mathcal B^2_{k_0}$  
fulfilling 
\begin{equation}\label{eq:comm_circ_cond}
R_s\bh(\bk_{s,t}) =  R_t\bh(\bk_{t,s}).
\end{equation}
The aim of this section is to parameterize 
the curve consisting of all points $\bsigma_{s,t} = R_s(\bh(\bk_{s,t})) =  R_t(\bh(\bk_{t,s}))$
and use this afterwards for describing the associated curves $\bk_{s,t}$ in their respective planes.
We will see that the first one is a circular arc in the intersection of two hemispheres, while the second one is an elliptic arc.
Having a parameterization with respect to the planes, where the curves are supported,
we can switch to their description via the Euler angles of the rotation $R_s^\top R_t$.
In Section \ref{sec:recon},
we will use this description to determine the Euler angles by minimizing a functional 
based on the matching  condition 
$\nu_s(\bk_{s,t})=\nu_t(\bk_{t,s})$. 

We start by defining the sets
\[ \mathcal H_t \coloneqq\left\{ R_t\bh(\bk):\bk \in \mathcal B^2_{k_0} \right\},\quad t\in[0,T], \]
which are by \eqref{eq:h} the hemispheres with radius~$k_0$ and center $-k_0R_t\be^3$, i.e.,
\begin{align} 
\mathcal H_t &= \{R_t\by:\by\in\partial\mathcal B_{k_0}^3(-k_0\be^3),\,y_3>-k_0\} \\
&= \{\bx\in\partial\mathcal B_{k_0}^3(-k_0R_t\be^3):\langle\bx,R_t\be^3\rangle>-k_0\}, \label{eq:Ht-sphere} 
\end{align}
see \autoref{fig:intersection_spheres}.
The intersection $\mathcal H_s\cap\mathcal H_t$, $s\ne t$ is an arc of a circle
and the reason why we call this approach ,,method of common circles''.
The following lemma gives its parameterization.

\begin{lemma}[Parameterization of the common circular arcs]\label{th:YsYt_param}
Let $s,t\in[0,T]$ such that $R_s\be^3\neq \pm R_t\be^3$. 
Then it holds $\mathcal H_s\cap\mathcal H_t = \{\bsigma_{s,t}(\beta):\beta\in J_{s,t}\}$ 
with the curve $\bsigma_{s,t}\colon J_{s,t}\to\R^3$ defined by
\begin{equation}\label{eq:YsYt_param}
\bsigma_{s,t}(\beta)\coloneqq a_{s,t}(\cos(\beta)-1)\bv_{s,t}^1+ a_{s,t}\sin(\beta)\bv_{s,t}^2,
\end{equation}
where we used the positively oriented, orthonormal basis
\begin{equation}\label{eq:YsYt_param:basis}
\bv_{s,t}^1\coloneqq\frac{R_s\be^3+R_t\be^3}{\norm{R_s\be^3+R_t\be^3}},\quad
\bv_{s,t}^2\coloneqq\frac{R_s\be^3\times R_t\be^3}{\norm{R_s\be^3\times R_t\be^3}},\quad
\bv_{s,t}^3\coloneqq\frac{R_s\be^3-R_t\be^3}{\norm{R_s\be^3-R_t\be^3}},
\end{equation}
the radius 
$$ a_{s,t}\coloneqq\frac{k_0}2\norm{R_s\be^3+R_t\be^3},$$ 
and the interval
\begin{equation}\label{eq:YsYt_param:angle}
J_{s,t}\coloneqq\begin{cases}(-\pi,\pi]&\text{if }\langle R_s\be^3,R_t\be^3\rangle\le0,\\(-\beta_{s,t},\beta_{s,t})&\text{if }\langle R_s\be^3,R_t\be^3\rangle>0\end{cases}\quad\text{with}\quad\beta_{s,t}\coloneqq\arccos\left(\frac{\langle R_s\be^3,R_t\be^3\rangle-1}{\langle R_s\be^3,R_t\be^3\rangle+1}\right).
\end{equation}
In particular, we have $\bsigma_{s,t}(\beta) = \bsigma_{t,s}(-\beta)$ for all $\beta \in J_{s,t}$.
\end{lemma}
Next, according to \eqref{eq:comm_circ_cond}, we intend to find the parameterization of $\bgam_{s,t}$ such that
$\bsigma_{s,t} (\beta) =  R_s\bh(\bgam_{s,t}(\beta))$, $\beta\in J_{s,t}$. Indeed, we see in the following lemma
that $\bgam_{s,t}$ is an elliptic arc. The relation between both parameterizations is illustrated in \autoref{fig:intersection_basis},
where $P(x_1,x_2,x_3)^\top\coloneqq(x_1,x_2)^\top$.
Note that the first case in \eqref{eq:YsYt_param:angle} corresponds with when the ``lense'' in the middle of \autoref{fig:intersection_basis} is fully closed.
\begin{lemma}[Parameterization by elliptic arcs]\label{th:gamma}
Let $s,t\in[0,T]$ such that $R_s\be^3\neq \pm R_t\be^3$ and let $\bsigma_{s,t}$ be defined as in \eqref{eq:YsYt_param}. Then, we have
\begin{equation}\label{eq:sigma_gamma}
\bsigma_{s,t}(\beta) = R_s\bh(\bgam_{s,t}(\beta))\quad\text{for all}\quad\beta\in J_{s,t}
\end{equation}
with the elliptic arc $\bgam_{s,t}\colon J_{s,t}\to\mathcal B^2_{k_0}$ determined by
\begin{equation}\label{eq:gamma}
\bgam_{s,t}(\beta)\coloneqq\tilde a_{s,t}(\cos(\beta)-1)\bw_{s,t}^1+ a_{s,t}\sin(\beta)\bw_{s,t}^2,
\end{equation}
where the directions of the axes are given by
\begin{equation}\label{eq:gamma:basis}
\bw_{s,t}^1\coloneqq\frac{P(R_s^\top R_t\be^3)}{\norm{P(R_s^\top R_t\be^3)}}
\quad\text{and}\quad
\bw_{s,t}^2\coloneqq \frac{P(\be^3\times R_s^\top R_t\be^3)}{\norm{P(\be^3\times R_s^\top R_t\be^3)}},
\end{equation}
and $\tilde a_{s,t}\coloneqq\frac{k_0}2\norm{\smash{P(R_s^\top R_t\be^3)}}$.
In particular, it holds
\begin{equation} \label{eq:commonCircleIdentity}
R_s \bh(\bgam_{s,t}(\beta)) = R_t \bh(\bgam_{t,s}(-\beta)) \quad \text{for all } \beta \in J_{s,t}.
\end{equation}%
\end{lemma}
\begin{figure}[!ht]
  \centering\includegraphics[width=.6\textwidth]{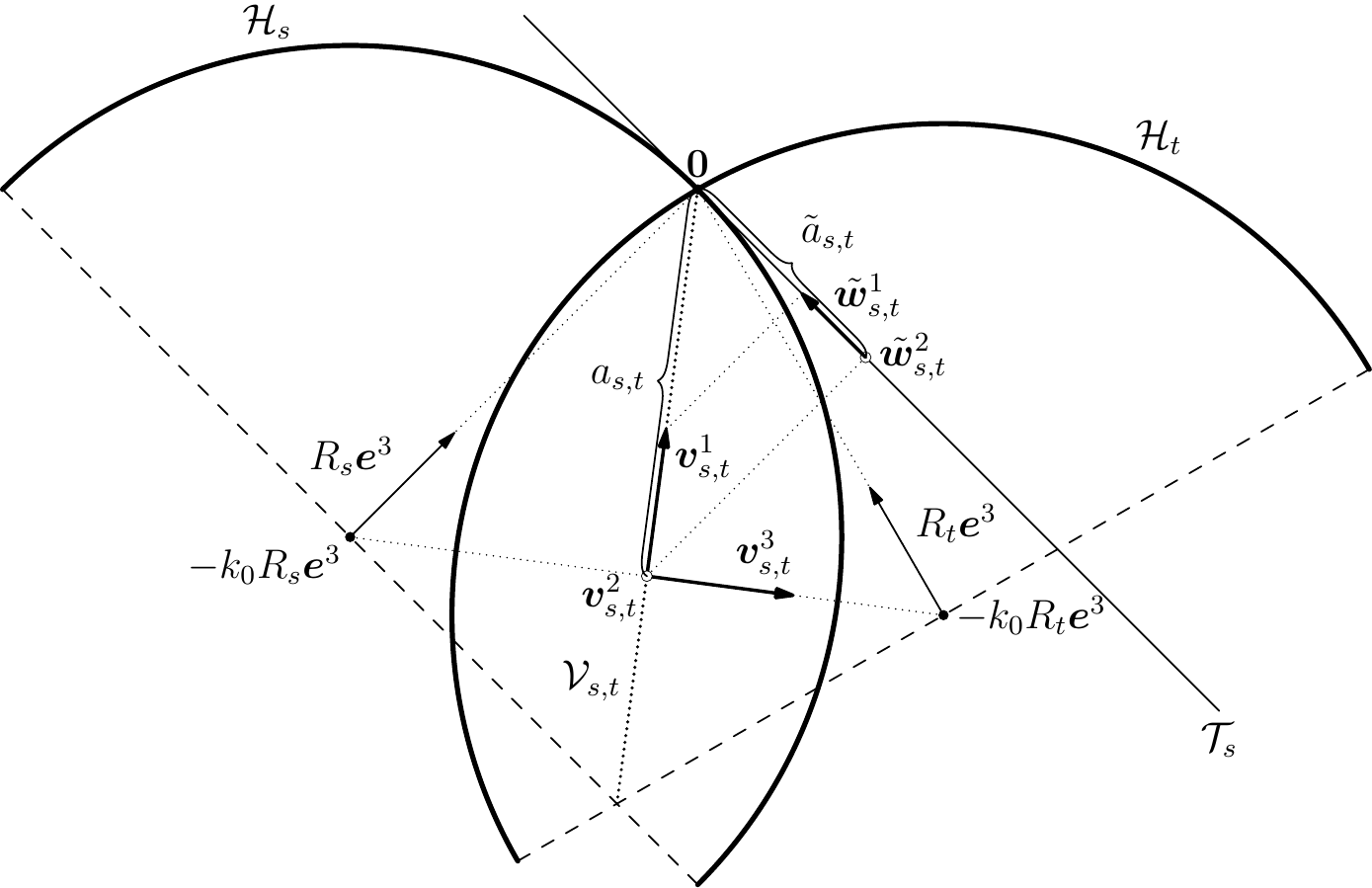}
  \caption{Intersection of two hemispheres $\mathcal H_s$ and $\mathcal H_t$ in the plane through the centers of the hemispheres and the origin. The circular arc $\mathcal H_s\cap\mathcal H_t$ lies in the plane $\mathcal V_{s,t}=\{\bx\in\R^3: \inn{\bx,(R_s-R_t)\be^3}=0\}$ perpendicular to the line between the centers. 
It is spanned by $\bv_{s,t}^j$, $j=1,2$ in \autoref{th:YsYt_param}.
The basis $\bw_{s,t}^j$, $j=1,2$ of $\R^2$ in \autoref{th:gamma} is illustrated by the orthogonal projection 
$\tilde{\bw}_{s,t}^j$ of $\bv_{s,t}^j$ to the tangent plane $\mathcal T_s$ of $\mathcal H_s$ at $\zb0$. 
They are explicitly related by $\bw_{s,t}^j=P(R_s^\top\tilde{\bw}_{s,t}^j)$.
  }\label{fig:intersection_basis}
\end{figure}

Finally, we want to express $\bgam_{s,t}$ in terms of the Euler angles of rotation matrix $R_s^\top R_t=(R_t^\top R_s)^\top$.
Recall that every rotation matrix in $\mathrm{SO}(3)$ can be written (in the $z$-$y$-$z$ convention) in the form
\[ Q^{(3)}(\varphi)\,Q^{(2)}(\theta)\,Q^{(3)}(\psi) \]
with the Euler angles $\varphi,\psi\in\R/(2\pi\Z)$ and $\theta\in[0,\pi]$, where $Q^{(2)}$ and $Q^{(3)}$ denote the rotation matrices
\[ 
Q^{(2)}(\alpha) \coloneqq \begin{pmatrix}\cos(\alpha)&0&\sin(\alpha)\\0&1&0\\-\sin(\alpha)&0&\cos(\alpha)\end{pmatrix}
\quad\text{and}\quad 
Q^{(3)}(\alpha) \coloneqq \begin{pmatrix}\cos(\alpha)&-\sin(\alpha)&0\\\sin(\alpha)&\cos(\alpha)&0\\0&0&1\end{pmatrix},
\]
around the $\be^2$ and $\be^3$ axis, respectively. The Euler angles are uniquely defined if we set $\psi=0$ for $\theta\in\{0,\pi\}$.

\begin{proposition}[Representation of $\bgam_{s,t}$ via the Euler angles of $R_s^\top R_t$]\label{th:Euler}
Let $s,t\in[0,T]$ such that $R_s\be^3\neq \pm R_t\be^3$ and let $(\varphi,\theta,\psi)\in(\R/(2\pi\Z))\times[0,\pi]\times(\R/(2\pi\Z))$ be the Euler angles of the rotation $R_s^\top R_t$, i.e,
\begin{equation} \label{eq:euler} 
R_s^\top R_t = Q^{(3)}(\varphi)\,Q^{(2)}(\theta)\,Q^{(3)}(\psi). 
\end{equation}
Then the elliptic arc $\bgam_{s,t}$ from \eqref{eq:gamma} is given in terms of the Euler angles by
$\bgam_{s,t}(\beta) = \bgam^{\varphi,\theta}(\beta)$, where
\begin{equation}\label{eq:gammaEuler}
\bgam^{\varphi,\theta}(\beta) \coloneqq \frac{k_0}2\sin(\theta)(\cos(\beta)-1)\begin{pmatrix}\cos(\varphi)\\
\sin(\varphi)\end{pmatrix}+k_0\cos(\tfrac\theta2)\sin(\beta)\begin{pmatrix}-\sin(\varphi)\\\cos(\varphi)\end{pmatrix},
\quad\beta\in J_{s,t}.
\end{equation}
\end{proposition}

\begin{figure}[!ht]
  \centering\includegraphics[width=0.99\textwidth]{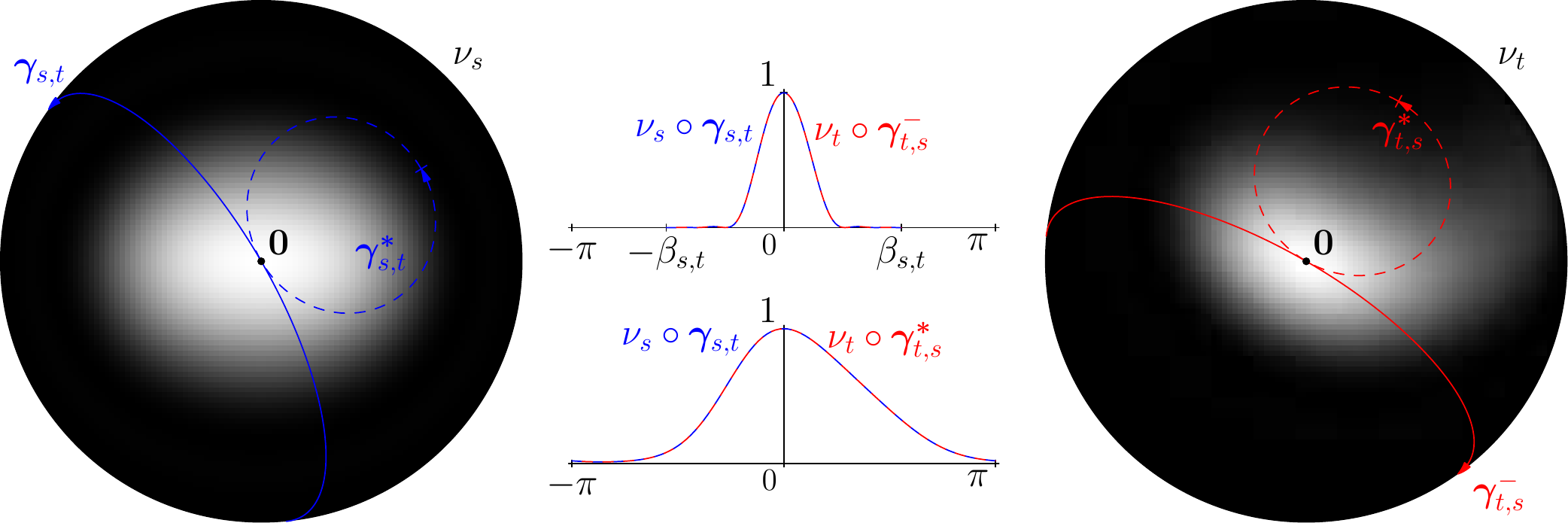}
  \caption{
    Scaled squared energies $\nu_s$ (left) and $\nu_t$ (right), see \eqref{eq:nu}, for a characteristic function~$f$ of an ellipsoid in $\R^3$. 
    For the relative rotation 
    $R_s^\top R_t = Q^{(3)}(\frac\pi6)Q^{(2)}(\frac\pi4)Q^{(3)}(\frac{2\pi}3)$, we show the paths of the corresponding two elliptic arcs $\bgam_{s,t}$ and $\bgam_{t,s}$ (solid lines), where $\bgam_{t,s}^-$ denotes the reversed elliptic arc 
		$\bgam_{t,s}^-(\beta)\coloneqq\bgam_{t,s}(-\beta)$, and their dual arcs $\zb{\bgam}^*_{s,t}$ and $\zb{\bgam}^*_{t,s}$ (dashed), cf.\ \autoref{th:dualCircle}. The relations \eqref{eq:commonCircleIdentity} and \eqref{eq:commonCircleIdentityDual} are verified in the center:
    The top plot shows that the graphs of the scaled squared energy $\nu_s\circ\bgam_{s,t}(\beta)$ along the elliptic arcs in blue and $\nu_t\circ \bgam_{t,s}^-(\beta)$ in red coincide for all $\beta$.
    The same can be seen in the bottom for the dual arcs.
  }\label{fig:ellipticArcs}
\end{figure}

Since the scattering potential $f$ is real-valued, its Fourier transform fulfills the symmetry property
\begin{equation}\label{eq:symmetryNu}
  \ktran[f](-\by) = \overline{\ktran[f](\by)}\quad\text{for all}\quad\by\in\R^3,
\end{equation}
which is also known as Friedel's law.
Thus we see analogously to \eqref{match} that $\nu_s(\zb{k}^*_{s,t}) = \nu_t(\zb{k}^*_{t,s})$ for all pairs $(\bk^*_{s,t},\bk^*_{t,s})\in\mathcal B^2_{k_0}\times\mathcal B^2_{k_0}$
satisfying the \emph{``dual'' condition} to \eqref{eq:comm_circ_cond}, i.e.,
\begin{equation} \label{eq:dualK} 
  R_s\bh(\zb{k}^*_{s,t}) = -R_t\bh(\zb{k}^*_{t,s})
\end{equation}
The parameterization 
can be handled in a similar way. 
To this end, we define the reflected hemisphere $-\mathcal H_t=\{-\bx:\bx\in\mathcal H_t\}$ and
summarize the results in the following proposition. A graphical illustration is given in \autoref{fig:ellipticArcs}.

\begin{proposition}[Parameterization in dual case]\label{th:dualCircle}
Let $s,t\in[0,T]$ such that $R_s\be^3\ne R_t\be^3$.
\\
(i)
It holds $\mathcal H_s\cap(-\mathcal H_t)=\{\zb{\sigma}^*_{s,t}(\beta):\beta\in J^*_{s,t}\}$ with the curve
$\zb{\sigma}^*_{s,t}\colon J^*_{s,t}\to\R^3$ defined by
\begin{equation}\label{eq:dualCircle}
\zb{\sigma}^*_{s,t}(\beta)\coloneqq a^*_{s,t}(\cos(\beta)-1)\bv_{s,t}^3- a^*_{s,t}\sin(\beta)\bv_{s,t}^2,
\end{equation}
where 
$$a^*_{s,t}\coloneqq\frac{k_0}2\norm{R_s\be^3-R_t\be^3}$$
and with the interval
\begin{equation}\label{eq:dualCircleAngle}
J^*_{s,t}\coloneqq\begin{cases}(-\pi,\pi]
&\text{if }\langle R_s\be^3,R_t\be^3\rangle\ge0,\\(-\beta^*_{s,t},\beta^*_{s,t})
&\text{if }\langle R_s\be^3,R_t\be^3\rangle<0\end{cases}
\quad\text{with}\quad
\beta^*_{s,t}\coloneqq\arccos\left(\frac{\langle R_s\be^3,R_t\be^3\rangle+1}{\langle R_s\be^3,R_t\be^3\rangle-1}\right).
\end{equation}
(ii) Further, we have
\begin{equation}\label{eq:checkSigma_checkGamma}
\zb{\sigma}^*_{s,t}(\beta) = R_s\bh(\zb{\gamma}^*_{s,t}(\beta))\quad\text{for all}\quad\beta\in J^*_{s,t}
\end{equation}
with the  elliptic arc $\bgam^*_{s,t}\colon J^*_{s,t}\to\mathcal B^2_{k_0}$ determined by
\begin{equation}\label{eq:dualEllipse}
\zb{\gamma}^*_{s,t}(\beta) \coloneqq -\tilde a_{s,t}(\cos(\beta)-1)\bw_{s,t}^1- a^*_{s,t}\sin(\beta)\bw_{s,t}^2.
\end{equation}
(iii)
Let $(\varphi,\theta,\psi)\in(\R/(2\pi\Z))\times[0,\pi]\times(\R/(2\pi\Z))$ denote the Euler angles 
of the rotation $R_s^\top R_t$, see \eqref{eq:euler}.
Then the elliptic arc $\zb{\gamma}^*_{s,t}$ has the form
$\zb{\gamma}^*_{s,t}(\beta) = \zb{\gamma}^{*,\varphi,\theta}(\beta) $, where
\begin{equation}\label{eq:dualEllipseEuler}
\zb{\gamma}^{*,\varphi,\theta}(\beta)\coloneqq -\frac{k_0}2\sin(\theta)(\cos(\beta)-1)\begin{pmatrix}\cos(\varphi)\\\sin(\varphi)\end{pmatrix}-k_0\sin(\tfrac\theta2)\sin(\beta)\begin{pmatrix}-\sin(\varphi)\\\cos(\varphi)\end{pmatrix}.
\end{equation}
In particular, it holds
\begin{equation} \label{eq:commonCircleIdentityDual}
R_s\bh(\zb{\gamma}^*_{s,t}(\beta)) = -R_t\bh(\zb{\gamma}^*_{t,s}(\beta)) \quad \text{for all } \beta\in J^*_{s,t}.
\end{equation}
\end{proposition}

So far, we have excluded the cases $R_s\be^3=\pm R_t\be^3$. 
The case $R_s\be^3=\pm R_t\be^3$ corresponds with a rotation in $x_1$ $x_2$ plane,
while the other contains an additional rotation of 180\,\textdegree.
The following proposition shows that these constellations 
can easily be detected. 
To this end, we define the matrices
\begin{equation} \label{eq:RS}
\mathrm{S}\coloneqq\begin{pmatrix}1&0\\0&-1\end{pmatrix}
\qquad\text{and}\qquad
  \mathrm{Q}(\alpha) \coloneqq\begin{pmatrix}\cos(\alpha)&-\sin(\alpha)\\\sin(\alpha)&\cos(\alpha)\end{pmatrix}, \quad \alpha\in\R.
 \end{equation}

\begin{proposition}[Special cases $R_s\be^3=\pm R_t\be^3$]\label{th:commonAxis}
(i)
Let $s,t\in[0,T]$ such that $R_s\be^3=R_t\be^3$. 
Then we have $R_s^\top R_t=Q^{(3)}(\alpha)$ for some $\alpha\in\R/(2\pi\Z)$ and
\begin{equation}\label{eq:commonAxis}
\nu_s(\mathrm{Q}(\alpha)\bk) = \nu_t(\bk)
\quad\text{for all}\quad\bk\in\mathcal B_{k_0}^2.
\end{equation}
(ii)
Let $s,t\in[0,T]$ such that $R_s\be^3=-R_t\be^3$. 
Then we have $R_s^\top R_t=Q^{(2)}(\pi)Q^{(3)}(\alpha)$ 
for some $\alpha\in\R/(2\pi\Z)$ and
\begin{equation}\label{eq:commonAxisReflection}
\nu_s(\mathrm{S} \mathrm{Q}(\alpha)\bk) = \nu_t(\bk)
\quad \text{for all} \quad\bk\in\mathcal B_{k_0}^2.
\end{equation}
\end{proposition}

Finally,  we can use our findings to formulate our main theorem 
which says that under certain conditions
the Euler angles of $R_s^\top R_t$
can be determined from the matching condition \eqref{match}.

\begin{theorem}[Reconstruction of Euler angles] \label{th:commonCircle}
Let $s,t\in[0,T]$ such that neither \eqref{eq:commonAxis} nor \eqref{eq:commonAxisReflection} hold for any $\alpha\in\R/(2\pi\Z)$.
Furthermore, assume that there exist unique angles $\varphi,\psi \in \R/(2\pi\Z)$ 
and $\theta\in [0,\pi]$ 
such that
\begin{alignat}{5}
  \nu_s(\bgam^{\varphi,\theta}(\beta)) 
  &= 
  \nu_t(\bgam^{\pi-\psi,\theta}(-\beta))
  \quad&&\text{for all}\quad \beta\in [-\tfrac\pi2,\tfrac\pi2]\quad\text{and}
  \label{eq:nuGamma}
  \\
  \nu_s(\bgam^{*,\varphi,\theta}(\beta)) 
  &= 
  \nu_t(\bgam^{*,\pi-\psi,\theta}(\beta))
  \quad&&\text{for all}\quad \beta\in [-\tfrac\pi2,\tfrac\pi2],
  \label{eq:nuGammaDual}
\end{alignat}
where $\bgam^{\varphi,\theta}$ and $\bgam^{*,\varphi,\theta}$ are defined in \eqref{eq:gammaEuler} and \eqref{eq:dualEllipseEuler}, respectively.
Then we have
\begin{equation}\label{eq:commonCircle}
  R_s^\top R_t = Q^{(3)}(\varphi) Q^{(2)}(\theta) Q^{(3)}(\psi).
\end{equation}
\end{theorem}

Since the curves $\zb{\gamma}^*$ in \eqref{eq:nuGammaDual} are traversed in the same direction 
and the curves $\bgam$ in \eqref{eq:nuGamma} in the opposite direction,
they can be distinguished.
Hence, under the uniqueness assumptions in \autoref{th:commonCircle}, it suffices
if only one of the equations \eqref{eq:nuGamma} or \eqref{eq:nuGammaDual} 
is fulfilled in order to obtain the Euler angles of $R_s^\top R_t$.
One could also obtain a slightly stronger result by replacing the interval $[-\tfrac\pi2,\tfrac\pi2]$ in \autoref{th:commonCircle} by a larger one depending on $\theta$, cf.\ \eqref{eq:YsYt_param:angle} and \eqref{eq:dualCircleAngle}.

Nevertheless, the reconstruction relies on the uniqueness of the elliptic arcs 
with property \eqref{eq:nuGamma}, which might fail if the function $f$ has too much symmetry. 
For example, if $f$ is rotationally invariant, then so is its Fourier transform $\mathcal F[f]$,
and therefore $\nu_s=\nu_t$ for all $s,t\in[0,T]$, which clearly makes it impossible to reconstruct any rotation. 
In the generic case, however, we expect that this problem does not occur and neither does it in our numerical examples.
Our variational model in Section \ref{sec:recon}
exploits both the curves and their dual versions.

We finish this section by two different remarks concerning common lines and common circles.

\begin{remark}[Methods of common circles and common lines]
We can obtain the relative rotation $R_s^\top R_t$, which the object undergoes at only \emph{two} different time steps $t$ and $s$, 
from two data sets $\nu_s$ and $\nu_t$. This is in contrast to the common line method for the ray transform \cite{ElbRitSchSchm20}, which requires to detect the intersection of pairwise common lines in images from at least \emph{three} different rotations in order to calculate the relative rotations, see \autoref{fig:common_line}. 
This is also apparent from the fact that our common circle formulation contains the three Euler angles from the three dimensional manifold $\SO$ as parameters, whereas in the common line method for the ray transform there are only the two parameters parameterizing the lines through the origin. 
\end{remark}

\begin{remark}[Stereographic projection]
  Instead of detecting elliptic arcs in the images $\nu_t$, we may use a different parameterization in which the arcs become straight lines.
  Such a parameterization is provided by the stereographic projection of the hemisphere $\mathcal H_t\setminus\{\zb0\}$ from the origin onto its equatorial plane
  $ \{\bx\in\R^3:\langle\bx,R_t\be^3\rangle = -k_0\}. $
  The stereographic projection maps the common circle \eqref{eq:YsYt_param} to a straight line in the equatorial plane.
  Then we can detect the rotations $R_s^\top R_t$ by finding common lines in the transformed, two-dimensional data.
  The details are provided in \autoref{se:stereo}.
\end{remark}

\section{Infinitesimal Common Circle Method} \label{sec:cont-cc}
%
In this section, we make the additional assumption that the rotations $R_t$ to which the object 
is exposed via the transformation \eqref{eq:motion0} depend smoothly on the time $t\in[0,T]$, i.e.,
we assume that $R\in C^1([0,T]\to\SO)$, where we consider $\SO$ as submanifold of $\R^{3\times3}$. Since the fact that the scattering potential $f$ 
has compact support implies by the Paley--Wiener theorem that $\ktran [f] \in C^\infty(\R^3 \to \C)$, 
we thus have that $\nu_t(\bk)$ is continuously differentiable both in time $t$ and space $\bk$.

In this setting, we can describe the relative rotation $R_s^\top R_t$ between two time steps $s,t\in[0,T]$ 
in the limit $s\to t$ by the derivative $(R_t')^\top R_t=(R_t^\top R_t')^\top$, where $R_t'$ 
denotes the time derivative of $R_t$ at $t\in[0,T]$.
The derivative of the defining identity $R_t^\top R_t = I$ 
with respect to $t$ is given by
$R_t^\top R_t' + (R_t^\top R_t')^\top =0$. 
Hence the \emph{angular velocity matrix} $W_t\coloneqq R_t^\top R_t'$ is skew-symmetric and thus can be described 
by three parameters. The \emph{angular velocity} $\zb\omega_t\colon[0,T]\to\R^3$ of the rotational motion $R_t$, 
cf.\ \cite[Chapter~VI]{LanLif81}, is defined by
\begin{equation}\label{eq:omega}
  R_t^\top R_t'\zb y
  = \zb\omega_t\times \zb y
  \quad \text{for all }
	t\in[0,T],\;\zb y\in\R^3.
\end{equation}
In particular, we have
\begin{equation}\label{eq:W}
  W_t=
  \begin{pmatrix}
    0&-\omega_{t,3}&\omega_{t,2}\\ 	
    \omega_{t,3}&0&-\omega_{t,1}\\ 
    -\omega_{t,2}&\omega_{t,1}&0
  \end{pmatrix}
,\quad\text{where}\quad 
\zb\omega_t \coloneqq (\omega_{t,1},\omega_{t,2},\omega_{t,3})^\top.
\end{equation}
In the following, we want to reconstruct the angular velocity $\zb\omega_t$ of the rigid motion at a time $t\in[0,T]$ from the behavior of the data $\nu$ in the vicinity of the time $t$, more precisely from the first derivatives of $\nu$ at the time $t$. 
We will utilize a similar approach as done for the ray transform in \cite{ElbRitSchSchm20}.

For the reconstruction, it is convenient to express $\zb\omega_t$ in cylindrical coordinates
\begin{equation}\label{eq:omega-decomp}
  \zb\omega_t 
  = \begin{pmatrix} \rho_t\zb\phi_t \\ \zeta_t \end{pmatrix}
  = \begin{pmatrix} \rho_t\phi_{t,1} \\ \rho_t\phi_{t,2} \\ \zeta_t \end{pmatrix}
\end{equation}
with the azimuth direction $\zb\phi_t\in\S^1_+ \coloneqq\{(\cos(\alpha),\sin(\alpha))^\top:\alpha\in[0,\pi)\}$, the cylindrical radius $\rho_t\in\R$, 
and the third component $\zeta_t\in\R$.
Note that, in contrast to conventional cylindrical coordinates, we allow negative radii $\rho_t$, but restrict in exchange $\zb\phi_t$ to $\S^1_+$.

To obtain the reconstruction formula, we consider for values $s,t\in[0,T]$ with $R_s\be^3\ne \pm R_t\be^3$ again the elliptic arcs $\bgam_{s,t}$  in \eqref{eq:gamma} fulfilling the identity
\begin{equation}\label{eq:commonCircleIdentity1}
\nu_s(\bgam_{s,t}(\beta)) = \nu_t(\bgam_{t,s}(-\beta)),
\end{equation}
in \eqref{eq:commonCircleIdentity}. Taking therein for fixed $t\in[0,T]$ the limit $s\to t$, we find that the relation
\begin{equation}\label{eq:commonCircleIdentityInf}
\lim_{s\to t}\frac{\nu_s(\bgam_{s,t}(\beta))-\nu_t(\bgam_{t,s}(-\beta))}{s-t} = 0
\end{equation}
holds, which gives us a relation between the first order derivatives of the function $(t,\bk)\mapsto\nu_t(\bk)$ involving only the angular velocity $\zb\omega_t$ of the rotations at the time $t$. As we will see, this first order part of \eqref{eq:commonCircleIdentity1} contains enough information to recover the angular velocity and therefore the whole rotations. This method can thus be seen as an infinitesimal version of \autoref{th:commonCircle}.
Rewriting the relation \eqref{eq:commonCircleIdentityInf} 
directly by expanding the functions in Taylor series in the variable $s$ around the point $t$ would be rather tedious, similar to the calculation in \cite{ElbRitSchSchm20} for the ray transform.
Therefore, we will simply verify the following lemma via direct computation.

\begin{lemma}[Infinitesimal common circle relation]\label{th:infRel}
Let the rotations $R\in C^1([0,T]\to\SO)$ be continuously differentiable and the associated angular velocities $\zb\omega_t\in \R^3$ be written in cylindrical coordinates \eqref{eq:omega-decomp}.
Then we have for every $r\in(-k_0,k_0)$ and $t\in[0,T]$ the relation
\begin{equation}\label{eq:infRel}
\partial_t\nu_t(r\zb\phi_t)=\left(\rho_t\left(k_0-\sqrt{k_0^2-r^2}\right)+r\zeta_t\right)\left<\nabla\nu_t(r\zb\phi_t),\begin{pmatrix}-\phi_{t,2}\\\phi_{t,1}\end{pmatrix}\right>,
\end{equation}
where $\partial_t\nu_t$ denotes the partial derivative of $\nu_t$ with respect to $t$, and $\nabla \nu_t(\bk)$ the gradient with respect to $\bk$.
\end{lemma}

This gives rise to the following reconstruction method. 
For the reconstruction to be unique, we require that 
\eqref{eq:infRel} has a unique solution $(\rho_t, \zb\phi_t, \zeta_t) \in \R \times \S^1_+ \times \R$,
which consists of the components of the angular velocity we want to reconstruct.
If the object (and therefore $f$) is asymmetric, it seems reasonable to assume there is indeed a unique solution $(\rho_t,\zb\phi_t,\zeta_t)$ to \eqref{eq:infRel},
which happens in all our numerical simulations.
Conditions for the unique reconstructability of $f$ are discussed in \cite{KetLam11,KurZic21}.

\begin{theorem}[Reconstruction of the angular velocity $\zb\omega_t$]\label{th:infRecon}
Let the rotations matrices $R\in C^1([0,T]\to\SO)$ be continuously differentiable and $t\in[0,T]$.
Let further ${\zb\phi}\in\S^1_+$ be a unique direction
with the property that there exist parameters $\rho,\zeta\in\R$ such that
\begin{equation}\label{eq:infRecon}
\partial_t\nu_t(r{\zb\phi})
=\left(\rho\left(k_0-\sqrt{k_0^2-r^2}\right)+r\zeta\right)
\left<\nabla\nu_t(r{\zb\phi}),
\begin{psmallmatrix}
-\phi_{2}\\
\phi_{1}
\end{psmallmatrix}\right>
\quad\text{for all}\quad r\in(-k_0,k_0).
\end{equation}
Provided that the set
\begin{align*}
\mathcal N_t&\coloneqq\left\{r\in(-k_0,k_0)\setminus\{0\}:\left<\nabla\nu_t(r{\zb\phi}),\begin{psmallmatrix}-\phi_{2}\\\phi_{1}\end{psmallmatrix}\right>\ne0\right\}
\end{align*}
contains at least two elements, then the angular velocity \eqref{eq:omega-decomp} is given by $\zb\omega_t = (\rho{\zb\phi},\zeta)^\top$.
\end{theorem}

\begin{proof}
From \autoref{th:infRel}, we find that the uniqueness implies that $\zb\phi_t=\zb\phi$ and therefore also
\begin{equation*}
\rho+\frac r{k_0-\sqrt{k_0^2-r^2}}\,\zeta = \rho_t+\frac r{k_0-\sqrt{k_0^2-r^2}}\,\zeta_t\quad \text{for all}\quad r\in\mathcal N_t.
\end{equation*}
Since the function $(-k_0,k_0)\setminus\{0\}\to\R\setminus[-1,1]$, $r\mapsto r / ({k_0-\sqrt{k_0^2-r^2}})$ is bijective, we have $\rho=\rho_t$ and $\zeta=\zeta_t$ if the equation is satisfied for two different values $r$.
\end{proof}

An alternative version of the last theorem via stereographic projection is found in Appendix~\ref{se:stereo_inf}.  
Once we have reconstructed the angular velocity $\zb\omega_t$ by the above theorem, we can obtain the rotation matrices $R_t$ as follows.

\begin{theorem}[Reconstruction of the rotation from the angular velocity]\label{th:infReconR}
  Let the rotations $R\in C^1([0,T]\to\SO)$ be continuously differentiable with associated angular velocity $\zb\omega$, see \eqref{eq:omega}.
  Then $R$ is the unique solution of the linear initial value problem
  \begin{equation} \label{eq:ode}
    \begin{split}
      R_t' &= R_tW_t,\quad t\in(0,T), \\
      R_0 &= I,
    \end{split}
  \end{equation}
  where the skew-symmetric matrix $W_t\in\R^{3\times3}$ is defined in \eqref{eq:W}. 
\end{theorem}

\begin{proof}
  The equation \eqref{eq:ode} follows directly from the definition \eqref{eq:omega} of the angular velocity.
  As a linear ordinary differential equation, the initial value problem \eqref{eq:ode} has a unique solution.
\end{proof}

\begin{remark}[Differences to infinitesimal common line method]
In contrast to the infinitesimal common line method \cite{ElbRitSchSchm20} for the ray transform \eqref{eq:ray}, 
which requires third order derivatives of the data function,
we only need first order derivatives of~$\nu$ in order to reconstruct the angular velocity completely.
Furthermore, we can uniquely recover the rotation, whereas for the ray transform there are always two possible solutions 
corresponding to a reflection in the direction of the imaging wave.
\end{remark}

\section{Reconstruction of the Translations} \label{sec:translations}
So far, we have only considered the computation of the rotations $R_t$, $t\in[0,T]$, in the motion \eqref{eq:motion0}, which we could obtain from the absolute values of the Fourier transforms of our measurements $m_t$, $t\in[0,T]$, that is, from the scaled squared energy $\nu_t$ defined in \eqref{eq:nu}.
To recover the translations $\bd_t\in\R^3$, we need to use in addition the phase information in our measurements $m_t$, see \eqref{eq:meas}.
Therefore we define the \emph{scaled measurement data} $\mu_t\colon\mathcal B^2_{k_0}\to\C$ by
\begin{equation} \label{eq:nu-F-complex}
\mu_t(\bk)\coloneqq-\i\sqrt{\frac2\pi}\,\kappa(\bk)\e^{-\i\kappa(\bk)\rM}\mathcal F[m_t](\bk).
\end{equation}
According to \eqref{eq:recon_n} this can be expressed in terms of the scattering potential $f$, the rotation $R_t\in\SO$, and the translation $\bd_t\in\R^3$ by
\begin{equation} \label{eq:mu-F}
\mu_t(\bk) = \ktran [f]\left(R_t\bh(\bk) \right)\e^{-\i\inner{\bd_t}{\bh(\bk)}}. 
\end{equation}
If we have already reconstructed the rotations $R_t$, 
then we know by \eqref{eq:commonCircleIdentity} and \eqref{eq:commonCircleIdentityDual} 
the elliptic arcs $\bgam_{s,t}$ 
and the duals $\zb{\gamma}^*_{s,t}$ along which the 
values of the scaled squared energies $\nu_s=\abs{\mu_s}^2$ and $\nu_t=\abs{\mu_t}^2$ 
coincide.
Therefore, the corresponding values of $\mu_s$ and $\mu_t$ only differ by a phase factor, 
which depends on the translation vectors $\bd_s$ and $\bd_t$.
We compute their relation explicitly in the following lemma.

\begin{lemma}[Complex phase shift along the common circles]\label{th:common-circle:translation}
Let $s,t\in[0,T]$ such that $R_s\be^3\neq \pm R_t\be^3$ and let $\bgam_{s,t}$ and $\bgam_{t,s}$ be the elliptic arcs defined in \autoref{th:gamma} and $\bsigma_{s,t}=R_s(\bh\circ\bgam_{s,t})$ be the corresponding common circular arc introduced in \autoref{th:YsYt_param}. Moreover, let $\zb{\gamma}^*_{s,t}$ and $\zb{\gamma}^*_{t,s}$ be the dual elliptic arcs and 
$\zb{\sigma}^*_{s,t}=R_s(\bh\circ\zb{\gamma}^*_{s,t})$ 
be the corresponding dual common circular arc as defined in \autoref{th:dualCircle}. Then we have
\begin{enumerate}
\item
for every $\beta\in J_{s,t}$ with $\mu_s(\bgam_{s,t}(\beta))\neq0$ that
\begin{equation}\label{eq:common-circle:translation}
\e^{\i\inner{R_t\bd_t-R_s\bd_s}{\bsigma_{s,t}(\beta)}} = \frac{\mu_s(\bgam_{s,t}(\beta))}{\mu_t(\bgam_{t,s}(-\beta))}.
\end{equation}
\item
for every $\beta\in J^*_{s,t}$ with $\mu_s(\zb{\gamma}^*_{s,t}(\beta))\neq0$ that
\begin{equation}\label{eq:common-circle:translationDual}
\e^{\i\inner{R_t\bd_t-R_s\bd_s}{\zb{\sigma}^*_{s,t}(\beta)}} 
= 
\frac{\mu_s(\zb{\gamma}^*_{s,t}(\beta))}{\overline{\mu_t(\zb{\gamma}^*_{t,s}(\beta))}}.
\end{equation}
\end{enumerate}
\end{lemma}

In the degenerate cases $R_s\be^3=\pm R_t\be^3$, a similar relation holds on the whole hemisphere.

\begin{lemma}[Special cases $R_s\be^3=\pm R_t\be^3$]\label{th:commonAxis:translation}
Let $R_s^\top R_t$ be known for some $s,t\in[0,T]$, and let $\mathrm{Q}$ and $\mathrm{S}$ be given in \eqref{eq:RS}.
\begin{enumerate}
\item
If $R_t\be^3=R_s\be^3$,  then we have for all $\bk\in\mathcal B_{k_0}^2$ with $\mu_s(\bk)\ne0$ that
\begin{equation}\label{eq:translDegnPlus}
\e^{\i\inner{R_t\bd_t-R_s\bd_s}{R_s\bh(\bk)}} = \frac{\mu_s(\bk)}{\mu_t(\mathrm{Q}(-\alpha)\bk)}
\end{equation}
with some $\alpha\in\R/(2\pi\Z)$ fulfilling $R_s^\top R_t\be^3=Q^{(3)}(\alpha)$ according to \autoref{th:commonAxis} (i).
\item
If $ R_t\be^3=-R_s\be^3$, then we have for all $\bk\in\mathcal B_{k_0}^2$ with $\mu_s(\bk)\ne0$ that
\begin{equation}\label{eq:translDegnMinus}
\e^{\i\inner{R_t\bd_t-R_s\bd_s}{R_s\bh(\bk)}} = \frac{\mu_s(\bk)}{\overline{\mu_t(\mathrm{Q}(-\alpha)S\bk)}}
\end{equation}
with some $\alpha\in\R/(2\pi\Z)$ fulfilling 
$R_s^\top R_t\be^3=Q^{(2)}(\pi)Q^{(3)}(\alpha)$ according to \autoref{th:commonAxis}~(ii).
\end{enumerate}
\end{lemma}

In contrast to data of the ray transform \eqref{eq:ray}, where the measurements are invariant to the object's position in
direction of the incident wave, the diffraction data are not invariant with respect to the third component of the translations.
By the following theorem, we can uniquely recover the translation vectors~$\bd_t$, $t\in[0,T]$, from the equations \eqref{eq:common-circle:translation}, \eqref{eq:common-circle:translationDual}, \eqref{eq:translDegnPlus}, and \eqref{eq:translDegnMinus}.

\begin{theorem}[Reconstruction of the translation]\label{th:reconTransl}
Let the relative rotation $R_s^\top R_t$ be known for some $s,t\in[0,T]$. 
\begin{enumerate}
\item
If $R_s\be^3\ne\pm R_t\be^3$, then the relative translation $R_s^\top R_t\bd_t-\bd_s$ is uniquely determined from the scaled measurements~$\mu_s$ and~$\mu_t$ by the equations \eqref{eq:common-circle:translation} and \eqref{eq:common-circle:translationDual}.
\item
If $R_s\be^3=\pm R_t\be^3$, then the relative translation $R_s^\top R_t\bd_t-\bd_s$ is uniquely determined from the scaled measurements~$\mu_s$ and~$\mu_t$ by the equation \eqref{eq:translDegnPlus} or \eqref{eq:translDegnMinus}.
\end{enumerate}
\end{theorem}

Thus, if $f$ is sufficiently asymmetric so that we find for each $t\in[0,T]$ a time $s\in[0,T]$, for which $R_s$ was already reconstructed (starting from the normalization $R_0=I$), and such that there either exist unique elliptic arcs in $\nu_s$ and $\nu_t$ as described in \autoref{th:commonCircle}, or we have $R_s\be^3=\pm R_t\be^3$ and there exists a unique angle $\alpha$ fulfilling either \eqref{eq:commonAxis} or \eqref{eq:commonAxisReflection}, then \autoref{th:commonAxis} and \autoref{th:commonCircle} determine uniquely the rotation $R_t$. With this knowledge, we get from \autoref{th:reconTransl} with $\bd_0=\zb0$ all the translations $\bd_t$, $t\in[0,T]$, and therefore the complete motion $\Psi$ of our object, introduced in \eqref{eq:motion0}.

The following remark gives an interesting relation to higher order moment methods.

\begin{remark}[Relation to higher order moment methods]
We can also detect the optical center $\mathcal C\in\R^3$ of the object, that is, the ratio
\[ \mathcal C\coloneqq\frac{\int_{\R^3}\bx f(\bx)\dd x}{\int_{\R^3}f(\bx)\dd x} = \frac{\i\nabla\ktran[f](\zb0)}{\ktran[f](\zb0)} \]
of the first and the zeroth moment of the function $f$, from the transformed data $\mu_t$, see \eqref{eq:nu-F-complex}, by realizing that
\[ \frac{\i\partial_{k_i}\mu_t(\zb0)}{\mu_t(\zb0)} = \frac{\i\inner{\nabla\ktran[f](\zb0)}{R_t\be^i}}{\ktran[f](\zb0)}+d_{t,i} = \inner{\mathcal C}{R_t\be^i}+d_{t,i},\quad i\in\{1,2\}. \]
If we have a time $t\in[0,T]$ for which the rotation $R_t$ has a rotation axis different from $\R\be^3$, this allows us (other than from data of the ray transform) to fully recover the point $\mathcal C\in\R^3$ without the need of first reconstructing $f$.
Theoretically, this approach also provides a reconstruction of arbitrary moments of the function~$f$ by incorporating higher-order derivatives of $\mu_t$, which was used in \cite{KurZic21} to prove the unique reconstructability of $f$.
\end{remark}

\section{Reconstruction Methods} \label{sec:recon}
Based on our previous results we can  provide concrete reconstruction methods for the motion parameters now.
We start by considering the rotations and continue with translations afterwards.

\subsection{Reconstruction of the rotation}

For reconstructing the rotations $R_t$, we can utilize the common circle method in \autoref{sec:comm_circ} 
and its infinitesimal counterpart in \autoref{sec:cont-cc}.
Here, we assume that $\nu_t$ from \eqref{eq:nu} are given.

\subsubsection{Direct common circle method} \label{sec:recDirect}
We utilize \autoref{th:commonCircle} to find the common circles and therefore reconstruct the rotation~$R_t$.
We want
to find the Euler angles $(\varphi,\theta,\psi)\in(\R/(2\pi\Z))\times[0,\pi]\times(\R/(2\pi\Z))$ 
of the rotation $R_s^\top R_t$ by solving \eqref{eq:nuGamma} and \eqref{eq:nuGammaDual}.
When working with measured data, it is unlikely that these equations hold exactly,
therefore we propose a least-squares approach:
we aim to minimize the functional
\begin{equation}\label{eq:dist}
  \mathcal E_{s,t}(\varphi,\theta,\psi)
  \coloneqq
  \int_{-\pi/2}^{\pi/2} 
  \abs{\nu_t (\bgam^{\pi-\psi,\theta}(-\beta))
    -\nu_s (\bgam^{\varphi,\theta}(\beta))}^2
  + \abs{\nu_t (\bgam^{*,\pi-\psi,\theta}(\beta))
    -\nu_s (\bgam^{*,\varphi,\theta}(\beta))}^2
  \dd \beta
\end{equation}
over $\varphi,\psi\in\R/(2\pi\Z)$ and $\theta\in[0,\pi]$,
where the elliptic arcs $\bgam^{\varphi,\theta}$ and $\bgam^{*,\varphi,\theta}$ are given in \eqref{eq:gammaEuler} and \eqref{eq:dualEllipseEuler}.
If $\nu_t$ is given on a grid, it needs to be interpolated in order to evaluate $\nu_t (\bgam^{\pi-\psi,\theta}(-\beta))$ in \eqref{eq:dist}.
Furthermore, the integral in \eqref{eq:dist} can be discretized via quadrature.

We consider the minimizer of $\mathcal E_{s,t}$ as a good approximation of the Euler angles \eqref{eq:euler} of the rotation $R_s^\top R_t$.
Using that $R_0=I$, we compute the minimizer of $\mathcal E_{0,t}$ to obtain $R_0^\top R_t = R_t$ for all $t$ by \autoref{alg:cc}.

\begin{algorithm}
  \KwIn{Scaled squared energy $\nu_t$, discretization parameter $N\in\N$, grid of Euler angles $(\varphi_\ell,\theta_\ell,\psi_\ell) \subset [0,2\pi)\times[0,\pi]\times[0,2\pi)$, $\ell=1,\dots,L$.}
  
  
  Set the grid $\beta_n \coloneqq n\pi/N,$ $n=-N,\dots,N$\;
  
  \For{$\ell=1,\dots,L$}{    
    Compute 
    $\mathtt E(\ell)
    \coloneqq
    \sum_{n=-N}^{N} 
    \abs{\nu_t (\bgam^{\pi-\psi_\ell,\theta_\ell}(-\beta_n))
      -\nu_0 (\bgam^{\varphi_\ell,\theta_\ell}(\beta_n))}^2 +\abs{\nu_t (\bgam^{*,\pi-\psi_\ell,\theta_\ell}(\beta_n))
      -\nu_0 (\bgam^{*,\varphi_\ell,\theta_\ell}(\beta_n))}^2$ using an interpolation of $\nu_t$ and $\nu_0$\;
  }
  Compute $\hat \ell = \arg\min_\ell \mathtt E$\;
  
  \KwOut{Rotation $R_t \approx Q^{(3)}(\varphi_{\hat \ell}) Q^{(2)}(\theta_{\hat \ell}) Q^{(3)}(\psi_{\hat \ell})$, see \eqref{eq:euler}.}
  
  \caption{Reconstruction of the rotation $R_t$ with the direct common circle method}
  \label{alg:cc}
\end{algorithm}

The accuracy of $R_t$ may be improved by incorporating reconstructions of $R_s^\top R_t$ for all $s,t$, similarly to cryo-EM~\cite{WanSinWen13}.
The minimization of $\mathcal E_{s,t}$ is a three-dimensional, non-linear and non-convex optimization problem, for which we can use a brute force method by searching on a grid of $\SO$.
The computation of the minimum of \eqref{eq:dist} becomes much more efficient if we have good initial values $(\varphi,\theta,\psi)$, which can be obtained by the infinitesimal method in the next subsection.

\subsubsection{Infinitesimal common circle method}
\label{sec:recInf}
The reconstruction is done in two steps. First, we reconstruct the angular velocity, second we use this to find the rotation.

\paragraph{Angular velocity}

Let $t \in (0,T)$ be arbitrary but fixed.
We reconstruct the angular velocity $\zb\omega_t=(\rho_t\cos(\phi_t),\rho_t\sin(\phi_t),\zeta_t)^\top$, $\rho_t,\zeta_t\in\R$, $\phi_t\in[0,\pi)$ using \autoref{th:infRecon}.
In particular, we construct a functional that we minimize over $(\rho,\phi,\zeta)$ in order to find the exact parameters $(\rho_t,\phi_t,\zeta_t)$.
For $r \in (-k_0,k_0)$ and $\phi\in[0,\pi)$, we set
\begin{equation}\label{eq:gh}
  \begin{aligned}
    g_{\phi}(r)
    &\coloneqq \partial_t\nu_t(r\cos(\phi),r\sin(\phi)),
    \\
    p_{\phi}(r)
    &\coloneqq
    \left(k_0-\sqrt{k_0^2-r^2}\right) \left<\nabla\nu_t(r\cos(\phi),r\sin(\phi)),\begin{psmallmatrix}-\sin(\phi)\\\cos(\phi)\end{psmallmatrix}\right>
    ,\\
    q_{\phi}(r)
    &\coloneqq
    r\left<\nabla\nu_t(r\cos(\phi),r\sin(\phi)),\begin{psmallmatrix}-\sin(\phi)\\\cos(\phi)\end{psmallmatrix}\right>.
  \end{aligned}
\end{equation}
Note that these functions are indeed continuous and they are obtained by differentiating the scaled squared energy $\nu_t$.
Then relation \eqref{eq:infRel} of the angular velocity can be written as
\begin{equation} \label{eq:d-nu-g}
  g_{\phi_t}(r) 
  = \rho_t\, p_{\phi_t}(r) + \zeta_t\, q_{\phi_t}(r)
  ,\qquad r\in(-k_0,k_0).
\end{equation}
As in the direct method, we use a least squares approach to solve \eqref{eq:d-nu-g} in order to recover the quantities $\rho_t$, $\phi_t$ and $\zeta_t$.
We want to minimize 
the functional
\begin{equation} \label{eq:J}
  \mathcal J(\rho,\phi,\zeta)
  \coloneqq \norm{g_{\phi} - \rho \, p_{\phi} - \zeta \, q_{\phi}}_{L^2(-k_0,k_0)}^2
  ,\qquad \rho,\zeta\in\R,\;\phi\in[0,\pi),
\end{equation}
which vanishes according to \eqref{eq:d-nu-g} for $(\rho,\phi,\zeta)=(\rho_t,\phi_t,\zeta_t)$, so that the desired angular velocity $\zb\omega_t$ is indeed a minimizer of $\mathcal J_t$.

We minimize $\mathcal J$ by a brute-force method.
For every $\phi \in [0,\pi)$ on a fixed grid, we compute the minimizer of the functional
$\mathcal J_{\phi}\colon\R^2\to\R$,
$\mathcal J_{\phi}(\rho,\zeta)\coloneqq\mathcal J(\rho,\phi,\zeta)$,
which we can explicitly get from the optimality condition
\begin{equation*}
  \zb0
  =
  \nabla\mathcal J_{\phi}(\rho,\zeta)
  =
  \begin{pmatrix}
    \frac{\partial}{\partial \rho}\mathcal J_{\phi}(\rho,\zeta)\\
    \frac{\partial}{\partial \zeta}\mathcal J_{\phi}(\rho,\zeta)
  \end{pmatrix}
  = 2
  \begin{pmatrix}
    \rho \inn{p_{\phi},p_{\phi}} - \inn{g_{\phi},p_{\phi}} + \zeta \inn{p_{\phi},q_{\phi}}\\
    \zeta \inn{q_{\phi},q_{\phi}} - \inn{g_{\phi},q_{\phi}} + \rho \inn{p_{\phi},q_{\phi}}
  \end{pmatrix},
\end{equation*}
where $\left<\cdot,\cdot\right>$ denotes the inner product on $L^2((-k_0,k_0)\to\R)$ here.
Provided that the functions $p_{\phi}$ and $q_{\phi}$ are linearly independent in $L^2((-k_0,k_0)\to\R)$, 
so that by Cauchy-Schwarz' inequality $\norm{p_\phi}_{L^2}\norm{q_\phi}_{L^2} \neq \abs{\inn{p_\phi,q_\phi}}$,  the above system has a unique solution.
The unique minimizer $(\hat\rho(\phi),\hat\zeta(\phi))$ is then given by
\begin{equation} \label{eq:Jmin}
\begin{pmatrix} \hat\rho(\phi)\\ \hat\zeta(\phi) \end{pmatrix} 
= 
\begin{pmatrix}
  \inn{p_{\phi},p_{\phi}} 
	& \inn{p_{\phi},q_{\phi}} \\ \inn{p_{\phi},q_{\phi}} & \inn{q_{\phi}, q_{\phi}}
\end{pmatrix}^{-1}
\begin{pmatrix} \inn{g_{\phi},p_{\phi}}\\ \inn{g_{\phi},q_{\phi}} \end{pmatrix}. 
\end{equation}
For every $\phi\in[0,\pi)$ on the grid, we thus first calculate the value
\begin{equation} \label{eq:j}
  j(\phi)
  \coloneqq
  \min_{\rho,\zeta\in\R} \mathcal J(\rho,\phi,\zeta) 
	= 
	\mathcal J(\hat\rho(\phi),\phi,\hat\zeta(\phi)),
\end{equation}
then we take as approximation of the angle $\phi_t$ in the angular velocity $\zb\omega_t$ the minimizer $\hat\phi\in[0,\pi)$ of $j(\phi)$ on this grid, 
and pick $\hat\rho(\hat\phi)$ and $\hat\zeta(\hat\phi)$
as approximations for $\rho_t$ and $\zeta_t$.
The reconstruction is summarized in \autoref{alg:infOmega}.

\begin{algorithm}
  \KwIn{Scaled squared energy $\nu_t(r_n \cos \phi_\ell, r_n\sin\phi_\ell)$ from \eqref{eq:nu} on a polar grid $r_n\in[0,k_0)$, $n=1,\dots,N$,
  and $\phi_\ell\in[0,\pi)$, $\ell=1,\dots,L$.}
  
  \For{$\ell=1,\dots,L$}{
    Compute the functions $g_{\phi_\ell}(r_n),$ $p_{\phi_\ell}(r_n),$ and $q_{\phi_\ell}(r_n)$, $n=1,\dots,N$ by \eqref{eq:gh}\;
    
    Compute 
    $\hat\rho(\phi_\ell)$ and $\hat\zeta(\phi_\ell)$ by \eqref{eq:Jmin};
    
    Compute $j(\phi_\ell) \coloneqq 
    \sum_{n=1}^{N} \abs{g_{\phi_\ell}(r_n) - \hat\rho(\hat\phi_\ell) \, p_{\phi_\ell}(r_n) 
		- \hat\zeta(\hat\phi_\ell) \, q_{\phi_\ell}(r_n)}^2
    $
  }
  Set $\hat\phi$ as minimizer of $j(\phi_\ell)$ over $\ell=1,\dots,L$\;
  
  \KwOut{Angular velocity $\zb\omega_t \approx \left(\hat\rho(\hat\phi)\, \cos\hat\phi,\,
    \hat\rho(\hat\phi)\, \sin\hat\phi,\,
    \hat\zeta(\hat\phi) \right)$.}
  
  \caption{Reconstruction of the angular velocity $\zb\omega_t$ with the infinitesimal method}
  \label{alg:infOmega}
\end{algorithm}

\begin{remark}\label{rem:phi-sign}
  The minimizer of $\mathcal J$ might not be unique in general, 
  depending on the symmetry of the scattering potential $f$.
  In the described method, there are two steps of possible non-uniqueness:
  Firstly, the functions $p_{\phi}$ and $q_{\phi}$ might be linearly dependent, then the minimizer of $\mathcal J_{\phi}$ is not unique.
  Secondly, the subsequent minimization over $\phi\in[0,\pi)$ might lead to more than one minimum point.
  However, in our numerical tests below with non-symmetric functions $f$, we always computed approximately the correct minima.
\end{remark}

\begin{remark} \label{rem:only-rotation}
  If the object rotates around the origin without any further translation,
  i.e., $\bd_t= \zb0$ for all $t$,
  then the function $\nu_t$ in Algorithms \ref{alg:cc} and \ref{alg:infOmega} for the common circle and infinitesimal common circle methods 
  can be replaced by the complex-valued function $\mu_t$ from \eqref{eq:mu-F}, where then the exponent in the phase factor vanishes.
\end{remark}

\paragraph{Rotation matrix}
We compute rotation matrices $R_t$ given angular velocities $\zb\omega_t$ for all $t\in(0,T)$.
According to \autoref{th:infReconR}, we can obtain $R_t$ from the angular velocity~$\zb \omega_t$ with the corresponding coefficient matrix $W_t$ by solving
the initial value problem \eqref{eq:ode} 
which has a unique solution $R_t\in\SO$ by~\cite[Section~IV.4]{HaiLubWan06}.
Numerically, we solve \eqref{eq:ode} with the forward \emph{Euler method} on $\R^{3 \times 3}$
followed by a so-called \emph{retraction} $\text{P}_{R_t}$, see \cite{ABS2008}, 
which maps the tangential vectors from the tangent space $\{R_t S: S \in \R^{3\times3} \text{ skew-symmetric} \}$
at $R_t \in \SO$ to $\SO$, in each iteration step.
More precisely, using discrete time steps $t_j \coloneqq j/n$ of resolution $n\in\N$, we compute for $j=0,\ldots,\lfloor Tn\rfloor$,
the reconstructed rotation matrix $\mathbf R_j$ by
\begin{equation} \label{eq:Euler-proj}
\begin{aligned}
  \mathbf R_0
  &\coloneqq I,
  \\
  \mathbf R_{{j+1}} &\coloneqq \text{P}_{\mathbf R_{j}}( (t_{j+1}-t_j)\,  \mathbf R_{j} W_{t_j}).
\end{aligned}
\end{equation}
Since the manifold $\SO$ is smooth and if we further assume the slightly higher regularity $R\in C^2([0,T]\to\SO)$, it is known that this method
converges with the same order as the classical Euler method, see \cite[Section~IV.4]{HaiLubWan06}.
The reconstruction is summarized in \autoref{alg:infR}.

Several retractions, that are  computations of $\text{P}_{\mathbf R_{j}}$, are possible in \eqref{eq:Euler-proj}, see \cite[Example 1.4.2]{ABS2008}, 
and we state two popular ones in the following.
If a large number of computations is necessary, e.g., when training neural networks, the chosen method influences the computational time substantially, 
see \cite{HHNPSS2020}.
\begin{enumerate}
\item Polar decomposition: Starting with the singular value decomposition $A = U\Sigma V^\top$ of a matrix $A\in\R^{3\times3}$ 
with diagonal matrix $\Sigma$ and orthogonal matrices $U$ and $V$,
its polar decomposition is given by  $A = \text{Polar}(A)\tilde\Sigma$,
where $\text{Polar}(A) \coloneqq U V^\top$ and $\tilde\Sigma \coloneqq V\Sigma V^\top$. 
If $\det A>0$, then $\text{Polar}(A)\in\SO$.
It is well-known, see, e.g.,  \cite{Moa02}, 
that $\text{Polar}(A)$ is the orthogonal projection of $A$ onto $\SO$ with respect to the 
Frobenius norm $\| \cdot \|_F$, i.e.,
$
  \text{Polar}(A) = \operatorname{argmin}_{Q \in \SO} \|A - Q\|_F.
$
Hence a retraction is given by 
\begin{equation} \label{eq:Polar}
\text{P}_{\mathbf R_{j}} (\tilde W) = \text{Polar}(\mathbf R_{j} + \tilde W).
\end{equation}
\item Cayley transform:  
Based on the \emph{Cayley transform},
a retraction is given for a skew-symmetric matrix $\tilde W\in\R^{3\times3}$ by
\begin{equation} \label{eq:Cay}
  \text{P}_{\mathbf R_{j}} (\tilde W)=
  {\mathbf R_{j}} \operatorname{Cay} ({\tilde W})
  ,\quad \text{where} \quad
  \operatorname{Cay} (\tilde W)
  \coloneqq
  (I- \tfrac12 \tilde W)^{-1}\,(I+ \tfrac12{\tilde W}).
\end{equation}
\end{enumerate}

\begin{algorithm}
  \KwIn{Angular velocity $\zb\omega_{t_j}$ on a grid $t_j = j/n$ for $j=1,\dots,\lfloor Tn\rfloor$.}
    
  Set $\mathbf R_0 \coloneqq I$\;
  
  \For{$j=0,\dots,\lfloor Tn\rfloor$}{
    Compute $W_{t_j}$ by \eqref{eq:W}\;
    
    Compute $\mathbf R_{{j+1}} \coloneqq \text{P}_{\mathbf R_{j}}( (t_{j+1}-t_j)\,  \mathbf R_{j} W_{t_j})$,
    where $\text{P}_{\mathbf R_{j}}$ is either \eqref{eq:Polar} or \eqref{eq:Cay}\;
  }
  
  \KwOut{Rotation $R_{t_j} \approx \mathbf R_{j}$.}
  
  \caption{Reconstruction of the rotation matrices $R_t$ with the infinitesimal method}
  \label{alg:infR}
\end{algorithm}

\subsection{Reconstruction of the translations}
The reconstruction of the translations $\bd_t$ is based on \autoref{th:reconTransl}.
We assume that the rotations $R_t$ are known from the section above.
We numerically solve
the nonlinear equations \eqref{eq:common-circle:translation} and \eqref{eq:common-circle:translationDual} for some $s,t\in[0,T]$ with the following approach.

The left-hand side of \eqref{eq:common-circle:translation} is continuous with respect to $\beta$,
and the term $\bsigma_{s,t}(\beta)$ in its exponent vanishes for $\beta=0$. 
We couple the logarithm of \eqref{eq:common-circle:translation} with a phase unwrapping \cite{Ito82}, which selects the correct branch of the complex logarithm by imposing the continuity of the desired function.
Note that the branches of the logarithm differ by adding $2\pi \i$.
We obtain
the linear system
\begin{equation} \label{eq:common-circle:unwrap}
  \inner{R_t\bd_t - R_s\bd_s}{\bsigma_{s,t}(\beta)}
  =
  \operatorname{unwrap} \left( \frac1\i\log \left( \frac{\mu_s(\bgam_{s,t}(\beta))}{\mu_t(\bgam_{t,s}(-\beta))} \right) \right)
  , \qquad \beta\in J,
\end{equation}
where $\operatorname{unwrap}$ denotes a phase unwrapping  that vanishes at $\beta=0$ and $J\subset J_{s,t}$ is an interval around $0$ on which $\mu_s\circ\bgam_{s,t}$ is nowhere zero.
Discretizing the interval $J$, we see that \eqref{eq:common-circle:unwrap} is a linear system of equations in $\bd_{t;s}$.
Analogously, we obtain from \eqref{eq:common-circle:translationDual} the equation
\begin{equation} \label{eq:common-circle:unwrapDual}
  \inner{R_t\bd_t - R_s\bd_s}{\bsigma^*_{s,t}(\beta)}
  =
  \operatorname{unwrap} \left( \frac1\i\log \left( \frac{\mu_s(\bgam^*_{s,t}(\beta))}{\mu_t(\bgam^*_{t,s}(\beta))} \right) \right)
  , \qquad \beta\in J.
\end{equation}
If $s=0$, we have $\bd_0=\zb0$ so that \eqref{eq:common-circle:unwrap} and \eqref{eq:common-circle:unwrapDual}
contain as unknown only $\bd_t$;
then we reconstruct $\bd_t$ as minimum norm solution fulfilling both \eqref{eq:common-circle:unwrap} and \eqref{eq:common-circle:unwrapDual}.
The procedure is summarized in \autoref{alg:translation}.
Note that \autoref{th:common-circle:translation} guarantees a unique solution of the continuous problem.
In order to improve the reconstruction for inexact data, we can also consider \eqref{eq:common-circle:unwrap} and \eqref{eq:common-circle:unwrapDual} for many pairs of $s$ and $t$ resulting in a large system of equation and solve for $\bd_t$ for all $t$ simultaneously, again incorporating $\bd_0 = \zb0$.

\begin{algorithm}
  \KwIn{Scaled squared energy $\mu_t$ from \eqref{eq:nu-F-complex} and rotations $R_t$.}
  
  Set $\bd_0 \coloneqq \zb0$\;
  
  \For{$j=0,\dots,\lfloor Tn\rfloor$}{
    Compute $\bd_{t_j}$ as the minimum norm least squares solution of both \eqref{eq:common-circle:unwrap} and \eqref{eq:common-circle:unwrapDual} with $s=0$\;
  }
  
  \KwOut{Translations $\bd_{t_j}$.}
  
  \caption{Reconstruction of the translation $\bd_t$}
  \label{alg:translation}
\end{algorithm}


\section{Numerical Simulations} \label{sec:numerics}
We perform numerical tests of the reconstruction algorithms from \autoref{sec:recon}.
We compare the approaches of Sections~\ref{sec:recDirect} and \ref{sec:recInf} for the case that the motion depends smoothly on time.
We consider two three-dimensional test functions for the scattering potential $f$, 
namely a cell phantom in \autoref{fig:cell}, which consists of different convex and concave shapes with constant function values,
and the Shepp--Logan phantom in \autoref{fig:sl}.
Both are evaluated on a uniform $N\times N\times N$ grid with $N=160$
and they are not rotationally symmetric. 
Otherwise, any symmetry would cause the motion detection to have multiple solutions.

\tikzset{font=\footnotesize}
\pgfplotsset{
  colormap={parula}{
    rgb255=(53,42,135)
    rgb255=(15,92,221)
    rgb255=(18,125,216)
    rgb255=(7,156,207)
    rgb255=(21,177,180)
    rgb255=(89,189,140)
    rgb255=(165,190,107)
    rgb255=(225,185,82)
    rgb255=(252,206,46)
    rgb255=(249,251,14)}}

\begin{figure}[!htp]\centering
  \begin{subfigure}{.49\textwidth}
      \begin{tikzpicture}
        \begin{axis}[
          width=.77\textwidth, 
          height=.616\textwidth, 
          enlargelimits=false,
          scale only axis,
          axis on top,
          colorbar,colorbar style={
            width=.15cm, xshift=-0.6em 
            },
          ]
          \addplot[point meta min=-0.09,point meta max=1.09] graphics [
          xmin=-28.28, xmax=28.28,
          ymin=-28.28, ymax=28.28,
          ] {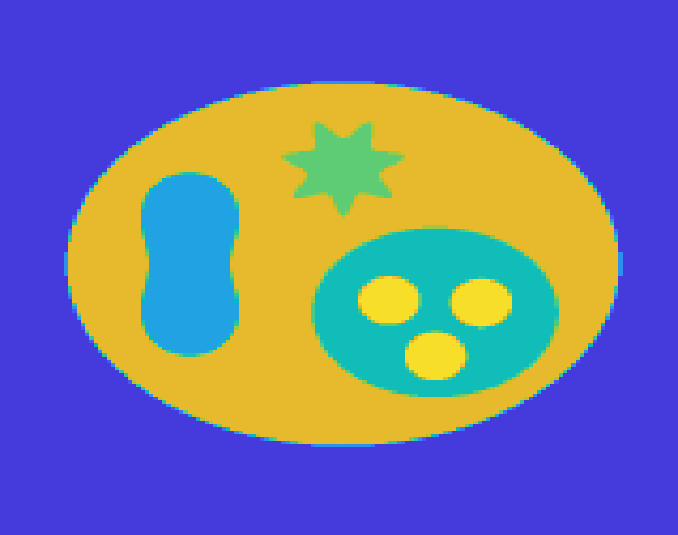};
        \end{axis}
      \end{tikzpicture}
      \caption{Cell phantom \label{fig:cell}}
    \end{subfigure}
  \begin{subfigure}{.49\textwidth}
  \begin{tikzpicture}
    \begin{axis}[
      width=.77\textwidth, 
      height=.616\textwidth, 
      enlargelimits=false,
      scale only axis,
      axis on top,
      colorbar,colorbar style={
        width=.15cm, xshift=-0.6em 
      },
      ]
      \addplot[point meta min=-0.09,point meta max=1.09] graphics [
      xmin=-28.28, xmax=28.28,
      ymin=-28.28, ymax=28.28,
      ] {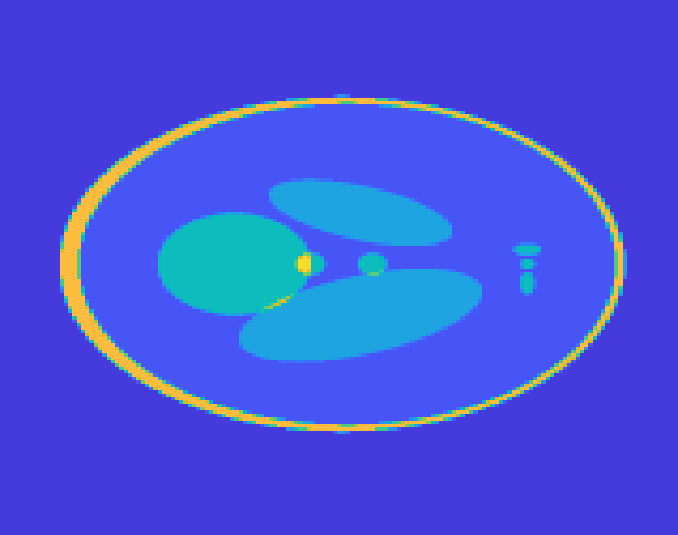};
    \end{axis}
  \end{tikzpicture}
  \caption{Shepp--Logan phantom \label{fig:sl}}
\end{subfigure}
    \caption{
      Slice plots the 3D phantoms $f$ at $x_3=0$.
      \label{fig:f}}
\end{figure}

The data $\nu_t$ is computed ``in silico'' via a numerical approximation of the Fourier transform $\mathcal F [f]$, where $f$ is discretized on a finer grid of $(3N)^3\approx 1.1\cdot10^8$ points.
This approximation is done with the nonuniform fast Fourier transform (NFFT) algorithm \cite{st97}, the same way as in \cite{KirQueRitSchSet21}.
We evaluate $\nu_t$ on a polar grid $(r \cos(\phi),r\sin(\phi))^\top$ on $\mathcal B^2_{k_0}$ consisting of $2N$ points in $r\in(-k_0,k_0)$ and $2N$ points of $\phi\in[0,\pi)$.
We set the wave number $k_0=2\pi$, which corresponds to a wavelength of one of the incident wave.
Furthermore, we have $4N=640$ equispaced samples of the time $t\in[0,2\pi)$.
The high number of grid points yields in a good numerical approximation of the time-derivative.
In total, we sample $\nu$ on about 65 million data points.
We first consider the case that the object is only rotated, but not translated, then we utilize the complex-valued $\mu_t$ to reconstruct the rotation, see \autoref{rem:only-rotation}.

\paragraph{Infinitesimal method}

Let
$\S^2
\coloneqq
\{\bx\in\R^2 : \norm{\bx} = 1\}
$
denote the two-dimensional sphere.
We first consider a constant rotation axis $\bn \in\mathbb S^2$ and the rotation angle $t\in[0,2\pi]$, that is, $R(t)=\exp(tN)$ with $N\in\R^{3\times3}$ defined by $N\bx=\bn\times\bx$ for all $\bx\in\R^3$, since we then have, according to Rodrigues' rotation formula, $R(t)\bx=\inner{\bn}{\bx}\bn+\sin(t)\bn\times\bx+\cos(t)(\bn\times\bx)\times\bn$ for all $\bx\in\R^3$. The angular velocity is in this case therefore the constant function $\zb\omega_t = \bn$.
In \autoref{fig:j}, we show the misfit functional $j$ from \eqref{eq:j} for the rotation axis $\bn=(0.96\cos(\pi/4),0.96\sin(\pi/4),0.28)^\top$ at the time $t=\pi/4$.
One can clearly spot the expected minimum of $j$ at $\hat\phi=\pi/4$.
Furthermore, we show in \autoref{fig:omega-error} the error of the angular velocity reconstructed with \autoref{alg:infOmega} for all time steps $t$ corresponding to a full turn of the object.
We note that the radius $\rho_t$ has a higher error than the other components.

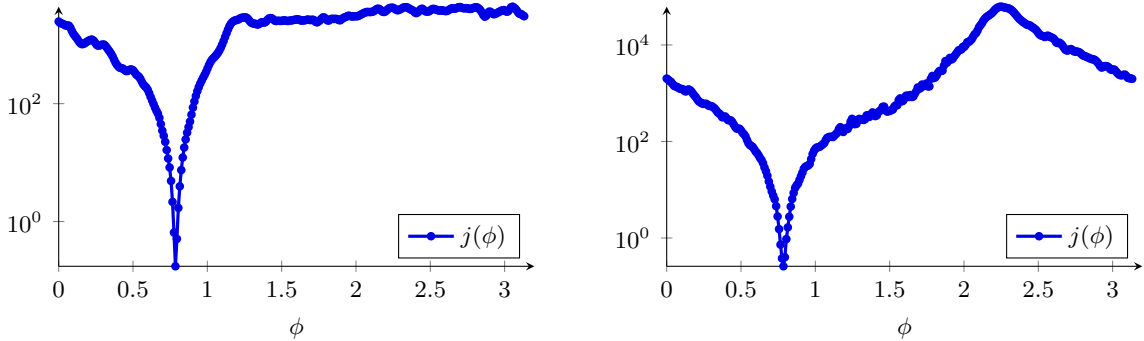
\begin{figure}[!ht]\centering
  \begin{tikzpicture}
    \begin{semilogyaxis}[xlabel=$\phi$, width=.49\textwidth, height=0.314\textwidth, xmin=0, xmax=3.2, 
      axis x line=bottom, axis y line=left, legend pos=south east, legend style={cells={anchor=west}},
      ]
      \addplot+[very thick,blue!90!black, mark size=1pt,
      ] table[x index=0,y index=1] {dat/cell3d-j.dat};
      \legend{$j(\phi)$};
    \end{semilogyaxis}
\begin{scope}[xshift=8cm]
  \begin{semilogyaxis}[xlabel=$\phi$, width=.49\textwidth, height=0.314\textwidth, xmin=0, xmax=3.2, 
    axis x line=bottom, axis y line=left, legend pos=south east, legend style={cells={anchor=west}},
    ]
    \addplot+[very thick,blue!90!black, mark size=1pt,
    ] table[x index=0,y index=1] {dat/sl-j.dat};
    \legend{$j(\phi)$};
  \end{semilogyaxis}
\end{scope}
  \end{tikzpicture}
  \caption{Plot of the function $j$ for time step $t=\pi/4$. 
  {\em Left:} cell phantom, {\em right:} Shepp--Logan.
    \label{fig:j}}
\end{figure}

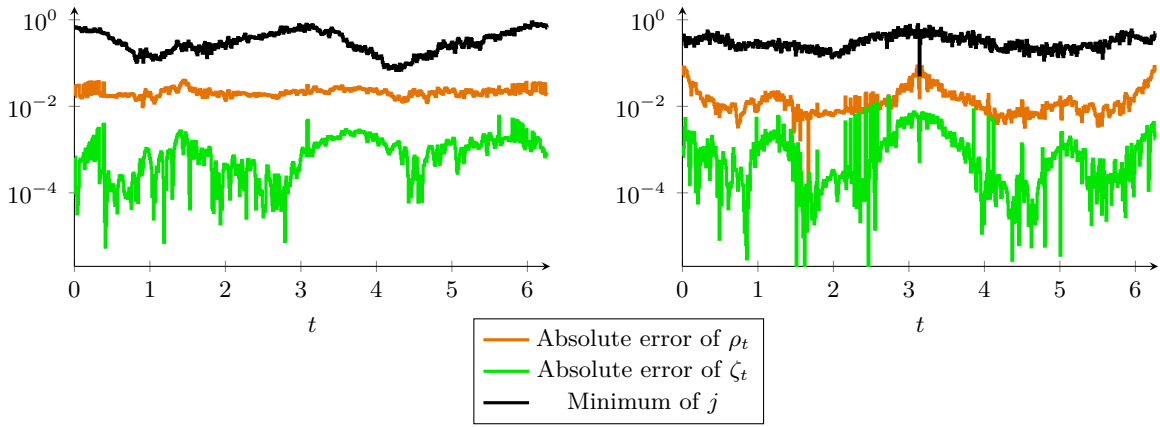
\begin{figure}[!ht]\centering
  \begin{tikzpicture}
    \begin{semilogyaxis}[xlabel=$t$, width=0.49\textwidth, height=0.314\textwidth, xmin=0, xmax=6.3, ymin=0.000002, ymax=2, 
      axis x line=bottom, axis y line=left,
      legend style={at={(1.45,-0.2)}},
      ]
      \addplot+[very thick,orange!90!black, const plot mark mid,mark=none
      ] table[x index=0,y index=1] {dat/cell3d-err.dat};
      \addplot+[very thick,green!90!black, const plot mark mid,mark=none
      ] table[x index=0,y index=3] {dat/cell3d-err.dat};
      \addplot+[very thick,black, const plot mark mid,mark=none
      ] table[x index=0,y index=4] {dat/cell3d-err.dat};
      \legend{Absolute error of $\rho_t$,
        Absolute error of $\zeta_t$, Minimum of $j$};
    \end{semilogyaxis}
\begin{scope}[xshift=8cm]
  \begin{semilogyaxis}[xlabel=$t$, width=0.49\textwidth,  height=0.314\textwidth, xmin=0, xmax=6.3, ymin=0.000002, ymax=2, 
    axis x line=bottom, axis y line=left,
    ]
    \addplot+[very thick,orange!90!black, const plot mark mid,mark=none
    ] table[x index=0,y index=1] {dat/sl-err.dat};
    \addplot+[very thick,green!90!black, const plot mark mid,mark=none
    ] table[x index=0,y index=3] {dat/sl-err.dat};
    \addplot+[very thick,black, const plot mark mid,mark=none
    ] table[x index=0,y index=4] {dat/sl-err.dat};
  \end{semilogyaxis}
\end{scope}
  \end{tikzpicture}
  \caption{
    Absolute error of the components \eqref{eq:omega-decomp} of the angular velocity $\zb\omega_t$
    and the minimum of the functional $j_t$, depending on $t$. 
    The reconstructed value for $\phi_t$ takes only values on the grid, 
    and in this case it is reconstructed exactly; note that the true value is also on the grid.
    {\em Left:} cell phantom, {\em right:} Shepp--Logan phantom.
    \label{fig:omega-error}}
\end{figure}

Inserting the reconstructed angular velocity,
we apply \autoref{alg:infR} to obtain the rotation matrices.
The reconstructions are denoted as 
$\mathbf R_t^{\mathrm{Pol}}$ with the polar decomposition \eqref{eq:Polar},
and $\mathbf R_t^{\mathrm{Cay}}$ with the Cayley transform \eqref{eq:Cay} as retraction.
The resulting error, measured in the Frobenius norm, is shown in \autoref{fig:R-error},
where we see that 
both retractions perform almost equally.

\begin{remark}[Sampling]
  Here, we assume that $\nu_t$ is given on a polar grid in order to easily compute the derivatives in \eqref{eq:gh},
  which we approximate numerically by central differences on the polar grid.
  However, the numerical reconstruction of $f$ for known rotations seems to be a little worse than with a uniform, rectangular grid for $\nu_t$ as considered in \cite{KirQueRitSchSet21}.
  However, the actual experiment takes measurements of the scattered wave $u_t$.
  Then $\nu_t$ is related to $u_t$ via a 2D Fourier transform in \eqref{eq:nu}.
  It seems natural that the images of $u_t$ are captured on a uniform grid.
  A canonical discretization of \eqref{eq:nu} is the fast Fourier transform, which gives an approximation of $\nu_t$ on a uniform grid, cf.\ \cite{BeiQue22}.
  Nevertheless, the nonuniform fast Fourier transform \cite[Section~7]{PlPoStTa18} can be used to evaluate $\nu_t$ accurately on any set such as a polar grid.
\end{remark}

\paragraph{Combination of the infinitesimal with the direct common circle method}

The error of the reconstruction in \autoref{fig:R-error} based on the infinitesimal method grows with the time $t$.
This behavior is quite expected since we make a small error in each time step and the errors accumulate.
In order to get a better reconstruction, we use the direct common circle method in \autoref{alg:cc}.
We minimize the functional $\mathcal E_{s,t}$, given in \eqref{eq:dist}, over $\SO$ iteratively with
the Nelder--Mead downhill simplex method \cite{LagReeWriWri98} implemented in Matlab's \texttt{fminsearch} routine, which does not require derivatives.
As starting solution, where we insert the Euler angles of $\mathbf R_t^{\mathrm{Cay}}$ reconstructed with the infinitesimal method as above.
The evaluation of $\nu_t$, which is sampled on a polar grid, 
at the curves $\bgam^{\varphi,\theta}$ utilizes cubic spline interpolation.
The error of the reconstruction with this combined approach is shown in \autoref{fig:R-error}.
We see that the reconstruction greatly benefits from the combined approach.

Furthermore, we have noticed that taking a random starting solution for the optimization of \eqref{eq:dist} often yields very bad results,
since $\mathcal E_{s,t}$ might have multiple local minima.
A possible approach would consist in evaluating $\mathcal E_{s,t}$ on a grid in $\R^3$ and taking the minimum or by using multiple random starting solutions.
However, this seems unnecessary, since we can rely on the good starting solution obtained with the infinitesimal method.

\newcommand{\ymin}{3e-7}
\newcommand{\ymax}{.1}
\begin{figure}[!ht]\centering
\begin{tikzpicture}
    \begin{semilogyaxis}[xlabel=$t$, width=0.49\textwidth, height=0.343\textwidth, xmin=0, xmax=6.3, ymin=\ymin, ymax=\ymax, 
      legend style={at={(1.45,-0.2)}},
      yticklabel style={
        /pgf/number format/fixed,
        /pgf/number format/precision=4
      },
      scaled y ticks=false
      ]
      \addplot+[blue!90!black, mark=none, very thick,
      ] table[x index=0,y index=5] {dat/cell3d-err.dat};
      \addplot+[green!90!black, mark=none, very thick, dashed,
      ] table[x index=0,y index=6] {dat/cell3d-err.dat};
      \addplot+[red!90!black, mark=none, very thick, 
      ] table[x index=0,y index=7] {dat/cell3d-err.dat};
      \legend{$\norm{R_t - \mathbf R_t^{\mathrm{Pol}}}_{{F}} / \norm{R_t}_{{F}} $,
        $\norm{\smash{R_t - \mathbf R_t^{\mathrm{Cay}}}}_{{F}}/ \norm{R_t}_{{F}}$,
        $\norm{R_t - \mathbf R_t^{\mathrm{CC}}}_{{F}}/ \norm{R_t}_{{F}}$};
    \end{semilogyaxis}
	\begin{scope}[xshift=8cm]
    \begin{semilogyaxis}[xlabel=$t$, width=0.49\textwidth, height=0.343\textwidth, xmin=0, xmax=6.3, ymin = \ymin, ymax=\ymax, 
    yticklabel style={
      /pgf/number format/fixed,
      /pgf/number format/precision=4
    },
    scaled y ticks=false
    ]
    \addplot+[blue!90!black, mark=none, very thick,
    ] table[x index=0,y index=5] {dat/sl-err.dat};
    \addplot+[green!90!black, mark=none, very thick, dashed,
    ] table[x index=0,y index=6] {dat/sl-err.dat};
    \addplot+[red!90!black, mark=none, very thick, 
    ] table[x index=0,y index=7] {dat/sl-err.dat};
  \end{semilogyaxis}
  \end{scope}
  \end{tikzpicture}
  \caption{
    Relative error of the reconstructed rotation matrices $\mathbf R_t^{\mathrm{Pol}}$ and $\mathbf R_t^{\mathrm{Cay}}$ using Euler's method \eqref{eq:Euler-proj}
    with the polar decomposition \eqref{eq:Polar}
    or the Cayley transform \eqref{eq:Cay}, respectively.
    Furthermore, $R_t^{\mathrm{CC}}$ refers to the rotation matrix reconstructed with the minimization of \eqref{eq:dist} to find the common circles,
    where the starting solution of the optimization was in each step computed with the infinitesimal method and the Cayley transform as above.
    {\em Left:} cell phantom, {\em right:} Shepp--Logan phantom.
    \label{fig:R-error}}
\end{figure}
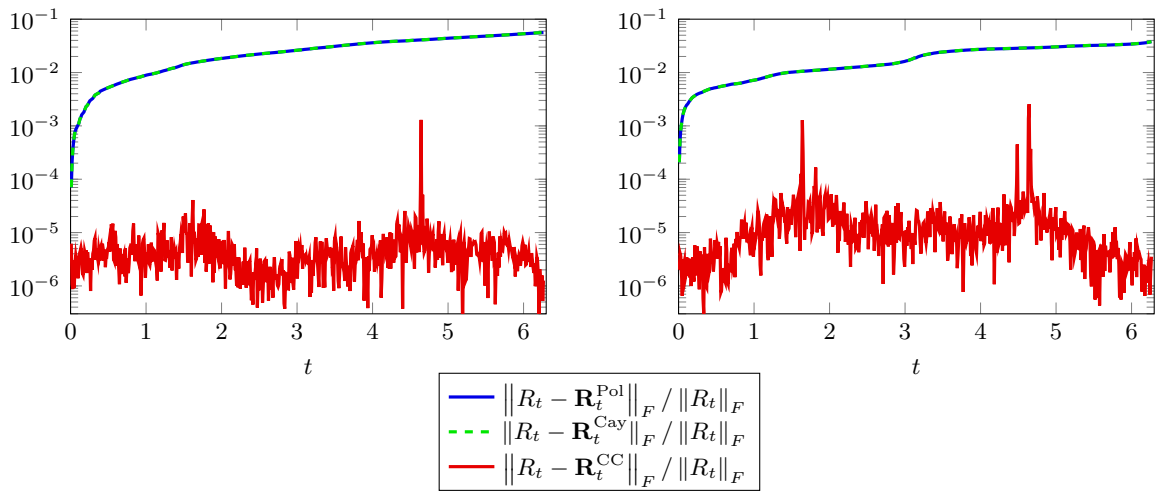

\paragraph{Moving rotation axis}

In our next simulation, we consider the time-dependent rotation axis
$
\bn(t)
=
(
\sqrt{1-a^2} \cos(b\,\sin(t/2)),
\sqrt{1-a^2} \sin(b\,\sin(t/2)),
a)^\top
\in\S^2
$
for $a=0.28$ and $b=0.5$.
The obtained error is shown in \autoref{fig:R-error-movaxis}.
Overall, the results are similar to the ones for the constant rotation axis.
However,  
there is a larger error  around $t\approx0$,
which might be explained by the fact that for a small rotation the respective hemispheres and thus also the data $\nu_t$ and $\nu_0$ are very close together, 
which makes detecting the common circles harder.
This could be circumvented by applying the common circles method to rotations that are farther apart.
A similar observation was made that the common lines in context of the ray transform
also become harder to detect in case of very small rotations where the infinitesimal method suits better, see \cite{ElbRitSchSchm20}.
\begin{figure}[!ht]\centering
    \begin{tikzpicture}
      \begin{semilogyaxis}[xlabel=$t$, width=0.49\textwidth, height=0.343\textwidth, xmin=0, xmax=6.3, ymin=\ymin, ymax=\ymax, 
        legend style={at={(1.45,-0.2)}},
        yticklabel style={
          /pgf/number format/fixed,
          /pgf/number format/precision=4
        },
        scaled y ticks=false
        ]
        \addplot+[blue!90!black, mark=none, very thick,
        ] table[x index=0,y index=5] {dat/cell3d-movaxis-err.dat};
        \addplot+[green!90!black, mark=none, very thick, dashed,
        ] table[x index=0,y index=6] {dat/cell3d-movaxis-err.dat};
        \addplot+[red!90!black, mark=none, very thick, 
        ] table[x index=0,y index=7] {dat/cell3d-movaxis-err.dat};
        \legend{$\norm{R_t - \mathbf R_t^{\mathrm{Pol}}}_{{F}} / \norm{R_t}_{{F}} $,
          $\norm{\smash{R_t - \mathbf R_t^{\mathrm{Cay}}}}_{{F}}/ \norm{R_t}_{{F}}$,
          $\norm{R_t - \mathbf R_t^{\mathrm{CC}}}_{{F}}/ \norm{R_t}_{{F}}$};
      \end{semilogyaxis}

	\begin{scope}[xshift=8cm]
      \begin{semilogyaxis}[xlabel=$t$, width=0.49\textwidth, height=0.343\textwidth, xmin=0, xmax=6.3, ymin=\ymin, ymax=\ymax, 
        yticklabel style={
          /pgf/number format/fixed,
          /pgf/number format/precision=4
        },
        scaled y ticks=false
        ]
        \addplot+[blue!90!black, mark=none, very thick,
        ] table[x index=0,y index=5] {dat/sl-movaxis-err.dat};
        \addplot+[green!90!black, mark=none, very thick, dashed,
        ] table[x index=0,y index=6] {dat/sl-movaxis-err.dat};
        \addplot+[red!90!black, mark=none, very thick, 
        ] table[x index=0,y index=7] {dat/sl-movaxis-err.dat};
      \end{semilogyaxis}
	\end{scope}
    \end{tikzpicture}
  \caption{
    Relative error of the rotation matrix $R_t$, reconstructed using the same setup and methods as in \autoref{fig:R-error}, but with a moving rotation axis.
    \label{fig:R-error-movaxis}}
\end{figure}
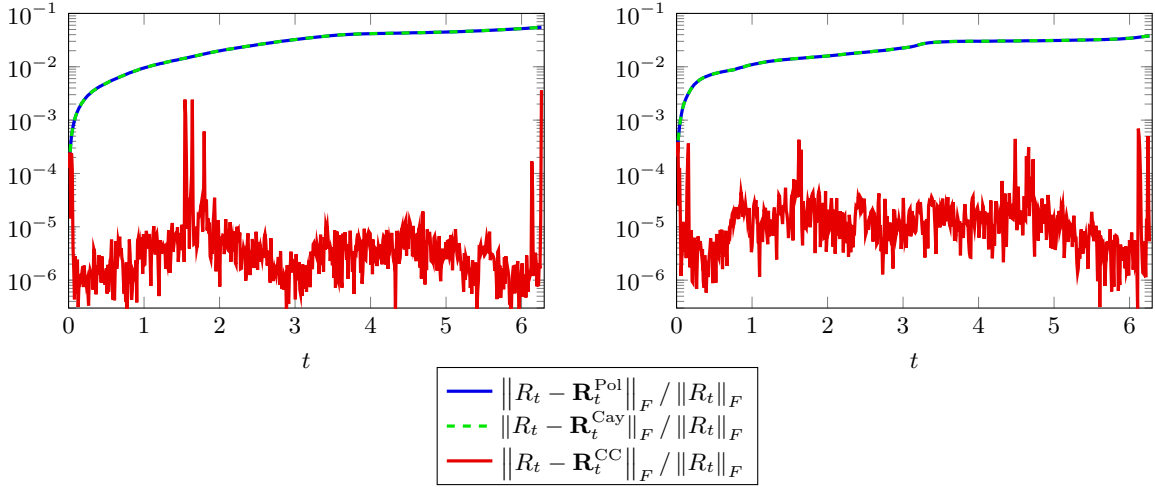

\paragraph{Reconstruction of the translation}
Now, we consider the case that the object also moves according to
the translation
$
\bd_t
=
4 (\sin t, \sin t, \sin t)^\top$,
$
t\in[0,2\pi],
$
and the rotation is the same as in the previous example with the moving axis.
We first reconstruct the rotations with the same methods as above.
Afterwards, we recover the translations $\bd_t$ by \autoref{alg:translation}.
The reconstruction error is shown in \autoref{fig:translation},
where we see that the translation is reconstructed quite reliably.
The error of the rotation is larger than in the case without translation, 
but it is still on an acceptable level. 
This is because we could only use the real-valued $\nu_t$ for reconstructing the rotations as the translations do not vanish, cf.\ \autoref{rem:only-rotation}.
Especially for large translations, we have noted in the simulations that the unwrapping in \eqref{eq:common-circle:unwrap} does not always yield good results because of the inexactness of the data.
This could be mitigated by combining it with a nonlinear optimization directly applied to \eqref{eq:common-circle:translation} and \eqref{eq:common-circle:translationDual}.

\begin{figure}[!ht]\centering
  \renewcommand{\ymin}{8e-6}
  \renewcommand{\ymax}{.2}
  \begin{tikzpicture}
    \begin{semilogyaxis}[xlabel=$t$, width=.49\textwidth, height=0.34\textwidth, xmin=0, xmax=6.3, ymin=\ymin, ymax=\ymax, 
      legend style={at={(1.45,-0.2)}},
      yticklabel style={
        /pgf/number format/fixed,
        /pgf/number format/precision=4
      },
      scaled y ticks=false
      ]
      \addplot+[blue!90!black, mark=none, very thick,
      ] table[x index=0,y index=5] {dat/cell3d-translation-err.dat};
      \addplot+[green!90!black, mark=none, very thick, dashed,
      ] table[x index=0,y index=6] {dat/cell3d-translation-err.dat};
      \addplot+[red!90!black, mark=none, thick, 
      ] table[x index=0,y index=7] {dat/cell3d-translation-err.dat};
      \addplot+[yellow!50!black, mark=none, thick, 
      ] table[x index=0,y index=9] {dat/cell3d-translation-err.dat};
      \legend{$\norm{R_t - \mathbf R_t^{\mathrm{Pol}}}_{{F}} / \norm{R_t}_{{F}} $,
        $\norm{\smash{R_t - \mathbf R_t^{\mathrm{Cay}}}}_{{F}}/ \norm{R_t}_{{F}}$,
        $\norm{R_t - \mathbf R_t^{\mathrm{CC}}}_{{F}}/ \norm{R_t}_{{F}}$,
        $\norm{\bd_t - \mathtt d_t}$};
    \end{semilogyaxis}
    
    \begin{scope}[xshift=8cm]
      \begin{semilogyaxis}[xlabel=$t$, width=.49\textwidth, height=0.34\textwidth, xmin=0, xmax=6.3, ymin=\ymin, ymax=\ymax, 
        legend style={at={(0.65,-0.25)}},
        yticklabel style={
          /pgf/number format/fixed,
          /pgf/number format/precision=4
        },
        scaled y ticks=false
        ]
        \addplot+[blue!90!black, mark=none, very thick,
        ] table[x index=0,y index=5] {dat/sl-translation-err.dat};
        \addplot+[green!90!black, mark=none, very thick, dashed,
        ] table[x index=0,y index=6] {dat/sl-translation-err.dat};
        \addplot+[red!90!black, mark=none, thick, 
        ] table[x index=0,y index=7] {dat/sl-translation-err.dat};
        \addplot+[yellow!50!black, mark=none, thick, 
        ] table[x index=0,y index=9] {dat/sl-translation-err.dat};
      \end{semilogyaxis}
    \end{scope}
  \end{tikzpicture}
  
  \caption{
    Error of the reconstructed rotation matrix and reconstructed translation $\mathtt d_t$,
    for the case of a non-zero translation.
    Note that $\norm{\bd_t}$ varies between 0 and approximately $6.9$, so it does not make sense to compute relative errors.
    {\em Left:} cell phantom,
    {\em Right:} Shepp--Logan phantom.
    \label{fig:translation}}
\end{figure}
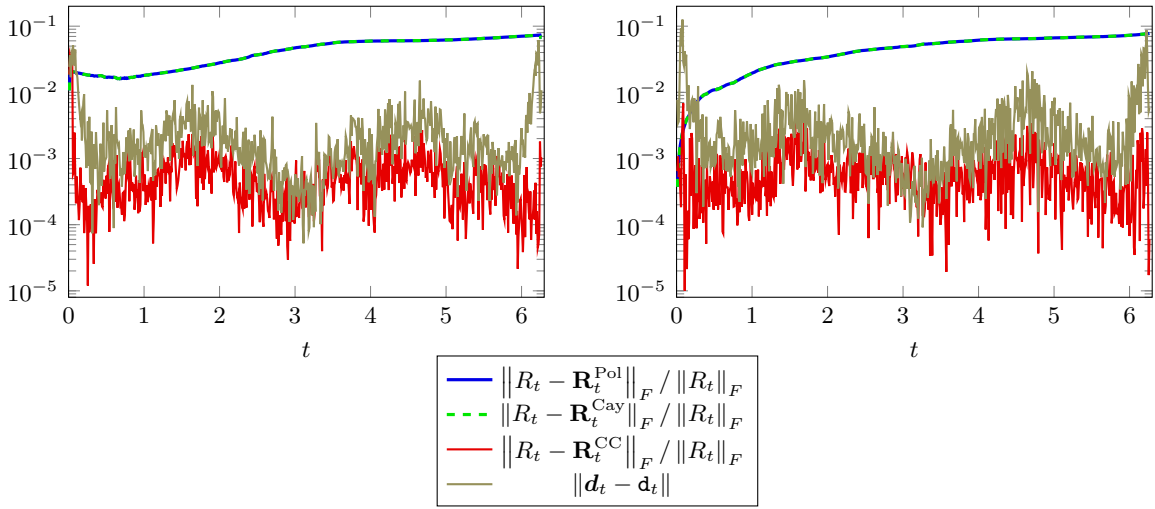

Finally, we show in \autoref{fig:translation-rec} the reconstructed images of the scattering potential $f$, where we have first computed the rotations and translations with the above combined common circle method for the moving rotation axis.
In the second part, i.e., the image reconstruction with known motion,
we use the nonuniform Fourier reconstruction technique from \cite{KirQueRitSchSet21}.
For the image reconstruction, we evaluate $\mu_t$ on a uniform grid instead of the polar grid used for the common circle method,
since the reconstruction of the image $f$ for a polar grid shows an inferior quality.
This observation is consistent with numerical evidence in \cite{FeKuPo06},
which showed that an approximate inversion of discrete Fourier transforms on a two-dimensional polar grid often shows large errors even for a very large number of sampling points.

\begin{figure}[!htp]\centering
\begin{subfigure}{.49\textwidth}
  \begin{tikzpicture}
    \begin{axis}[
      width=.77\textwidth, 
      height=.616\textwidth, 
      enlargelimits=false,
      scale only axis,
      axis on top,
      colorbar,colorbar style={
        width=.15cm, xshift=-0.6em
      },
      ]
      \addplot[point meta min=-0.09,point meta max=1.09] graphics [
      xmin=-28.28, xmax=28.28,
      ymin=-28.28, ymax=28.28,
      ] {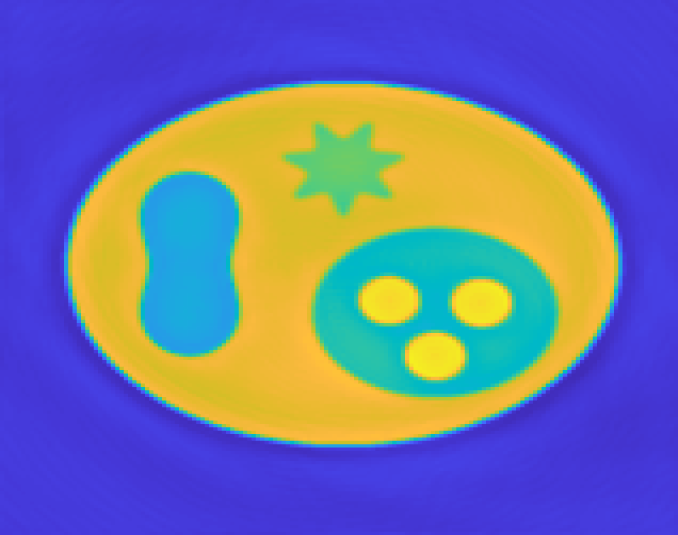};
    \end{axis}
  \end{tikzpicture}
  \caption{Cell phantom (PSNR 32.21, SSIM 0.754)}
\end{subfigure}
\hfil
\begin{subfigure}{.49\textwidth}
  \begin{tikzpicture}
    \begin{axis}[
      width=.77\textwidth, 
      height=.616\textwidth, 
      enlargelimits=false,
      scale only axis,
      axis on top,
      colorbar,colorbar style={
        width=.15cm, xshift=-0.6em
      },
      ]
      \addplot[point meta min=-0.09,point meta max=1.09] graphics [
      xmin=-28.28, xmax=28.28,
      ymin=-28.28, ymax=28.28,
      ] {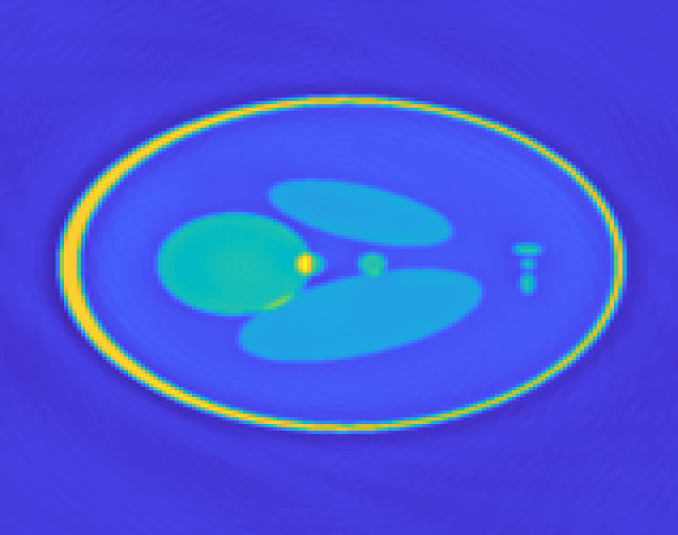};
    \end{axis}
  \end{tikzpicture}
  \caption{Shepp--Logan (PSNR 30.85, SSIM 0.772)}
\end{subfigure}
\caption{
  Slice plots of the reconstructed scattering potential $f$, where the rotation and translation was estimated with the common circle method as in \autoref{fig:translation}.
  The image quality is assessed via the peak signal-to-noise ratio (PSNR) and the structural similarity index (SSIM).
  \label{fig:translation-rec}}
\end{figure}

\paragraph{Computational time}

The numerical simulations were performed with Matlab on a standard PC with an 8-core Intel i7-10700 processor and 32 GB of memory. We utilized the NFFT software package \cite{KeKuPo09,nfft3} for the Fourier transforms. The reconstruction of the motion parameters for all 640 time steps as in \autoref{fig:translation} took about 40 seconds.
The image reconstruction in \autoref{fig:translation-rec} took about 90 seconds.

\section{Conclusions}
\label{sec:conclusions}

In this paper, we have considered the reconstruction of the motion of an object in diffraction tomography.
For the reconstruction of the rotation, we have presented a common circle method and its infinitesimal version.
While the former method usually produced more accurate results,
it benefits from using a starting solution with the computationally faster infinitesimal approach.
Furthermore, we have shown that also the translation of the object can be uniquely recovered from the diffraction data.
For this, we have required that the scattering potential is real-valued.
We note that, in contrast to projection images corresponding to the ray transform,
also the position and orientation of the object in direction of the incident wave can be detected here.

Future research will focus on the real-world application related to optical diffraction tomography with acoustical tweezers.
Furthermore, we intend to incorporate phase retrieval methods for the motion detection
since often only the intensities of the field $\us+\ui$ can be measured.

\section*{Acknowledgments}
Funding by the DFG under the SFB ``Tomography Across the Scales'' (STE 571/19-1, project number: 495365311) is gratefully acknowledged. 
Moreover, PE and OS are supported by the Austrian Science Fund (FWF),
with SFB F68 ``Tomography Across the Scales'', project F6804-N36 and F6807-N36.
The financial support by the Austrian Federal Ministry for Digital and Economic
Affairs, the National Foundation for Research, Technology and Development and the Christian Doppler
Research Association is gratefully acknowledged.
This research was funded in whole, or in part, by the Austrian Science
Fund (FWF) P 34981. For the purpose of open access, the authors have applied
a CC BY public copyright license to any Authors Accepted Manuscript version arising
from this submission.

\bibliographystyle{abbrv}
\bibliography{nocsc}

\begin{thebibliography}{10}

\bibitem{ABS2008}
P.-A. Absil, R.~Mahony, and R.~Sepulchre.
\newblock {\em Optimization Algorithms on Matrix Manifolds}.
\newblock Princeton University Press, 2008.

\bibitem{BeiQue22}
R.~Beinert and M.~Quellmalz.
\newblock Total variation-based reconstruction and phase retrieval for
  diffraction tomography.
\newblock {\em SIAM Journal on Imaging Sciences}, 15(3):1373--1399, 2022.

\bibitem{BenBarSin20}
T.~Bendory, A.~Bartesaghi, and A.~Singer.
\newblock Single-particle cryo-electron microscopy: Mathematical theory,
  computational challenges, and opportunities.
\newblock {\em IEEE Signal Processing Magazine}, 37(2):58--76, 2020.

\bibitem{BorTeg11}
G.~Bortel and M.~Tegze.
\newblock Common arc method for diffraction pattern orientation.
\newblock {\em Acta Cryst.}, A67:533--543, 2011.

\bibitem{ColKre13}
D.~Colton and R.~Kress.
\newblock {\em Inverse Acoustic and Electromagnetic Scattering Theory}.
\newblock Number~93 in Applied Mathematical Sciences. Springer, Berlin, 3rd
  edition, 2013.

\bibitem{Dev82}
A.~Devaney.
\newblock A filtered backpropagation algorithm for diffraction tomography.
\newblock {\em Ultrasonic Imaging}, 4(4):336--350, 1982.

\bibitem{dholakia2020comparing}
K.~Dholakia, B.~W. Drinkwater, and M.~Ritsch-Marte.
\newblock Comparing acoustic and optical forces for biomedical research.
\newblock {\em Nature Reviews Physics}, 2(9):480--491, 2020.

\bibitem{ElbRitSchSchm20}
P.~Elbau, M.~Ritsch-Marte, O.~Scherzer, and D.~Schmutz.
\newblock Motion reconstruction for optical tomography of trapped objects.
\newblock {\em Inverse Problems}, 36(4):044004, 2020.

\bibitem{FauKirQueSchSet21}
F.~Faucher, C.~Kirisits, M.~Quellmalz, O.~Scherzer, and E.~Setterqvist.
\newblock Diffraction tomography, {F}ourier reconstruction, and full waveform
  inversion.
\newblock In K.~Chen, C.-B. Sch{\"o}nlieb, X.-C. Tai, and L.~Younes, editors,
  {\em Handbook of Mathematical Models and Algorithms in Computer Vision and
  Imaging}, pages 273--312. Springer, Cham, 2023.

\bibitem{FeKuPo06}
M.~Fenn, S.~Kunis, and D.~Potts.
\newblock On the computation of the polar {FFT}.
\newblock {\em Applied and Computational Harmonic Analysis}, 22:257--263, 2007.

\bibitem{HaiLubWan06}
E.~Hairer, C.~Lubich, and G.~Wanner.
\newblock {\em Geometric Numerical Integration}, volume~31 of {\em Springer
  Series in Computational Mathematics}.
\newblock Springer, Berlin, 2nd edition, 2006.

\bibitem{HHNPSS2020}
M.~Hasannasab, J.~Hertrich, S.~Neumayer, G.~Plonka, S.~Setzer, and G.~Steidl.
\newblock {{Parseval}} proximal neural networks.
\newblock {\em Journal of Fourier Analysis and Applications}, 26(59):1--31,
  2020.

\bibitem{HulSzoHaj03}
G.~Huldt, A.~Szőke, and J.~Hajdu.
\newblock Diffraction imaging of single particles and biomolecules.
\newblock {\em J. Struct. Biol.}, 144(1):219--227, 2003.

\bibitem{Ito82}
K.~Itoh.
\newblock Analysis of the phase unwrapping problem.
\newblock {\em Applied Optics}, 21(14), 1982.

\bibitem{JonMarVol15}
P.~H. Jones, O.~M. Maragò, and G.~Volpe.
\newblock {\em Optical Tweezers.}
\newblock Cambridge University Press, Cambridge, 2015.

\bibitem{KakSla01}
A.~C. Kak and M.~Slaney.
\newblock {\em Principles of Computerized Tomographic Imaging}.
\newblock Number~33 in Classics in Applied Mathematics. Society for Industrial
  and Applied Mathematics (SIAM), Philadelphia, PA, 2001.
\newblock Reprint of the 1988 original.

\bibitem{Kam80}
Z.~Kam.
\newblock The reconstruction of structure from electron micrographs of randomly
  oriented particles.
\newblock {\em Journal of Theoretical Biology}, 82(1):15--39, 1980.

\bibitem{nfft3}
J.~Keiner, S.~Kunis, and D.~Potts.
\newblock {NFFT 3.5, C subroutine library}.
\newblock \url{https://www.tu-chemnitz.de/~potts/nfft}.
\newblock Contributors: F.~Bartel, M.~Fenn, T.~G\"orner, M.~Kircheis, T.~Knopp,
  M.~Quellmalz, M.~Schmischke, T.~Volkmer, A.~Vollrath.

\bibitem{KeKuPo09}
J.~Keiner, S.~Kunis, and D.~Potts.
\newblock Using {NFFT3} - a software library for various nonequispaced fast
  {Fourier} transforms.
\newblock {\em {ACM} Transactions on Mathematical Software}, 36:Article
  19,1--30, 2009.

\bibitem{KetLam11}
J.~Ketola and L.~Lamberg.
\newblock An algorithm for recovering unknown projection orientations and
  shifts in 3-d tomography.
\newblock {\em Inverse Problems and Imaging}, 5(1):75--93, 2011.

\bibitem{KirQueRitSchSet21}
C.~Kirisits, M.~Quellmalz, M.~Ritsch-Marte, O.~Scherzer, E.~Setterqvist, and
  G.~Steidl.
\newblock Fourier reconstruction for diffraction tomography of an object
  rotated into arbitrary orientations.
\newblock {\em Inverse Problems}, 37(11):115002, 2021.

\bibitem{KurZic21}
P.~Kurlberg and G.~Zickert.
\newblock Formal uniqueness in {E}wald sphere corrected single particle
  analysis.
\newblock {\em ArXiv 2104.05371}, 2021.

\bibitem{KvaPreRit21}
M.~Kv{\aa}le~L{\o}vmo, B.~Pressl, G.~Thalhammer, and M.~Ritsch-Marte.
\newblock Controlled orientation and sustained rotation of biological samples
  in a sono-optical microfluidic device.
\newblock {\em Lab on a Chip}, 21(8):1563--1578, 2021.

\bibitem{LagReeWriWri98}
J.~C. Lagarias, J.~A. Reeds, M.~H. Wright, and P.~E. Wright.
\newblock Convergence properties of the {N}elder--{M}ead simplex method in low
  dimensions.
\newblock {\em SIAM Journal on Optimization}, 9(1):112--147, 2009.

\bibitem{LanLif81}
L.~D. Landau and E.~M. Lifshitz.
\newblock {\em Mechanics}, volume~1 of {\em Course of Theoretical Physics}.
\newblock Butterworth--Heinemann, Oxford, 3rd edition, 1981.

\bibitem{Moa02}
M.~Moakher.
\newblock Means and averaging in the group of rotations.
\newblock {\em SIAM Journal on Matrix Analysis and Applications}, 24(1):1–16,
  2002.

\bibitem{MueSchuGuc15_report}
P.~M{\"{u}}ller, M.~Sch{\"{u}}rmann, and J.~Guck.
\newblock The theory of diffraction tomography, 2015.
\newblock ArXiv 1507.00466v3.

\bibitem{NatWue01}
F.~Natterer and F.~W{\"{u}}bbeling.
\newblock {\em Mathematical Methods in Image Reconstruction}.
\newblock Number~5 in Monographs on Mathematical Modeling and Computation.
  SIAM, Philadelphia, PA, 2001.

\bibitem{PlPoStTa18}
G.~Plonka, D.~Potts, G.~Steidl, and M.~Tasche.
\newblock {\em Numerical Fourier Analysis}.
\newblock Applied and Numerical Harmonic Analysis. Birkhäuser, Cham, 2018.

\bibitem{Schm17}
D.~Schmutz.
\newblock Reconstruction of projection orientations in cryo-electron
  microscopy, 2017.
\newblock Master's thesis, University of Vienna.

\bibitem{ShaKilKhoLanSin20}
N.~Sharon, J.~Kileel, Y.~Khoo, B.~Landa, and A.~Singer.
\newblock Method of moments for 3d single particle ab initio modeling with
  non-uniform distribution of viewing angles.
\newblock {\em Inverse Problems}, 36(4):044003, 2020.

\bibitem{SinCoiSigCheShk10}
A.~Singer, R.~R. Coifman, F.~J. Sigworth, D.~W. Chester, and Y.~Shkolnisky.
\newblock Detecting consistent common lines in cryo-{EM} by voting.
\newblock {\em Journal of Structural Biology}, 169(3):312--322, 2010.

\bibitem{st97}
G.~Steidl.
\newblock A note on fast {F}ourier transforms for nonequispaced grids.
\newblock {\em Advances in Computational Mathematics}, 9(3-4):337--353, 1998.

\bibitem{ThaSteMeiHilBerRit11}
G.~Thalhammer, R.~Steiger, M.~Meinschad, M.~Hill, S.~Bernet, and
  M.~Ritsch-Marte.
\newblock Combined acoustic and optical trapping.
\newblock {\em Biomedical Optics Express}, 2(10):2859--2870, 2011.

\bibitem{Hee87}
M.~van Heel.
\newblock Angular reconstitution: A posteriori assignment of projection
  directions for 3d reconstruction.
\newblock {\em Ultramicroscopy}, 21(2):111--123, 1987.

\bibitem{Hee00}
M.~van Heel, B.~Gowen, R.~Matadeen, E.~V. Orlova, R.~Finn, T.~Pape, D.~Cohen,
  H.~Stark, R.~Schmidt, M.~Schatz, and A.~Patwardhan.
\newblock Single-particle electron cryo-microscopy: towards atomic resolution.
\newblock {\em Quarterly Reviews of Biophysics}, 33(4):307–369, 2000.

\bibitem{WanSinWen13}
L.~Wang, A.~Singer, and Z.~Wen.
\newblock Orientation determination of cryo-{EM} images using least unsquared
  deviations.
\newblock {\em {SIAM} Journal on Imaging Sciences}, 6(4):2450--2483, 2013.

\bibitem{Wol69}
E.~Wolf.
\newblock Three-dimensional structure determination of semi-transparent objects
  from holographic data.
\newblock {\em Optics Communications}, 1:153--156, 1969.

\end{thebibliography}

\appendix
\section{Proofs of \autoref{sec:comm_circ}}
\label{sec:proof}
\begin{proof}[of \autoref{th:YsYt_param}]
  By \eqref{eq:Ht-sphere}, a point $\bx\in\R^3$ is in the intersection $\mathcal H_s\cap\mathcal H_t$ if and only if it fulfills the equations
  \begin{equation}\label{eq:YsYt_param:conditions}
    \norm{\bx+k_0R_s\be^3}^2=k_0^2\quad\text{and}\quad\norm{\bx+k_0R_t\be^3}^2=k_0^2
  \end{equation}
  and the two inequalities
  \begin{equation}\label{eq:YsYt_param:inequalities}
    \langle\bx,R_s\be^3\rangle > -k_0\quad\text{and}\quad\langle\bx,R_t\be^3\rangle > -k_0.
  \end{equation}
  Taking their sum and their difference, the two equations \eqref{eq:YsYt_param:conditions} are seen to be equivalent to
  \[ \norm{\bx+\frac{k_0}2(R_s\be^3+R_t\be^3)}^2=\frac{k_0^2}4\norm{R_s\be^3+R_t\be^3}^2\quad\text{and}\quad\langle\bx,R_s\be^3-R_t\be^3\rangle=0, \]
  meaning that $\bx$ is on a circle with radius $ a_{s,t}$ around the point $- a_{s,t}\bv_{s,t}^1$ in the subspace $\mathcal V_{s,t}$ spanned by the vectors $\bv_{s,t}^1$ and $\bv_{s,t}^2$, so that we can write
  \[ \bx =  a_{s,t}(\cos(\beta)-1)\bv_{s,t}^1+ a_{s,t}\sin(\beta)\bv_{s,t}^2\quad\text{for some}\quad\beta\in(-\pi,\pi]. \]
  
  For such a point $\bx$, the two inequalities \eqref{eq:YsYt_param:inequalities} are equivalent and reduce to the condition
  \[ \frac{k_0}2(\cos(\beta)-1)(1+\langle R_s\be^3,R_t\be^3\rangle) > -k_0,\quad\text{that is},\quad\cos(\beta)>\frac{\langle R_s\be^3,R_t\be^3\rangle-1}{\langle R_s\be^3,R_t\be^3\rangle+1} \]
  for the variable $\beta$.
\end{proof}

\begin{proof}[of \autoref{th:gamma}]
  Since $R_s\bh$ is a parameterization of $\mathcal H_s$ and $\bh$ restricted to its first two components is just the identity, we can solve the relation $R_s\bh(\bk)=\bx$ for every $\bx\in\mathcal H_s$ by orthogonally projecting $R_s^\top\bx$ onto its first two components
  \[ \bk = P(R_s^\top\bx). \]
  Therefore we find directly from the representation \eqref{eq:YsYt_param} of $\mathcal H_s\cap\mathcal H_t$ that
  \begin{equation}\label{eq:gamma1}
    \bgam_{s,t}(\beta) =  a_{s,t}(\cos(\beta)-1)P(R_s^\top\bv_{s,t}^1)+ a_{s,t}\sin(\beta)P(R_s^\top\bv_{s,t}^2).
  \end{equation}
  Further we get for the projections of the basis vectors
  \begin{equation}\label{eq:Pv}
    \begin{split}
      P(R_s^\top\bv_{s,t}^1) &= \frac{P(\be^3+R_s^\top R_t\be^3)}{\norm{R_s\be^3+R_t\be^3}} 
      = \frac{\tilde a_{s,t}}{ a_{s,t}}\bw_{s,t}^1\quad\text{and}\\
      P(R_s^\top\bv_{s,t}^2)
      &=\frac{P(\be^3\times R_s^\top R_t\be^3)}{\norm{R_s\be^3\times R_t\be^3}}=\bw_{s,t}^2,
    \end{split}
  \end{equation}
  so that the equation \eqref{eq:gamma1} for $\bgam_{s,t}$ becomes \eqref{eq:gamma}.
\end{proof}

\begin{proof}[of \autoref{th:Euler}]
  We have that
  \[ R_s^\top R_t\be^3 = Q^{(3)}(\varphi)Q^{(2)}(\theta)Q^{(3)}(\psi)\be^3 
	= Q^{(3)}(\varphi)Q^{(2)}(\theta)\be^3 = \begin{pmatrix}\cos(\varphi)\sin(\theta)\\\sin(\varphi)\sin(\theta)\\\cos(\theta)\end{pmatrix}. \]
  Plugging this into $\bgam_{s,t}$ from \eqref{eq:gamma}, we find for the lengths of the semi-axes 
  \begin{align*}
    \tilde a_{s,t} &= \frac{k_0}2\norm{P(R_s^\top R_t\be^3)} = \frac{k_0}2\sin(\theta)\quad\text{and} \\
    a_{s,t} &= \frac{k_0}2\norm{R_s\be^3+R_t\be^3} = \frac{k_0}2\norm{\be^3+R_s^\top R_t\be^3} = \frac{k_0}2\sqrt{2+2\cos(\theta)} = k_0\cos(\tfrac\theta2);
  \end{align*}
  and for the directions of the semi-axes $\bw_{s,t}^1 = \begin{psmallmatrix}\cos(\varphi)\\\sin(\varphi)\end{psmallmatrix}$ and $\bw_{s,t}^2 = \begin{psmallmatrix}-\sin(\varphi)\\\cos(\varphi)\end{psmallmatrix}$.
\end{proof}

\begin{proof}[of \autoref{th:dualCircle}]
  The proof goes along the same lines as \autoref{th:YsYt_param}, \autoref{th:gamma}, and \autoref{th:Euler}.
  
  \begin{enumerate}
    \item
    By replacing $R_t\be^3$ by $-R_t\be^3$ in \autoref{th:YsYt_param}, 
		we directly get the parameterization $\zb{\sigma}^*_{s,t}$ of 
		$\mathcal H_s\cap(-\mathcal H_t)$ in the form of \eqref{eq:dualCircle}.
    \item
    Proceeding as in \autoref{th:gamma}, we find the curve $\zb{\gamma}^*_{s,t}$ by
    \[ \zb{\gamma}^*_{s,t}(\beta) = P(R_s^\top\zb{\sigma}^*_{s,t}(\beta)) 
		=  
		a^*_{s,t}(\cos(\beta)-1)P(R_s^\top\bv_{s,t}^3)- a^*_{s,t}\sin(\beta)P(R_s^\top\bv_{s,t}^2). \]
    With
    \begin{equation}\label{eq:Pv3}
      P(R_s^\top\bv_{s,t}^3) = \frac{P(\be^3-R_s^\top R_t\be^3)}{\norm{R_s\be^3-R_t\be^3}} = -\frac{k_0}2\frac{P(R_s^\top R_t\be^3)}{a^*_{s,t}} = -\frac{\tilde a_{s,t}}{a^*_{s,t}}\bw_{s,t}^1,
    \end{equation}
    this yields the equation \eqref{eq:dualEllipse} for $\zb{\gamma}^*_{s,t}$.
    \item
    Taking finally the expressions of $\tilde a_{s,t}$, $\bw_{s,t}^1$, and $\bw_{s,t}^2$ in terms of the Euler angles of $R_s^\top R_t$ from the proof of \autoref{th:Euler}, we obtain with
    \[
		a^*_{s,t}
		\coloneqq
		\frac{k_0}2\norm{\be^3-R_s^\top R_t\be^3} 
		= 
		\frac{k_0}2\sqrt{2-2\cos(\theta)} = k_0\sin(\tfrac\theta2) 
		\]
    the identity in \eqref{eq:dualEllipseEuler}.
  \end{enumerate}
\end{proof}

\begin{proof}[of \autoref{th:commonAxis}]
(i) Since $R_s^\top R_t\be^3=\be^3$, the rotation $R_s^\top R_t$ has the rotation axis $\be^3$ and is therefore of the form $R_s^\top R_t=Q^{(3)}(\alpha)$ for some $\alpha\in\R/(2\pi\Z)$. Then, we see from the definition \eqref{eq:h} of $\bh$ that we have for all $\bk\in\mathcal B_{k_0}^2$
\begin{equation}\label{eq:commonAxisRel}
  R_s^\top R_t\bh(\bk) = Q^{(3)}(\alpha)\bh(\bk) = \bh(\mathrm{Q}(\alpha)\bk)
\end{equation}
and therefore, according to \eqref{eq:nu-F},
\[ \nu_s(\mathrm{Q}(\alpha)\bk) = \abs{\ktran[f](R_s\bh(\mathrm{Q}(\alpha)\bk))}^2 = \abs{\ktran[f](R_t\bh(\bk))}^2 = \nu_t(\bk). \]

(ii) Since $Q^{(2)}(\pi)R_s^\top R_t\be^3=-Q^{(2)}(\pi)\be^3=\be^3$, the rotation $Q^{(2)}(\pi)R_s^\top R_t$ has the rotation axis~$\be^3$ and we therefore have $R_s^\top R_t=Q^{(2)}(\pi)Q^{(3)}(\alpha)$ for some $\alpha\in\R/(2\pi\Z)$. Then, we see from the definition \eqref{eq:h} of $\bh$ that we have for all $\bk\in\mathcal B_{k_0}^2$
\begin{equation}\label{eq:commonAxisReflectionRel}
  R_s^\top R_t\bh(\bk) = Q^{(2)}(\pi)Q^{(3)}(\alpha)\bh(\bk) = Q^{(2)}(\pi)\bh(\mathrm{Q}(\alpha)\bk) = -\bh(\mathrm{SQ}(\alpha)\bk)
\end{equation}
and therefore, according to \eqref{eq:nu-F} and \eqref{eq:symmetryNu},
\begin{equation*} \nu_s(\mathrm{SQ}(\alpha)\bk) = \abs{\ktran[f](R_s\bh(S\mathrm{Q}(\alpha)\bk))}^2 = \abs{\ktran[f](-R_t\bh(\bk))}^2 = \abs{\ktran[f](R_t\bh(\bk))}^2 = \nu_t(\bk).  
\end{equation*}
\end{proof}

\begin{proof}[of \autoref{th:commonCircle}]
According to \autoref{th:commonAxis}, the fact that neither 
$\nu_s(\mathrm{Q}(\alpha)\bk)=\nu_t(\bk)$ nor $\nu_s(\mathrm{SQ}(\alpha)\bk)=\nu_t(\bk)$ 
holds for all $\bk\in\mathcal B_{k_0}^2$ for any choice of parameter $\alpha\in\R/(2\pi\Z)$ 
excludes the cases where $R_s\be^3=\pm R_t\be^3$ and we can find the elliptic arcs $\bgam_{s,t}$ and $\bgam_{t,s}$ 
and the dual arcs $\zb{\gamma}^*_{s,t}$ and $\zb{\gamma}^*_{t,s}$ as in \autoref{th:gamma} and \autoref{th:dualCircle}.

We parameterize the matrix $R_s^\top R_t$ in Euler angles 
$(\tilde\varphi,\tilde\theta,\tilde\psi)\in(\R/(2\pi\Z))\times[0,\pi]\times(\R/(2\pi\Z))$
as in \eqref{eq:euler}.
Then, the representation of the transposed matrix $R_t^\top R_s$ in Euler angles is given by
\begin{equation}\label{eq:RtRsEuler}
R_t^\top R_s = (R_s^\top R_t)^\top = Q^{(3)}(-\tilde\psi)Q^{(2)}(-\tilde\theta)Q^{(3)}(-\tilde\varphi) = Q^{(3)}(\pi-\tilde\psi)Q^{(2)}(\tilde\theta)Q^{(3)}(\pi-\tilde\varphi),
\end{equation}
where we used the identity $Q^{(2)}(-\tilde\theta)=Q^{(3)}(\pi)Q^{(2)}(\tilde\theta)Q^{(3)}(\pi)$ 
to shift the angles for $R_t^\top R_s$ into the chosen area of definition.
By \eqref{eq:gammaEuler}, we see that $\bgam^{\tilde\varphi,\tilde\theta}=\bgam_{s,t}$ 
and $\bgam^{\pi-\tilde\psi,\tilde\theta}=\bgam_{t,s}$,
and by \eqref{eq:dualEllipseEuler} that $\zb{\gamma}^{*,\tilde\varphi,\tilde\theta}=\zb{\gamma}^*_{s,t}$ 
and 
$\zb{\gamma}^{*,\pi-\tilde\psi,\tilde\theta}=\zb{\gamma}^*_{t,s}$.
Since the curves $\bgam_{s,t}$ and $\bgam_{t,s}$ as well as $\zb{\gamma}^*_{s,t}$ 
and $\zb{\gamma}^*_{t,s}$ fulfill 
by construction the equations \eqref{eq:commonCircleIdentity} and \eqref{eq:commonCircleIdentityDual}, 
and $\bgam^{\varphi,\theta}$, $\bgam^{\pi-\psi,\theta}$, $\bgam^{*,\varphi,\theta}$, and $\bgam^{*,\pi-\psi,\theta}$ are by assumption the only elliptic arcs of this form fulfilling \eqref{eq:nuGamma} and \eqref{eq:nuGammaDual}, 
we have that $\varphi=\tilde\varphi$, $\theta=\tilde\theta$, and $\psi=\tilde\psi$, which implies \eqref{eq:commonCircle}.
We note that the two pairs of curves can in general not be interchanged because in \eqref{eq:nuGamma} the curve $\bgam^{\varphi,\theta}(\cdot)$ is traversed conter-clockwise while $\bgam^{\pi-\psi,\theta}(-\cdot)$ is traversed clockwise,
whereas in \eqref{eq:nuGammaDual} both $\zb{\gamma}^{*,\varphi,\theta}$ and $\zb{\gamma}^{*,\pi-\psi,\theta}$ are traversed counter-clockwise.
\end{proof}

\section{Proof of \autoref{th:infRel}}
\label{sec:proof_1}
\begin{proof}
We define the continuously differentiable function $H\colon[0,T]\times\mathcal B_{k_0}^2\to\R^3$ by $H(t,\bk)\coloneqq R_t\bh(\bk)$ with $\bh$ being the parameterization from \eqref{eq:h}. 
By the definition~\eqref{eq:omega} of the angular velocity $\zb\omega_t$, we have $R_t' \bk = R_t (\zb\omega_t\times \bk)$.
Then 
for the partial derivative of $H$ with respect to the first argument $t$ at the point $\bk=r\zb\phi_t$ reads
\[ \partial_tH(t,r\zb\phi_t) 
=  R_t' \,\bh(r\zb\phi_t)
= R_t\big(\zb\omega_t\times\bh(r\zb\phi_t)\big) = R_t\begin{pmatrix}\omega_{t,2}h_3(r\zb\phi_t)-r{\omega_{t,3}}\phi_{t,2}\\-\omega_{t,1}h_3(r\zb\phi_t)+r{\omega_{t,3}}\phi_{t,1}\\r(\omega_{t,1}\phi_{t,2}-\omega_{t,2}\phi_{t,1})\end{pmatrix}. \]
Inserting the expression \eqref{eq:omega-decomp} of $\zb\omega_t$ in cylindrical coordinates and using that, according to \eqref{eq:h}, $h_3(r\zb\phi_t)=\sqrt{k_0^2-r^2}-k_0$, this becomes
\begin{equation} \label{eq:dtH}
\partial_tH(t,r\zb\phi_t) = \left(\rho_t\left(k_0-\sqrt{k_0^2-r^2}\right)+r\zeta_t\right)R_t\begin{pmatrix}-\phi_{t,2}\\\phi_{t,1}\\0\end{pmatrix}. \end{equation}
Denoting by $DH$ the Jacobi matrix of~$H$ with respect to $\bk$, we find with the chain rule
\begin{equation} \label{eq:dkH} DH(t,r\zb\phi_t)\begin{pmatrix}-\phi_{t,2}\\\phi_{t,1}\end{pmatrix} = R_t\begin{pmatrix}1&0\\0&1\\\frac{-r\phi_{t,1}}{\sqrt{k_0^2-r^2}}&\frac{-r\phi_{t,2}}{\sqrt{k_0^2-r^2}}\end{pmatrix}\begin{pmatrix}-\phi_{t,2}\\\phi_{t,1}\end{pmatrix} = R_t\begin{pmatrix}-\phi_{t,2}\\\phi_{t,1}\\0\end{pmatrix}. \end{equation}
Comparing \eqref{eq:dtH} and \eqref{eq:dkH}, we have that
\begin{equation} \label{eq:dH} \partial_tH(t,r\zb\phi_t)=\left(\rho_t\left(k_0-\sqrt{k_0^2-r^2}\right)+r\zeta_t\right)DH(t,r\zb\phi_t)\begin{pmatrix}-\phi_{t,2}\\\phi_{t,1}\end{pmatrix}. 
\end{equation}
Recalling the definition $\nu_t(\bk)=\abs{\ktran[f](H(t,\bk))}^2$,
we have again by the chain rule
\begin{align*}
\partial_t \nu_t(\bk)
&=
2 \operatorname{Re}\left( \ktran[f](H(t,\bk)) \right)
\inn{\nabla\ktran[f](H(t,\bk)), \partial_t H(t,\bk)}
\quad \text{ and}
\\
\inn{\nabla \nu_t(\bk), \begin{psmallmatrix} -\phi_{t,2}\\\phi_{t,1}\end{psmallmatrix}}
&=
2 \operatorname{Re}\left( \ktran[f](H(t,\bk)) \right)
\inn{\nabla\ktran[f](H(t,\bk)), D H(t,\bk) \begin{psmallmatrix} -\phi_{t,2}\\\phi_{t,1}\end{psmallmatrix}} ,
\end{align*}
where 
$\nabla \ktran[f](\by)$ denotes the gradient of $\ktran[f](\by)$ with respect to $\by\in\R^3$.
Inserting $\bk = r\bphi_t$ and using \eqref{eq:dH} yields the assertion.
\end{proof}

\section{Proofs of \autoref{sec:translations}} \label{sec:proof_2}
\begin{proof}[of \autoref{th:common-circle:translation}]
\begin{enumerate}
\item
Since the elliptic arcs $\bgam_{s,t}$ and $\bgam_{t,s}$ have by construction the symmetry \eqref{eq:commonCircleIdentity}, 
we obtain for all $\beta\in J_{s,t}$ that
\begin{align*}
\mu_s(\bgam_{s,t}(\beta)) &= \ktran [f]\left(R_s\bh(\bgam_{s,t}(\beta)) \right)\e^{-\i\inner{R_s\bd_s}{R_s\bh(\bgam_{s,t}(\beta))}} \\
&= \ktran{[f]}\left(R_t\bh(\bgam_{t,s}(-\beta)) \right)\e^{-\i\inner{R_s\bd_s}{R_t\bh(\bgam_{t,s}(-\beta))}} \\
&= \mu_t(\bgam_{t,s}(-\beta))\,\e^{\i\inner{R_t\bd_t-R_s\bd_s}{R_t\bh(\bgam_{t,s}(-\beta))}},
\end{align*}
which implies \eqref{eq:common-circle:translation} provided that $\mu_s(\bgam_{s,t}(\beta))$ (or, equivalently, $\mu_t(\bgam_{t,s}(-\beta))$) does not vanish.

\item
In the same way, the elliptic arcs $\zb{\gamma}^*_{s,t}$ and $\zb{\gamma}^*_{t,s}$ 
have the symmetry \eqref{eq:commonCircleIdentityDual}. With the symmetry property \eqref{eq:symmetryNu} of $\ktran[f]$, we get for all $\beta\in J^*_{s,t}$ that
\begin{align*}
\mu_s(\zb{\gamma}^*_{s,t}(\beta)) 
&= \ktran [f]\left(R_s\bh(\zb{\gamma}^*_{s,t}(\beta)) \right)
\e^{-\i\inner{R_s\bd_s}{R_s\bh(\zb{\gamma}^*_{s,t}(\beta))}} \\
&= \ktran{[f]}\left(-R_t\bh(\zb{\gamma}^*_{t,s}(\beta)) \right)\e^{\i\inner{R_s\bd_s}{R_t\bh(\zb{\gamma}^*_{t,s}(\beta))}} \\
&= \overline{\mu_t(\zb{\gamma}^*_{t,s}(\beta))}\,
\e^{\i\inner{R_s\bd_s-R_t\bd_t}{R_t\bh(\zb{\gamma}^*_{t,s}(\beta))}},
\end{align*}
which implies \eqref{eq:common-circle:translationDual} provided that $\mu_s(\zb{\gamma}^*_{s,t}(\beta))$ 
(or, equivalently, $\mu_t(\zb{\gamma}^*_{t,s}(\beta))$) does not vanish.
\end{enumerate}
\end{proof}

\begin{proof}[of \autoref{th:commonAxis:translation}]
\begin{enumerate}
\item
If $ R_t\be^3=R_s\be^3$, we get for all $\bk\in\mathcal B_{k_0}^2$, using that $R_s\bh(\bk)=R_t\bh(\mathrm{Q}(-\alpha)\bk)$ according to \eqref{eq:commonAxisRel},
\begin{align*}
\mu_t(\mathrm{Q}(-\alpha)\bk) &= \ktran[f](R_t\bh(\mathrm{Q}(-\alpha)\bk))\e^{-\i\inner{R_t\bd_t}{R_t\bh(\mathrm{Q}(-\alpha)\bk)}} \\
&= \ktran[f](R_s\bh(\bk))\e^{-\i\inner{R_t\bd_t}{R_s\bh(\bk)}} = \mu_s(\bk)\e^{-\i\inner{R_t\bd_t-R_s\bd_s}{R_s\bh(\bk)}}.
\end{align*}

\item
Similarly, we get for all $\bk\in\mathcal B_{k_0}^2$ in the case $ R_t\be^3=-R_s\be^3$ with the corresponding relation 
$-R_s\bh(\bk)=R_t\bh(\mathrm{Q}(-\alpha)\mathrm{S}\bk)$ from \eqref{eq:commonAxisReflectionRel} that
\begin{align*}
\mu_t(\mathrm{Q}(-\alpha)\mathrm{S}\bk) 
&= \ktran[f](R_t\bh(\mathrm{Q}(-\alpha)\mathrm{S}\bk))\e^{-\i\inner{R_t\bd_t}{R_t\bh(\mathrm{Q}(-\alpha)S\bk)}} \\
&= \ktran[f](-R_s\bh(\bk))\e^{\i\inner{R_t\bd_t}{R_s\bh(\bk)}} = \overline{\mu_s(\bk)}\e^{\i\inner{R_t\bd_t-R_s\bd_s}{R_s\bh(\bk)}}.
\end{align*}
\end{enumerate}
\end{proof}

\begin{proof}[of \autoref{th:reconTransl}]
\begin{enumerate}
\item
Since we know from our assumption of the scattering potential $f$ being real-valued that 
\[ \abs{\mu_s(\zb0)} = \abs{\ktran[f](\zb0)} = (2\pi)^{-\frac32}\int_{\R^3}f(\bx)\dd\bx > 0, \]
the equations \eqref{eq:common-circle:translation} and \eqref{eq:common-circle:translationDual} hold for all $\beta$ in some open interval $J$ around $0$. By taking the logarithm of these equations, we find with the circular arcs $\bsigma_{s,t}$ and $\zb{\sigma}^*_{s,t}$, defined in \eqref{eq:YsYt_param} and \eqref{eq:dualCircle}, that
\begin{align}
\inner{R_t\bd_t-R_s\bd_s}{\bsigma_{s,t}(\beta)} &= M(\beta)\quad\text{and} \label{eq:common-circle:translationLog}\\
\inner{R_t\bd_t-R_s\bd_s}{\zb{\sigma}^*_{s,t}(\beta)} &=  M^*(\beta) \label{eq:common-circleDual:translationLog}
\end{align}
for all $\beta\in J$, where the functions $M\colon J\to\C$ and $M^*\colon J\to\C$, given by
\begin{align*}
M(\beta) &\coloneqq -\i\int_0^\beta\frac{F'(\tilde\beta)}{F(\tilde\beta)}\dd\tilde\beta,\quad F(\beta)\coloneqq \frac{\mu_s(\bgam_{s,t}(\beta))}{\mu_t(\bgam_{t,s}(-\beta))},\quad\text{and} \\
 M^*(\beta) &\coloneqq -\i\int_0^\beta\frac{(F^*)'(\tilde\beta)}{F^*(\tilde\beta)}\dd\tilde\beta,
\quad
F^*(\beta)\coloneqq\frac{\mu_s(\zb{\gamma}^*_{s,t}(\beta))}{\overline{\mu_t(\zb{\gamma}^*_{t,s}(\beta))}},
\end{align*}
are explicitly known. Here, we used that the left-hand sides of \eqref{eq:common-circle:translationLog} and \eqref{eq:common-circleDual:translationLog} vanish for $\beta=0$ to choose the correct branch of the logarithm of the continuously differentiable and nowhere vanishing functions $F$ and $F^*$.

Inserting the expressions \eqref{eq:YsYt_param} and \eqref{eq:dualCircle} for the circular arcs $\bsigma_{s,t}$ and $\zb{\sigma}^*_{s,t}$, respectively, we find, using the notation from \autoref{th:YsYt_param} and \autoref{th:dualCircle}, that we have for all $\beta\in J$ the equation system
\begin{align*}
a_{s,t}(\cos(\beta)-1)\inner{R_s^\top R_t\bd_t-\bd_s}{R_s^\top\bv_{s,t}^1}+a_{s,t}\sin(\beta)\inner{R_s^\top R_t\bd_t-\bd_s}{R_s^\top\bv_{s,t}^2} &= M(\beta), \\
a^*_{s,t}(\cos(\beta)-1)\inner{R_s^\top R_t\bd_t-\bd_s}{R_s^\top\bv_{s,t}^3}+ a^*_{s,t}\sin(\beta)\inner{R_s^\top R_t\bd_t-\bd_s}{R_s^\top\bv_{s,t}^2} &=  M^*(\beta).
\end{align*}
Since the functions $\beta\mapsto\cos(\beta)-1$ and $\beta\mapsto\sin(\beta)$ are linearly independent on every interval with positive length, this implies that the coefficients
\[ \inner{R_s^\top R_t\bd_t-\bd_s}{R_s^\top\bv_{s,t}^j},\quad j\in\{1,2,3\}, \]
are uniquely determined by this (recalling that we explicitly know the parameters $a_{s,t}\neq0$ and $a^*_{s,t}\neq0$). 
Since $(R_s^\top\bv_{s,t}^j)_{j=1}^3$ is an orthonormal basis of $\R^3$ (which we also know explicitly), 
this uniquely determines the vector $R_s^\top R_t\bd_t-\bd_s\in\R^3$.

\item
Since $\mu_s(\zb0)\ne0$, we find an open disk $A\subset\mathcal B_{k_0}^2$ that contains $\zb0$ such that we have $\mu_s(\bk)\ne0$ for all $\bk\in A$. If $R_s^\top R_t\be^3=\be^3$, we have that $R_s^\top R_t=Q^{(3)}(\alpha)$ for some $\alpha\in\R/(2\pi\Z)$ and \eqref{eq:translDegnPlus} implies
\[ \inner{R_s^\top R_t\bd_t-\bd_s}{\bh(\bk)} 
= 
-\i\int_{C_{\zb0,\bk}}\frac{\nabla G(\zb{\tilde k})}{G(\zb{\tilde k})}\dd\zb{\tilde k}\quad\text{with}
\quad 
G(\bk)\coloneqq \frac{\mu_s(\bk)}{\mu_t(\mathrm{Q}(-\alpha)\bk)} 
\]
for all $\bk\in A$, where $C_{\zb0,\bk}$ denotes an arbitrary curve from $\zb0$ to $\bk$ in $A$. Since the vectors $\bh(\bk)$ cover for $\bk\in A$ an open subset of the hemisphere $\mathcal H_0$, they span all of $\R^3$, and thus this equation uniquely determines the vector $R_s^\top R_t\bd_t-\bd_s\in\R^3$.

Similarly, we have for $R_s^\top R_t\be^3=-\be^3$ that $R_s^\top R_t=Q^{(2)}(\pi)Q^{(3)}(\alpha)$ for some $\alpha\in\R/(2\pi\Z)$ and according to \eqref{eq:translDegnMinus}
\[ \inner{R_s^\top R_t\bd_t-\bd_s}{\bh(\bk)} = -\i\int_{C_{\zb0,\bk}}\frac{\nabla G^*(\zb{\tilde k})}{G^*(\zb{\tilde k})}\dd\zb{\tilde k}
\quad\text{with}\quad
G^*(\bk)\coloneqq \frac{\mu_s(\bk)}{\overline{\mu_t(\mathrm{Q}(-\alpha)S\bk)}} \]
for all $\bk\in A$, which again uniquely determines $R_s^\top R_t\bd_t-\bd_s$.
\end{enumerate}
\end{proof}


\section{Parameterization via Stereographic Projection} \label{se:stereo}

Based on the stereographic projection, 
we describe in this section a transformation that turns the elliptic arcs $\bgam$, see \eqref{eq:gamma}, into straight lines in $\R^2$.
Applying this transformation to the data $\nu_t$, see \eqref{eq:nu},
then one needs to detect common lines in the two-dimensional plane in order to reconstruct the rotation parameters.
There are existing algorithms for detecting common lines in the context of motion detection the ray transform, cf.\ \cite{Hee87}.
However, these lines all contain the origin, which is not the case for the diffraction tomography we consider here where we need an additional parameter to describe the lines.


\subsection{Direct common circle method}

We consider the stereographic projection $\bpi_t\colon \partial\mathcal B_{k_0}^3(-k_0R_t\be^3)\setminus\{\zb0\}\to\mathcal P_t$ of the hemisphere $\mathcal H_t\setminus\{\zb0\}$ of the sphere $\partial\mathcal B_{k_0}^3(-k_0R_t\be^3)$ from the origin onto the equatorial plane
\[ \mathcal P_t\coloneqq\{\bx\in\R^3:\langle\bx,R_t\be^3\rangle = -k_0\}. \]
This maps every circle $\partial\mathcal B_{k_0}^3(-k_0R_s\be^3)\cap \partial\mathcal B_{k_0}^3(-k_0R_t\be^3)$ (as it passes through the origin, which could be defined to be mapped to infinity) to a straight line in $\mathcal P_t$.

The stereographic projection $\bpi_t$ of a point $\bx \in \partial\mathcal B_{k_0}^3(-k_0R_t\be^3)\setminus \{\zb0\}$ is hereby defined as the intersection of the line through $\zb0$ and $\bx$ with the plane $\mathcal P_t$.
In particular, we have for $t=0$ where the rotation is $R_0=I$ that 
$\bpi_0\colon\partial\mathcal B_{k_0}^3(-k_0\be^3)\setminus\{\zb0\}\to\mathcal P_0$
\begin{equation} \label{eq:pi0_def}
 \bpi_0(\bx)=\left(-k_0\frac{x_1}{x_3},-k_0\frac{x_2}{x_3},-k_0\right)^\top=-k_0\frac{\bx}{x_3}.
\end{equation}
The stereographic projection $\bpi_t$ for general $t\in[0,T]$ is then obtained 
by rotating a point $\bx\in \partial\mathcal B_{k_0}^3(-k_0R_t\be^3)$ 
first to $\partial\mathcal B_{k_0}^3(-k_0\be^3)$ and rotating the projected point in $\mathcal P_0$ back to 
$\mathcal P_t$, i.e.,
\[ \bpi_t(\bx) \coloneqq R_t\bpi_0(R_t^\top\bx). \]
The following lemma shows that we can write all the projections $\bpi_t$ as restrictions of the function
$\bpi\colon\R^3\setminus\{\zb0\}\to\R^3\setminus\{\zb0\}$ defined by
\begin{equation} \label{eq:pi}
 \bpi(\bx)\coloneqq2k_0^2\frac{\bx}{\norm{\bx}^2},
 \end{equation}
whose inverse is given by $\bpi^{-1}=\bpi$.

\begin{lemma}\label{th:stereo}
  For every $t\in[0,T]$, we have
  \[ \bpi_t(\bx) = \bpi(\bx)\quad\text{for all}\quad \bx \in \partial\mathcal B_{k_0}^3(-k_0R_t\be^3)\setminus\{\zb0\}. \]
\end{lemma}

\begin{proof}
  Let $\bx\in\partial\mathcal B_{k_0}^3(-k_0R_t\be^3)\setminus\{\zb0\}$. Then we obtain
  \[ 0 = \norm{\bx+k_0R_t\be^3}^2-k_0^2 = \norm{\bx}^2+2k_0\langle R_t^\top\bx,\be^3\rangle \]
  and therefore
  \[ \bpi_t(\bx) = R_t\bpi_0(R_t^\top\bx) = -k_0R_t\frac{R_t^\top\bx}{\langle R_t^\top\bx,\be^3\rangle} = 2k_0^2\frac{\bx}{\norm{\bx}^2} = \bpi(\bx). \]
\end{proof}

Next, we consider for arbitrary $t\in[0,T]$ the function
$\tau\colon\mathcal B_{k_0}^2\setminus\{\zb0\}\to\R^2\setminus\overline{\mathcal B_{k_0}^2}$ defined by
\begin{equation}\label{eq:tau}
 \tau(\bk)\coloneqq P(R_t^\top\bpi_t(R_t\bh(\bk))) = P(\bpi_0(\bh(\bk))) = \frac{k_0}{k_0-\kappa(\bk)}\bk,
 \end{equation}
which describes the change from the parameterization via $R_t\bh$ to the one via stereographic projection and is conveniently independent of the choice of $t\in[0,T]$. 
It maps by definition the data point $\bk$ by the parameterization $R_t\bh$ onto the hemisphere $\mathcal H_t$, 
stereographically projects it to $\mathcal P_t$ (with image $\mathcal P_t\setminus\overline{\mathcal B_{k_0}^3(-k_0R_t\be^3)}$), 
and extracts the two components in the plane by rotating it to $\mathcal P_0$ and orthogonally projecting it with $P$ 
to the first two components. Therefore it maps every elliptic arc $\bgam_{t,s}$ to a straight line. The codomain of $\tau$ is  chosen so that $\tau$ is bijective, and its inverse is given by
\[ \tau^{-1}(\by) = \frac{2k_0^2}{k_0^2+\norm\by^2}\by,\quad\by\in\R^2\setminus\overline{\mathcal B_{k_0}^2}. \]

\begin{lemma}\label{th:tau_gamma}
  Let $s,t\in[0,T]$ such that $R_s\be^3\neq \pm R_t\be^3$ and $(\varphi,\theta,\psi)\in(\R/(2\pi\Z))\times[0,\pi]\times(\R/(2\pi\Z))$ be the Euler angles of the rotation $R_s^\top R_t$ as in \eqref{eq:euler}.
  \begin{enumerate}
    \item
    The elliptic arc $\bgam_{s,t}$ defined in \eqref{eq:gamma} fulfills
    \begin{equation}\label{eq:tau_gamma}
      \tau(\bgam_{s,t}(\beta)) = -k_0\tan(\tfrac\theta2)\begin{pmatrix}\cos(\varphi)\\\sin(\varphi)\end{pmatrix}+\frac{k_0}{\cos(\frac\theta2)}\cot(\tfrac\beta2)\begin{pmatrix}-\sin(\varphi)\\\cos(\varphi)\end{pmatrix},\quad\beta\in J_{s,t}\setminus\{0\}.
    \end{equation}
    \item
    The dual elliptic arc $\zb{\gamma}^*_{s,t}$ given by \eqref{eq:dualEllipse} fulfills
    \begin{equation}\label{eq:tau_gamma_dual}
      \tau(\zb{\gamma}^*_{s,t}(\beta)) 
			= 
			k_0\cot(\tfrac\theta2)\begin{pmatrix}\cos(\varphi)\\\sin(\varphi)\end{pmatrix} -\frac{k_0}{\sin(\frac\theta2)}\cot(\tfrac\beta2)
			\begin{pmatrix}-\sin(\varphi)\\
			\cos(\varphi)\end{pmatrix},\quad\beta\in J^*_{s,t}\setminus\{0\}.
    \end{equation}
  \end{enumerate}
\end{lemma}
\goodbreak
\begin{proof}
  \begin{enumerate}
    \item
    We use \eqref{eq:sigma_gamma} to write
    \[ \tau(\bgam_{s,t}(\beta)) = P(\bpi_0(\bh(\bgam_{s,t}(\beta)))) = P(\bpi_0(R_s^\top\bsigma_{s,t}(\beta))). \]
    Plugging in the expression \eqref{eq:YsYt_param} for the circular arc $\bsigma_{s,t}$ and the definition \eqref{eq:pi0_def} of the function $\bpi_0$, we arrive at
    \[ \tau(\bgam_{s,t}(\beta)) = k_0\frac{ a_{s,t}(\cos(\beta)-1)P(R_s^\top\bv_{s,t}^1)+ a_{s,t}\sin(\beta)P(R_s^\top\bv_{s,t}^2)}{\frac{k_0}2(1-\cos(\beta))(1+\langle R_s\be^3,R_t\be^3\rangle)}. \]
    As in \autoref{th:gamma}, where we already calculated the projections of the basis vectors $\bv_{s,t}^1$ and~$\bv_{s,t}^2$ in \eqref{eq:Pv}, we can rewrite this  in the form
    \[ \tau(\bgam_{s,t}(\beta)) = 2\frac{\tilde a_{s,t}(\cos(\beta)-1)\bw_{s,t}^1+ a_{s,t}\sin(\beta)\bw_{s,t}^2}{(1-\cos(\beta))(1+\langle R_s\be^3,R_t\be^3\rangle)}. \]
    Using $1+\langle R_s\be^3,R_t\be^3\rangle = \frac12\norm{R_s\be^3+R_t\be^3}^2 = \frac2{k_0^2} a_{s,t}^2$
    and the trigonometric identity $\frac{\sin(\beta)}{1-\cos(\beta)}=\cot(\frac\beta2)$,
    this becomes 
    \[ \tau(\bgam_{s,t}(\beta)) = -\frac{k_0^2\tilde a_{s,t}}{ a_{s,t}^2}\bw_{s,t}^1+\frac{k_0^2}{ a_{s,t}}\cot(\tfrac\beta2)\bw_{s,t}^2. \]
    Inserting the expressions for the parameters in Euler angles as in \autoref{th:Euler}, we obtain \eqref{eq:tau_gamma}.
    \item
    In the same way, we find with the results and the notation of \autoref{th:dualCircle} that
    \[ \tau(\zb{\gamma}^*_{s,t}(\beta)) 
		= 
		P(\bpi_0(R_s^\top\zb{\sigma}^*_{s,t}(\beta))) 
		= -k_0\frac{a^*_{s,t}(\cos(\beta)-1)P(R_s^\top\bv_{s,t}^3)- a^*_{s,t}\sin(\beta)P(R_s^\top\bv_{s,t}^2)}{\frac{k_0}2(\cos(\beta)-1)(1-\langle R_s\be^3,R_t\be^3\rangle)}. \]
    Using \eqref{eq:Pv} and \eqref{eq:Pv3} to express the projections of the basis vectors $\bv_{s,t}^2$ and $\bv_{s,t}^3$, we get with $1-\langle R_s\be^3,R_t\be^3\rangle = \frac12\norm{R_s\be^3-R_t\be^3}^2 = \frac2{k_0^2} (a^*_{s,t})^2$ that
    \[ \tau(\zb{\gamma}^*_{s,t}(\beta)) 
		= 
		k_0^2\frac{\tilde a_{s,t}}{(a^*_{s,t})^2}\bw_{s,t}^1 -\frac{k_0^2}{a^*_{s,t}}\cot(\tfrac\beta2)\bw_{s,t}^2. \]
    Inserting the expressions for the parameters in Euler angles as in \autoref{th:Euler} and \autoref{th:dualCircle}, this becomes \eqref{eq:tau_gamma_dual}.
  \end{enumerate}
\end{proof}

From the definitions of the intervals $J_{s,t}$ and $J^*_{s,t}$,
the functions $\tau\circ\bgam_{s,t}\colon J_{s,t}\setminus\{0\}\to\R^2\setminus\overline{\mathcal B_{k_0}^2}$ from \eqref{eq:tau_gamma} 
and 
$\tau\circ\zb{\gamma}^*_{s,t}\colon J^*_{s,t}\setminus\{0\}\to\R^2\setminus\overline{\mathcal B_{k_0}^2}$ from \eqref{eq:tau_gamma_dual} 
 parameterize the parts of straight lines in $\R^2$ which are outside the ball $\overline{\mathcal B_{k_0}^2}$, see \autoref{fig:stereographic}.

\begin{figure}[htb]
  \centering\includegraphics[width=0.9\textwidth]{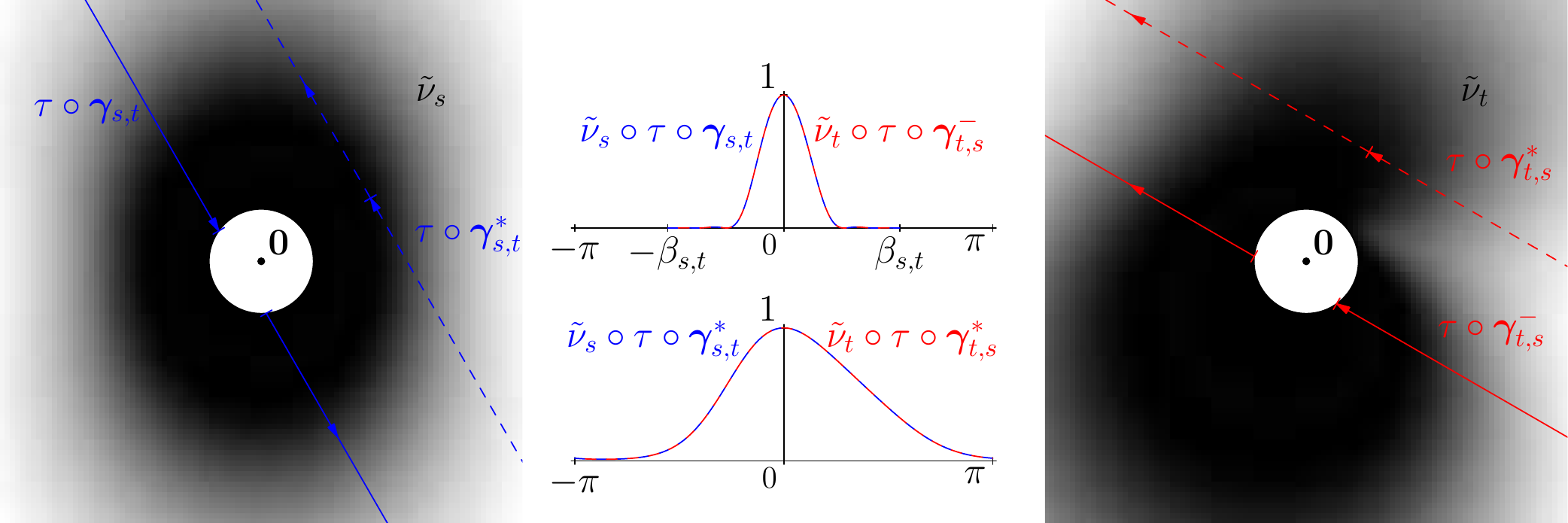}
  \caption{Transformed scaled squared energies $\tilde\nu_s$ and $\tilde\nu_t$, see \eqref{eq:hatNu}, for the same situation as in \autoref{fig:ellipticArcs} and the paths of the corresponding two straight lines $\tau\circ\bgam_{s,t}$ and $\tau\circ\bgam_{t,s}^-$ (the minus again indicating the reversed direction) and their dual straight lines $\tau\circ\zb{\gamma}^*_{s,t}$ and $\tau\circ\zb{\gamma}^*_{t,s}$. The values of $\tilde\nu_s$ and $\tilde\nu_t$ along the lines are plotted in the center of the figure.}\label{fig:stereographic}
\end{figure}

Thus, looking for straight lines in the \emph{transformed scaled squared energy} $\tilde\nu_t\colon\R^2\setminus\overline{\mathcal B_{k_0}^2}\to[0,\infty)$, defined by
\begin{equation}\label{eq:hatNu}
  \tilde\nu_t\coloneqq\nu_t\circ\tau^{-1}\quad\text{for all}\quad t\in[0,T],
\end{equation}
we can recover the Euler angles as in \autoref{th:commonCircle}. This is summarized in the following theorem.

\begin{theorem}\label{th:reconStereo}
  Let $s,t\in[0,T]$  such that $R_s\be^3\ne\pm R_t\be^3$ and assume that there uniquely exist two pairs $(\zb{\Gamma}_\ell)_{\ell=1}^2$ and $(\zb{\Gamma}^*_\ell)_{\ell=1}^2$ of straight lines of the form
  \begin{align}
    \zb{\Gamma}_\ell\colon\R\to\R^2,\quad&\zb{\Gamma}_\ell(\xi)\coloneqq -b\bw_\ell^1+\xi\bw_\ell^2,\quad\text{and} \label{eq:Gamma} \\
    \zb{\Gamma}^*_\ell\colon\R\to\R^2,\quad&\zb{\Gamma^*}_\ell(\xi)\coloneqq \frac{k_0^2}b\bw_\ell^1-\xi\bw_\ell^2,\quad\ell\in\{1,2\}, \label{eq:GammaDual}
  \end{align}
  for some parameter $b\in(0,\infty)$ and two positively oriented, orthonormal bases $(\bw_1^j)_{j=1}^2$ and $(\bw_2^j)_{j=1}^2$ of $\R^2$ such that we have for the transformed scaled squared energy $\tilde\nu_t$ that
  \begin{alignat*}{2}
    &\tilde\nu_s(\zb{\Gamma}_1(\xi)) = \tilde\nu_t(\zb{\Gamma}_2(-\xi))\quad&&\text{for all}\quad\xi\in\R\quad\text{with}\quad\xi^2>k_0^2-b^2, \\
    &\tilde\nu_s(\zb{\Gamma}^*_1(\xi)) = \tilde\nu_t(\zb{\Gamma}^*_2(\xi))\quad&&\text{for all}
		\quad\xi\in\R\quad
		\text{with}\quad\xi^2>k_0^2-k_0^4b^{-2}.
  \end{alignat*}
  Then the relative rotation is given by
  \begin{equation}\label{eq:reconStereoEuler}
    R_s^\top R_t = Q^{(3)}(\arg(\bw_1^1))Q^{(2)}(2\arctan(\tfrac b{k_0}))Q^{(3)}(\pi-\arg(\bw_2^1)).
  \end{equation}
\end{theorem}

\begin{proof}
  We parameterize $R_s^\top R_t$ in Euler angles $(\varphi,\theta,\psi)\in(\R/(2\pi\Z))\times[0,\pi]\times(\R/(2\pi\Z))$ as in \eqref{eq:euler} and get for the representation of $R_t^\top R_s$ in Euler angles the formula \eqref{eq:RtRsEuler}. Then, we consider the straight lines $\tau\circ\bgam_{s,t}$ and $\tau\circ\zb{\gamma}^*_{s,t}$, given by \eqref{eq:tau_gamma} and \eqref{eq:tau_gamma_dual}, where $\bgam_{s,t}$ and ${\zb\gamma}^*_{s,t}$ denote the elliptic arcs introduced in \eqref{eq:gamma} and \eqref{eq:dualEllipse}, and reparameterize them via the functions
  \begin{alignat*}{2}
    &\Xi\colon J_{s,t}\setminus\{0\}\to\{\xi\in\R:\xi^2>k_0^2-k_0^2\tan^2(\tfrac\theta2)\},\quad&&\Xi(\beta)\coloneqq\frac{k_0}{\cos(\frac\theta2)}\cot(\tfrac\beta2),\quad\text{and}\\
    &\Xi^*\colon J^*_{s,t}\setminus\{0\}\to\{\xi\in\R:\xi^2>k_0^2-k_0^2\cot^2(\tfrac\theta2)\},
		\quad&&\Xi^*(\beta)\coloneqq\frac{k_0}{\sin(\frac\theta2)}\cot(\tfrac\beta2),
  \end{alignat*}
  which are seen to be bijective by using that $\beta\in J_{s,t}$ is by definition \eqref{eq:YsYt_param:angle} for $\beta\in(-\pi,\pi]$ equivalent to $\cos(\beta)>\frac{\cos(\theta)-1}{\cos(\theta)+1}$, which is equivalent to $\cot^2(\frac\beta2)=\frac{1+\cos(\beta)}{1-\cos(\beta)}>\cos(\theta)$, and therefore to
  \[ (\Xi(\beta))^2+k_0^2\tan^2(\tfrac\theta2) > \frac{k_0^2}{\cos^2(\frac\theta2)}(\cos^2(\tfrac\theta2)-\sin^2(\tfrac\theta2))+k_0^2\tan^2(\tfrac\theta2) = k_0^2. \]
  Analogously, we find that $\beta\in J^*_{s,t}$ 
	is by definition \eqref{eq:dualCircleAngle} for $\beta\in(-\pi,\pi]$ equivalent to $\cos(\beta)>\frac{\cos(\theta)+1}{\cos(\theta)-1}$, which is equivalent to $\cot^2(\frac\beta2)=\frac{1+\cos(\beta)}{1-\cos(\beta)}>-\cos(\theta)$, and therefore to
  \[ (\Xi^*(\beta))^2+k_0^2\cot^2(\tfrac\theta2) > \frac{k_0^2}{\sin^2(\frac\theta2)}(\sin^2(\tfrac\theta2)-\cos^2(\tfrac\theta2))+k_0^2\cot^2(\tfrac\theta2) = k_0^2. \]
  
  Then, according to \eqref{eq:tau_gamma}, the curves $\tau\circ\bgam_{s,t}\circ\Xi^{-1}$ and $\tau\circ\bgam_{t,s}\circ\Xi^{-1}$ are with
  \begin{equation}\label{eq:reconStereoParam}
    b = k_0\tan(\tfrac\theta2),\quad\bw_1^1=\bw_{s,t}^1=\begin{pmatrix}\cos(\varphi)\\\sin(\varphi)\end{pmatrix},\quad\text{and}\quad\bw_2^1=\bw_{t,s}^1=\begin{pmatrix}\cos(\pi-\psi)\\\sin(\pi-\psi)\end{pmatrix}
  \end{equation}
  on the set $X\coloneqq\{\xi\in\R:\xi^2>k_0^2-b^2\}$ of the form \eqref{eq:Gamma}; and according to \eqref{eq:tau_gamma_dual}, the dual curves $\tau\circ\zb{\gamma}^*_{s,t}\circ(\Xi^*)^{-1}$ and $\tau\circ\zb{\gamma}^*_{t,s}\circ (\Xi^*)^{-1}$ are with this on the set 
	$X^*\coloneqq\{\xi\in\R:\xi^2>k_0^2-k_0^4b^{-2}\}$ of the form~\eqref{eq:GammaDual}. Moreover, these curves fulfill according to \eqref{eq:commonCircleIdentity} and \eqref{eq:commonCircleIdentityDual} the relations
  \begin{align}
    \tilde\nu_s(\tau(\bgam_{s,t}(\Xi^{-1}(\xi)))) 
		&= \nu_s(\bgam_{s,t}(\Xi^{-1}(\xi))) 
		\\
		&= \nu_t(\bgam_{t,s}(-\Xi^{-1}(\xi))) = \tilde\nu_t(\tau(\bgam_{t,s}(\Xi^{-1}(-\xi)))),\quad \xi\in X,
		\\
    \tilde\nu_s(\tau(\zb{\gamma}^*_{s,t}((\Xi^*)^{-1}(\xi)))) 
		&= \nu_s(\zb{\gamma}^*_{s,t}((\Xi^*)^{-1}(\xi))) 
		\\
		&= \nu_t(\zb{\gamma}^*_{t,s}((\Xi^*)^{-1}(\xi))) 
		= \tilde\nu_t(\tau(\zb{\gamma}^*_{t,s}((\Xi^*)^{-1}(\xi)))),\quad \xi\in X^*.
  \end{align}  
  The uniqueness of the pairs $(\zb\Gamma_\ell)_{\ell=1}^2$ and $(\zb{\Gamma}^*_\ell)_{\ell=1}^2$ 
	therefore implies 
	$\zb\Gamma_1=\tau\circ\bgam_{s,t}\circ\Xi^{-1}$, $\zb\Gamma_2=\tau\circ\bgam_{t,s}\circ\Xi^{-1}$, $\zb{\Gamma}^*_1=\tau\circ\zb{\gamma}^*_{s,t}\circ(\Xi^*)^{-1}$, and $\zb{\Gamma}^*_2=\tau\circ\zb{\gamma}^*_{t,s}\circ(\Xi^*)^{-1}$, so that we can read off the Euler angles from the correspondencies \eqref{eq:reconStereoParam}, giving us the reconstruction \eqref{eq:reconStereoEuler}.
\end{proof}

\subsection{Infinitesimal common circle method}
\label{se:stereo_inf}

We can also formulate the infinitesimal common circle method from \autoref{sec:cont-cc} via common lines.
We show the following analogue of \autoref{th:infRel},
where
the coefficient of the spatial derivative becomes affine.

\begin{lemma}\label{th:infRelStereo}
  Let the rotations $R\in C^1([0,T]\to\SO)$ be continuously differentiable and the associated angular velocities $\zb\omega_t\in \R^3$ be written in cylindrical coordinates \eqref{eq:omega-decomp}.
  Then, the transformed scaled squared energy $\tilde\nu_t$, defined in \eqref{eq:hatNu}, satisfies for every $ r\in\R\setminus[-k_0,k_0]$ and $t\in[0,T]$ the relation
  \begin{equation}\label{eq:infRelStereo}
    \partial_t\tilde\nu_t( r\zb\phi_t)=(k_0\rho_t+ r\zeta_t)\left<\nabla\tilde\nu_t( r\zb\phi_t),\begin{psmallmatrix}-\phi_{t,2}\\\phi_{t,1}\end{psmallmatrix}\right>.
  \end{equation}
\end{lemma}

\begin{proof}
  The transformed data $\tilde\nu_t$ is given by
  \[ \tilde\nu_t(\by) = \nu_t(\tau^{-1}(\by)) = \abs{\ktran[f](R_t\bh(\tau^{-1}(\by)))}^2. \]
  Using now that we have by definition \eqref{eq:pi0_def} of $\bpi_0$ and definition \eqref{eq:tau} of $\tau$ that
  \[ \begin{pmatrix}\tau(\tau^{-1}(\by))\\-k_0\end{pmatrix} = \bpi_0(\bh(\tau^{-1}(\by))) = R_t^\top\bpi_t(R_t\bh(\tau^{-1}(\by))), \]
  we have with \autoref{th:stereo} and $\bpi^{-1}=\bpi$ the relation
  \begin{equation}\label{eq:hatH}
    \tilde\nu_t(\by)
    = \abs{\ktran[f](R_t\bh(\tau^{-1}(\by)))}^2
    = \abs{\ktran[f](\bpi((t,\by)))}^2
  \end{equation}
  with the function $K\colon[0,T]\times\R^2\setminus\overline{\mathcal B_{k_0}^2}\to\R^3$ defined by $K(t,\by)\coloneqq R_t\begin{psmallmatrix}\by\\-k_0\end{psmallmatrix}$.
  
  Since the partial derivative $\partial_tK$ of $K$ with respect to $t$ fulfills
  \[ \partial_t K(t, r\zb\phi_t) = R_t\left(\omega_t\times\begin{psmallmatrix} r\zb\phi_t\\-k_0\end{psmallmatrix}\right) = R_t\left(\begin{psmallmatrix}\rho_t\zb\phi_t\\\zeta_t\end{psmallmatrix}\times\begin{psmallmatrix} r\zb\phi_t\\-k_0\end{psmallmatrix}\right) = (k_0\rho_t+ r\zeta_t)R_t\begin{psmallmatrix}-\phi_{t,2}\\\phi_{t,1}\\0\end{psmallmatrix} \]
  and the Jacobi matrix $DK$ of $K$ with respect to $\by$ satisfies
  \[ DK(t, r\zb\phi_t)\begin{psmallmatrix}-\phi_{t,2}\\\phi_{t,1}\end{psmallmatrix} = R_t\begin{psmallmatrix}1&0\\0&1\\0&0\end{psmallmatrix}\begin{psmallmatrix}-\phi_{t,2}\\\phi_{t,1}\end{psmallmatrix} = R_t \begin{psmallmatrix}-\phi_{t,2}\\\phi_{t,1}\\0\end{psmallmatrix}, \]
  we have
  \[ \partial_tK(t, r\zb\phi_t) = (k_0\rho_t+ r\zeta_t)DK(t, r\zb\phi_t)\begin{psmallmatrix}-\phi_{t,2}\\\phi_{t,1}\end{psmallmatrix}, \]
  which implies with \eqref{eq:hatH} directly \eqref{eq:infRelStereo}.
\end{proof}

We obtain the following analogue to \autoref{th:infRecon} for reconstructing the angular velocity $\zb\omega_t$, from which we can determine the rotation matrices $R_t$ by \autoref{th:infReconR}.

\begin{theorem}\label{th:infReconStereo}
  Let the rotations $R\in C^1([0,T]\to\SO)$ be continuously differentiable and $t\in[0,T]$.
  Let further ${\zb\phi}\in\S^1_+$ be a unique direction
  with the property that there exist parameters $\rho,\zeta\in\R$ with
  \[ \partial_t\tilde\nu_t( r{\zb\phi})
  =
  (k_0\rho+ r\zeta)\left<\nabla\tilde\nu_t( r{\zb\phi}),
  \begin{psmallmatrix}-\phi_{2}\\\phi_{1}
  \end{psmallmatrix}\right>\quad\text{for all}\quad r\in\R\setminus[-k_0,k_0]
  \]
  for the transformed scaled squared energy $\tilde\nu_t$ in \eqref{eq:hatNu}.
  Provided that the set
  \begin{equation*}
    {\mathcal M}_t\coloneqq\left\{ r\in\R\setminus[-k_0,k_0]:\left<\nabla\tilde\nu_t( r{\zb\phi}),\begin{psmallmatrix}-\phi_{2}\\\phi_{1}\end{psmallmatrix}\right>\ne0\right\}
  \end{equation*}
  contains at least two elements, then the angular velocity is given by $\zb\omega_t = (\rho{\zb\phi},\zeta)^\top$.
\end{theorem}
\begin{proof}
  From \autoref{th:infRelStereo}, we find that the uniqueness implies that $\zb\phi_t={\zb\phi}$ and therefore also
  \begin{equation*}
    k_0\rho+ r\zeta = k_0\rho_t+ r\zeta_t\quad \text{for all}\quad r\in{\mathcal M}_t.
  \end{equation*}
  Hence we have $\rho=\rho_t$ and $\zeta=\zeta_t$ if the equation is satisfied for two different values $ r$.
\end{proof}

\end{document}